\documentclass[11pt,leqno]{article}

\usepackage{amsmath,amsthm, amsfonts,amssymb, epsfig}
\usepackage[usenames]{color}
\usepackage[section]{placeins}
\usepackage{enumitem}
\usepackage{mathscinet}

\usepackage[all]{xy}
\usepackage{psfrag}

\usepackage{hyperref}

\newtheorem{theorem}{Theorem}[section]
\newtheorem{corollary}[theorem]{Corollary}
\newtheorem{lemma}[theorem]{Lemma}
\newtheorem{proposition}[theorem]{Proposition}
\newtheorem{hproposition}[theorem]{Heuristic Proposition}
\newtheorem{quasi-theorem}[theorem]{Quasi-Theorem}
\newtheorem{conjecture}[theorem]{Conjecture}

\theoremstyle{definition}
\newtheorem{definition}[theorem]{Definition}

\newtheorem{question}[theorem]{Question}
\newtheorem{remark}[theorem]{Remark}
\newtheorem{caveat}[theorem]{Caveat}
\newtheorem{example}[theorem]{Example}

\newtheorem{outline}[theorem]{Outline}
\newtheorem{construction}[theorem]{Construction}

\def\cA{\mathcal A}\def\cB{\mathcal B}\def\cC{\mathcal C}\def\cD{\mathcal D}
\def\cE{\mathcal E}\def\cF{\mathcal F}\def\cH{\mathcal H}

\def\cM{\mathcal M}\def\cO{\mathcal O}\def\cP{\mathcal P}
\def\cT{\mathcal T}
\def\cU{\mathcal U}\def\cW{\mathcal W}\def\cX{\mathcal X}

\def\fC{\mathfrak C}

\newcommand{\aff}{\mathbb{A}}

\newcommand{\B}{\cB}
\newcommand{\Bu}{\cB^{\phi}}
\newcommand{\Bc}{\Cone_{\omega}}
\newcommand{\Bpre}{\Bu_{\mathrm{pre}}}
\newcommand{\bnr}{\mathrm{bnr}}
\newcommand{\chamber}{\overline{\fC}}
\newcommand{\cartan}{\mathfrak t}
\newcommand{\cc}{{\mathbb C}}
\newcommand{\cHlin}{\cH_{\mathrm{lin}}}
\newcommand{\cHsph}{\cH_{\mathrm{sph}}}

\newcommand{\euc}{\mathbb{E}}

\newcommand{\g}{\mathfrak g}
\newcommand{\hu}{h^{\phi}}
\newcommand{\hw}{h^{\omega}}
\newcommand{\hpre}{h^{\mathrm{pre}}}
\newcommand{\hull}{\mathrm{hull}}

\newcommand{\Met}{\mathrm{Met}}

\newcommand{\nuu}{\overrightarrow{\nu}}

\newcommand{\phii}{\tilde{\phi}}
\newcommand{\pp}{{\mathbb P}}
\newcommand{\rr}{\mathbb{R}}
\newcommand{\sll}{\mathfrak{sl}}
\newcommand{\trans}{T}
\newcommand{\Vectc}{\mathrm{Vect}_{\mathbb{C}}}
\newcommand{\vecd}{\overrightarrow{d}}
\newcommand{\waff}{W_{\mathrm{aff}}}
\newcommand{\wsph}{W_{\mathrm{sph}}}
\newcommand{\wlin}{W_{\mathrm{lin}}}

\DeclareMathOperator{\Aff}{Aff}

\DeclareMathOperator{\Char}{char}
\DeclareMathOperator{\Cone}{Cone}
\DeclareMathOperator{\id}{id}

\DeclareMathOperator{\GL}{GL}
\DeclareMathOperator{\Hom}{\mathrm{Hom}}
\DeclareMathOperator{\re}{Re}
\DeclareMathOperator{\Shend}{\underline{End}}
\DeclareMathOperator{\Rep}{Rep}
\DeclareMathOperator{\SL}{SL}
\DeclareMathOperator{\SU}{SU}
\DeclareMathOperator{\Sym}{Sym}
\DeclareMathOperator{\tot}{tot}

\begin{document}

\title{Harmonic Maps to Buildings and Singular Perturbation Theory}
\author{Ludmil Katzarkov, Alexander Noll, Pranav Pandit and Carlos Simpson}
\date{\today}
\maketitle

\begin{abstract}
The notion of a universal building associated with a point in the Hitchin base is introduced. This is a building equipped with a harmonic map from a Riemann surface that is initial among harmonic maps which induce the given cameral cover of the Riemann surface. In the rank one case, the universal building is the leaf space of the quadratic differential defining the point in the Hitchin base.

The main conjectures of this paper are: (1) the universal building always exists; (2) the harmonic map to the universal building controls the asymptotics of the Riemann-Hilbert correspondence and the non-abelian Hodge correspondence; (3) the singularities of the universal building give rise to Spectral Networks; and (4) the universal building encodes the data of a 3d Calabi-Yau category whose space of stability conditions has a connected component that contains the Hitchin base.

The main theorem establishes the existence of the universal building, conjecture (3), as well as the Riemann-Hilbert part of conjecture (2), in the case of the rank two example introduced in the seminal work of Berk-Nevins-Roberts on higher order Stokes phenomena. It is also shown that the asymptotics of the Riemann-Hilbert correspondence is always controlled by a harmonic map to a certain building, which is constructed as the asymptotic cone of a symmetric space.

\end{abstract}
\tableofcontents

\section{Introduction}

The theory of Nonabelian Hodge structures grew out of the work of many people, including Hitchin, Donaldson, Corlette, Deligne and Simpson. This theory, which connects  the moduli space of representations of the fundamental group of a smooth projective variety with the moduli space of Higgs bundles via a big twistor family, has led to many geometric applications. To list a few:

\begin{enumerate}

\item Proof of  the fact that $\SL (N,\mathbb{Z})$ cannot be the fundamental group of a smooth projective variety for $N > 2$ (see \cite{Carlos-Higgs-Local}).  

\item Proof of the Shafarevich Conjecture for smooth projective varieties with 
faithful linear representations of their fundamental groups \cite{Ludmil-Ramachandran, EKPR}.

\end{enumerate}

The latter result makes essential use of the Gromov-Schoen \cite{Gromov-Schoen}  theory of harmonic maps to buildings. The main ingredients in the proof are the construction of  a spectral covering  associated with a harmonic map to a  building, the use of  factorization theorems introduced first in \cite{Ludmil-factorization}, and an exhaustion function using the distance function on the building. These ingredients are all based on mimimizing the effect of the singularities of the harmonic map. 

The approach of this paper goes in the opposite direction, having as its motivating goal the  construction of a category out of the singularities of the harmonic map to the building. 
This approach is part of a bigger foundational program initiated by Kontsevich -- developing the theory of Stability Hodge Structrures. The hope is that
the moduli space of stability conditions 
can be included in a twistor type of family with a ``Dolbeault'' type
space as the zero fiber.  This zero fiber should be the
moduli space of complex structures.

Now we recall the foundations of Nonabelian Hodge theory on a Riemann surface. Let $X$ be a compact Riemann surface, $x_0\in X$. The moduli space of local systems on $X$ comes in various different incarnations: 

\begin{itemize}
\renewcommand{\labelitemi}{\textemdash}
\item the character variety, or \emph{Betti} moduli space 
$$
M_B:= Hom(\pi_1(X,x_0),SL_r)/SL_r;
$$

\item Hitchin's moduli space of Higgs bundles which we call the \emph{Dolbeault} moduli space
$$
M_{Dol}=\{ (\cE,\varphi )\} / \mbox{S-equiv};
$$
\item and the \emph{de Rham} moduli space of vector bundles with integrable algebraic connection 
$$
M_{DR}= \{ (\cE,\nabla )\} / \mbox{S-equiv}.
$$

\end{itemize}

The statement that these all parametrize local systems gives topological homeomorphisms
$$
M_B^{\rm top} \cong M_{Dol}^{\rm top} \cong M_{DR}^{\rm top},
$$
the equivalence between $M_B$ and $M_{DR}$ being furthermore complex analytic. These are not however isomorphisms
of algebraic varieties, and the algebraic moduli problems which they solve are very different. 

These moduli spaces are all noncompact, in fact $M_B$ is an affine variety. We may therefore think of choosing algebraic
compactifications, and it becomes an interesting question to understand the asymptotic nature of the
homeomorphisms as we approach the boundary. 

In the case of $M_{Dol}$ we have furthermore the \emph{Hitchin fibration}
$$
M_{Dol}\rightarrow {\mathbb A}^N
$$
which is a map of algebraic varieties, and we can choose a compatible compactification. Thanks to the $\cc^{\ast}$-action
$t:(\cE,\varphi )\mapsto (\cE,t\varphi )$, there is in fact a very natural compactification. Let
$M_{Dol}^{\ast}$ denote the complement of the nilpotent cone, that is to say the inverse image of ${\mathbb A}^N-\{ 0\}$.
Then there is an orbifold compactification $\overline{M}_{Dol}$ with
$$
\overline{M}_{Dol} = M_{Dol} \sqcup (M_{Dol}^{\ast}/\cc ^{\ast} ).
$$
More precisely, let $(M_{Dol}\times {\mathbb A}^1)^{\ast}$ be the complement of the nilpotent cone in $M_{Dol}\times \{ 0\}$,
then 
$$
\overline{M}_{Dol} = (M_{Dol}\times {\mathbb A}^1)^{\ast}/\cc ^{\ast}.
$$
This compactification was discussed by Hausel in \cite{Hausel-compactification}. It follows the general method of Bialynicki-Birula \cite{Birula-Swiecicka-quotients, Birula-Swiecicka-moment}. The Hitchin fibration extends to a map 
$$
\overline{M}_{Dol}\rightarrow \widetilde{{\mathbb P}}^N
$$
towards a weighted projective space compactifying the Hitchin base ${\mathbb A}^N$. It is weighted, because $\cc ^{\ast}$ acts on the
different terms in the characteristic polynomials of Higgs fields by different weights. 

There is a natural family of moduli spaces constituting a deformation relating $M_{Dol}$ and $M_{DR}$, and using this family
itself we can also construct a natural compactification of the family, following what was said above with \cite{Birula-Swiecicka-quotients, Birula-Swiecicka-moment}. We get 
$$
\overline{M}_{Hod} \rightarrow {\mathbb A}^1
$$
whose fiber over $0$ is $\overline{M}_{Dol}$ and whose fibers over $\lambda \neq 0$ are always $\overline{M}_{DR}$. 
The family of divisors at infinity is nicely behaved, for example in cases where the moduli space is smooth then these divisors
are smooth if we consider the compactification as an orbifold. (The orbifold points correspond to ``cyclic Higgs bundles''.)
Furthermore, the divisors at infinity are all the same, independently of $\lambda$: 
$$
\overline{M}_{DR} - M_{DR} = M_{Dol}^{\ast}/\cc ^{\ast}.
$$
In this way, even though $M_{DR}$ doesn't have a Hitchin fibration, its structure at infinity looks very similar to the structure at
infinity of $M_{Dol}$. Limiting points are identified with non-nilpotent Higgs bundles. 

The identification between limiting points at infinity of $M_{DR}$ and non-nilpotent Higgs bundles can be made very concrete in the
following way. One easy way to write down a family of connections going out to infinity is to choose a vector bundle $\cE$,
a Higgs field $\varphi$, and an initial connection $\nabla _0$ on $\cE$. Then we may consider the family of connections on $\cE$,
depending on a complex parameter $t\in \cc$, defined by
\begin{equation}
\label{mainnabla}
\nabla _t = \nabla _0  + t \varphi .
\end{equation}
If $(\cE,\varphi )$ is a semistable Higgs bundle and not in the nilpotent cone (i.e. the spectral curve of $\varphi$ is not concentrated at 
the zero-section), then this family of connections has for limit, as $t\rightarrow \infty$, the point in
the divisor at infinity 
$$
(\cE,\varphi )\in M_{Dol}^{\ast}/\cc ^{\ast}.
$$
In particular, the family does indeed go to infinity, i.e. its limit cannot be a point inside $M_{DR}$. 
Furthermore, we obtain a curve which is transverse to the divisor at infinity. This condition fixes, in some sense, the order of the
parameter $t$ in a canonical way. 

A more general family of connections might look like $(\cE_t, \nabla _t)$ defined for $t$ in a disc around $\infty$, and
it will have limiting point $(\cE,\varphi )$ if the vector bundles with $t^{-1}$-connection $(\cE_t,t^{-1}\nabla _t)$ converge
to $(\cE,\varphi )$ in the moduli space $M_{Hod}$ of vector bundles with $\lambda$-connection. In this case, the resulting map from 
the disc to $\overline{M}_{DR}$ will be transverse to the divisor at infinity, because we used the normalization $t^{-1}$. 

The family of connections considered by Gaiotto-Moore-Neitzke (\cite{GMN-WKB, GMN-Spectral-Networks}) fits into the more general situation of the previous paragraph.
Indeed, they start with a harmonic bundle $(\cE,\partial , \overline{\partial}, \varphi , \varphi ^{\dagger})$ and consider
the holomorphic bundle $\cE_t = (\cE,\overline{\partial} + t^{-1}\varphi ^{\dagger})$ which converges to $(\cE,\overline{\partial})$,
together with the connections $\nabla _t = \partial + t\varphi $. So, with a varying family of bundles, their
connections have for limit the Higgs bundle underlying the harmonic bundle $(\cE,\overline{\partial}, \varphi )$. 

In what follows, we will mostly be refering to the basic situation of a family of connections such as \eqref{mainnabla},
but everything should apply equally well in the general situation. For such a family of connections, we can formulate
the \emph{Riemann-Hilbert WKB problem}. Let $\rho _t : \pi _1(X,x_0)\rightarrow SL_r(\cc )$ be the monodromy 
representations\footnote{Here and throughout the paper we will be assuming that our connections have trivial determinant bundle and correspondingly
the Higgs fields have trace zero, so the structure group is $SL_r$.} of
$\nabla _t$. More generally, if $P,Q\in \widetilde{X}$ are two points on the universal cover joined by a unique homotopy class of
paths, let $T_{PQ}(t):\cE_P \rightarrow \cE_Q$ be the transport matrix for the connection $\nabla _t$. Then we would like to know:
\begin{itemize}
\renewcommand{\labelitemi}{\textemdash}
\item \emph{what is the asymptotic behavior of the matrices $\rho _t(\gamma )$ or $T_{PQ}(t)$ as a function of $t$?}
\end{itemize}

This is just a reformulation of the classical ``WKB problem'' in our present language. In terms of moduli spaces,
we are asking, for our algebraic curve $\cc \rightarrow M_{DR}$ in the de Rham moduli space, 
what is the asymptotic behavior of the corresponding 
holomorphic curve $\cc \rightarrow M_B$ in the Betti moduli space? This is a first approximation towards fully understanding the
Riemann-Hilbert transformation from $M_{DR}$ to $M_B$ in a neighborhood of the divisor at infinity. 

We could also look at the Hitchin moduli space $M_{Dol}$. Here, we have some very natural curves going out to infinity
given by the $\cc ^{\ast}$-action. Namely, consider the family of Higgs bundles $(\cE,t\varphi )$ for $t\in \cc$ going to infinity. 
Solving the Hitchin equations, we obtain a family of harmonic metrics $h_t$ on $\cE$, and corresponding flat connections $\nabla _t$.
We can again let $\rho _t: \pi _1(X,x_0)\rightarrow SL_r(\cc )$ be the monodromy 
representations and $T_{PQ}(t)$ be the transport matrices, and formulate the \emph{Hitchin WKB problem:}

\begin{itemize}
\renewcommand{\labelitemi}{\textemdash}
\item \emph{what is the asymptotic behavior of $\rho _t$ and $T_{PQ}(t)$ as $t\rightarrow \infty$ in Hitchin's situation?}
\end{itemize}

These two flavors of WKB are the subject matter of our paper. 

\subsection{Structure at infinity---some general discussion}

Before getting to the more specific content, let us provide some motivation. Consider the most basic example:
let $X={\mathbb P}^1$ with $4$ orbifold points,  and consider connections of rank $r=2$. Then,

\begin{itemize}
\renewcommand{\labelitemi}{\textemdash}

\item the character variety $M_B$ is the classic \emph{Fricke-Klein} cubic surface minus a triangle of lines, and

\item $M_{DR}$ is the \emph{space of initial conditions for Painlev\'e VI}. 

\end{itemize}

Arinkin and others have given several different
points of view on its structure, for example one can say that it is $\pp^1\times \pp^1$ blown up $8$ times at $4$ points on the
diagonal, minus some stuff. The same type of description holds for $M_{Dol}$. This is discussed in some detail in the paper \cite{LoraySaitoSimpson}. 

For $M_B$ we obtain a compactification $\overline{M}_B$ such that the divisor at infinity is a triangle of $\pp^1$'s. 
In particular, its \emph{incidence complex} is a real triangle, which
has the homotopy type of a circle. Notice that for $M_B$ the choice of compactification depended on the choice of some ``cluster-type'' coordinates
depending on picking three loops in $\pp ^1-4$ points. Different choices of loops will give different but birationally equivalent compactifications. 
Stepanov's theorem (see below) (\cite{Stepanov-dual-complex, Stepanov-resolution, Thuillier-toroidal}) says that the homotopy types of the incidence complexes are always the same. In this situation the invariance is easily
understood by hand. 

For $M_{Dol}$, the Hitchin fibration is particularly simple in this case: 
\[
\xymatrix{
J \ar[r] & M_{Dol} \ar[d]  \\
&  {\mathbb A}^1 \\
}
\]
with fiber an elliptic curve $J$, on which the monodromy acts by $-1$. 
The divisor at infinity in $\overline{M}_{Dol}$  is $J/\pm 1$
which is $\pp^1$ with four orbifold double points. These orbifold double points lead to the $-2$-curves in the compactification
described previously by blowing up $\pp ^1\times \pp ^1$ eight times. 

Let $\overline{N}_B$ and $\overline{N}_{Dol}$ denote small neighborhoods of the divisors at infinity, and let 
$N_B$ and $N_{Dol}$ denote their intersections with $M_B$ and $M_{Dol}$ respectively. These have well-defined homotopy types,
and the Hitchin correspondence yields a well-defined identification of homotopy types $N_B\sim N_{Dol}$.

We have conjectured that there is a relationship between the incidence complex of the divisor at infinity for $M_B$, and
the sphere at infinity in the Hitchin base for $M_{Dol}$. If 
$$
\overline{M}_{B} -M_{B} = \bigcup _{i} D_i,
$$
and define a simplicial complex with one $n$-simplex for each connected component
of $D_{i_0}\cap \cdots \cap D_{i_n}$. This is the \emph{incidence complex}. Stepanov and Thuillier show that
its homotopy type is independent of the choice of compactification, so we denote 
the realization by $|{\rm Step}(M_B)|$.
We have a map, well-defined up to homotopy, $N_B \rightarrow |{\rm Step}(M_B)|$. 

On the Hitchin side, the Hitchin fibration gives us a map to the sphere at infinity in the Hitchin base
$$
N_{Dol} \rightarrow S^{2N-1}.
$$

\begin{conjecture}
There is a homotopy-commutative diagram 
$$
\begin{array}{ccc}
N_{Dol} & \stackrel{\sim}{\rightarrow} & N_B \\
\downarrow & & \downarrow \\
S^{2N-1} & \stackrel{\sim}{\rightarrow} & |{\rm Step}(M_B)| .
\end{array}
$$
\end{conjecture}

The main motivation is that it holds in the first example above. This statement may also be viewed as a 
version of the ``$P=W$'' conjecture of Hausel \emph{et al} \cite{Hausel-et-al-PW}, relating
Leray stuff for the Hitchin fibration to weight stuff on the Betti side. 

A. Komyo  has shown explicitly $|{\rm Step}(M_B)| \cong S^3$ for the case of $\pp ^1-5$ points
\cite{Komyo}. Kontsevich has proposed a general type of argument saying that in many cases $M_B$ are ``cluster varieties'', hence
log-Calabi-Yau, from which it follows that the incidence complex is a sphere.

One can furthermore hope to have a more geometrically precise description of the 
relationship between $N_{Dol}$ and $N_B$. One should note that it will interchange ``small'' and ``big'' subsets.
Indeed, in all examples that we know of, 
the neighborhood of a single vertex of the divisor at infinity in $\overline{N}_B$ corresponds to a whole chamber in $S^{2N-1}$
and hence in $N_{Dol}$. 

It isn't clear whether one might be able to get a fully homotopy-theoretic proof of the above conjecture, for example by
applying the kinds of techniques that Hausel and his co-workers have introduced for their cohomological $P=W$ conjecture.
Otherwise, we will need to have a more precise geometric description. Wentworth has been able to identify 
a transformation between dense subsets of spheres on both sides of this picture. Of course, in order to get a hold of
the full homotopy type, one would need to get a precise geometric description everywhere, and this seems for now to be fairly far away. 

We might proceed by looking at big sets on one side, corresponding to small sets on the other. One of the main considerations in this direction
comes from the work of Kontsevich and Soibelman \cite{Kontsevich-Soibelman-WCS}. They develop a 
picture where vertices or $0$-dimensional pieces of the
Betti divisor at infinity, correspond to \emph{chambers} in $S^{2N-1}$, and
$1$-dimensional pieces of the divisor at infinity in  correspond to \emph{walls} in $S^{2N-1}$ or equivalently ${\mathbb A}^N$. 
Their \emph{wall-crossing formulas} express the change of cluster coordinate systems as we go along these one-dimensional pieces.  

In the present paper, we will be discussing what happens at the opposite end of the range of dimensions, where we expect: 

\begin{itemize}
\item[] divisor components in $\overline{M}_B$
$\longleftrightarrow$ single directions in the Hitchin base.
\end{itemize}

Now, divisor components correspond to \emph{valuations} of the coordinate ring ${\mathcal O}_{M_B}$. However, there are also non-divisorial valuations. 
We expect more generally that all valuations correspond to directions in the Hitchin base, which in turn correspond to spectral curves
$\Sigma \subset T^{\ast}X$ (up to scaling). This is what happens in classical Thurston theory. 

Furthermore, valuations correspond to harmonic maps to buildings. Indeed if $K_v$ is the valued field corresponding to a valuation
on ${\mathcal O}_{M_B}$, then the map
$$
\pi _1(X,x_0)\rightarrow SL_r({\mathcal O}_{M_B})
$$
composes with ${\mathcal O}_{M_B} \subset K_v$ to give 
$$
\pi _1(X,x_0) \rightarrow SL_r(K_v)
$$
hence an action of $\pi _1$ on the Bruhat-Tits building. One can then take the Gromov-Schoen harmonic map. 

In the situation of Gromov-Schoen harmonic maps, we have already known for many years the correspondence:

\begin{itemize}
\item[] harmonic maps to buildings $\longleftrightarrow$ spectral curves\vspace*{.3cm}
\end{itemize}
Indeed, a harmonic map has a differential which is the real part of a multivalued holomorphic form defining a spectral curve. 

We would like to understand the correspondence with the differential equations picture at the same time, and furthermore to
understand the relationship with the \emph{spectral networks} which have recently been introduced by \textbf{Gaiotto-Moore-Neitzke}. 

\subsection{The $SL_2$ case and trees}

Thurston's Teichmuller theory involves a lot of discussion of questions closely related to the WKB problem, particularly
for the group $SL_2$. In that case, the spectral curve of a Higgs field corresponds exactly to a quadratic differential.
The quadratic differential defines a foliation, the leaf space of which is an $\rr$-tree. We get in this way to harmonic maps
towards $\rr$-trees (see \cite{Daskalopoulos-Dostolgou-Wentworth}). One may interpret classical rank two WKB theory as describing the asymptotic behavior of the transport
matrix $T_{PQ}(t)$ for a connection of the form \eqref{mainnabla}, by the distance transverse to the foliation. This picture has
strongly motivated what we will be doing.

\subsection{WKB and harmonic maps to buildings}

In the higher rank case $r\geq 3$, it is natural for the reasons explained above, to look for a relationship between
the WKB problems and harmonic mappings to eucildean buildings. This should generalize the picture we have relating
WKB problems for $SL_2$ and harmonic mappings to trees. 

Recall that $X$ is a Riemann surface, $x_0\in X$, $\cE\rightarrow X$ a vector bundle of rank $r$ with $\bigwedge ^r\cE \cong \cO _X$,
and 
$$
\varphi : \cE\rightarrow \cE\otimes \Omega ^1_X
$$
a Higgs field with ${\rm Tr}(\varphi )=0$. Let 
$$
\Sigma \subset T^{\ast}X \stackrel{p}{\rightarrow} X
$$
be the spectral curve, which we assume to be reduced.

We have a tautological form 
$$
\phi \in H^0(\Sigma , p^{\ast}\Omega ^1_X)
$$
which is thought of as a multivalued differential form. Locally we write
$$
\phi = (\phi _1,\ldots , \phi _r), \;\;\;\; \sum \phi _i = 0.
$$
The assumption that $\Sigma$ is reduced amounts to saying that $\phi _i$ are distinct.

Let $D=p_1+\ldots + p_m$ be the locus over which $\Sigma$ is branched, and $X^{\ast}:= X-D$. 
The $\phi _i$ are locally well defined on $X^{\ast}$. 

We have distinguished two kinds of WKB problems associated to this set of data. \\

\noindent \textbf{(1)} The \textbf{Riemann-Hilbert} or \textbf{complex WKB} problem:\\

Choose a connection $\nabla _0$ on $\cE$ and set 
$$
\nabla _t:= \nabla _0 + t\varphi 
$$
for $t\in \rr _{\geq 0}$. Let 
$$
\rho _t : \pi _1(X,x_0)\rightarrow SL_r(\cc )
$$
be the monodromy representation. We also choose a fixed metric $h$ on $\cE$.

From the flat structure which depends on $t$ we get a family of maps
$$
h_t: \widetilde{X} \rightarrow SL_r(\cc )/SU_r
$$
which are $\rho _t$-equivariant. We would like to understand the asymptotic behavior of
$\rho _t$ and $h_t$ as $t\rightarrow \infty$.

\begin{definition}
For $P,Q\in \widetilde{X}$, let $T_{PQ}(t): \cE_P\rightarrow \cE_Q$
be the transport matrix of $\rho _t$. Define the \emph{WKB exponent} 
$$
\nu_{PQ} := \limsup _{t\rightarrow \infty} \frac{1}{t} \log \| T_{PQ}(t)\| 
$$
where $\| T_{PQ}(t)\| $ is the operator norm with respect to $h_P$ on $\cE_P$ and
$h_Q$ on $\cE_Q$. 
\end{definition}

\textbf{Gaiotto-Moore-Neitzke} \cite{GMN-WKB} consider a variant on the Riemann-Hilbert WKB problem,
associated with a harmonic bundle $(\cE,\overline{\partial} , \varphi , \partial , \varphi ^{\dagger})$ setting
$$
d_t:= \partial + \overline{\partial} + t\varphi + t^{-1} \varphi ^{\dagger}
$$
which corresponds to the holomorphic flat connection $\nabla _t= \partial + t\varphi $
on the holomorphic bundle $(\cE,\overline{\partial} + t^{-1}\varphi ^{\dagger})$. We expect this to have the same
behavior as the complex WKB problem. \\

\noindent \textbf{(2)} The \textbf{Hitchin WKB problem}:\\

Assume $X$ is compact, or that we have some other control over the behavior at infinity. Suppose
$(\cE,\varphi )$ is a stable Higgs bundle. Let $h_t$ be the Hitchin Hermitian-Yang-Mills metric on $(\cE,t\varphi )$
and let $\nabla _t$ be the associated flat connection. Let $\rho _t: \pi _1(X,x_0)\rightarrow SL_r(\cc )$
be the monodromy representation.

Our family of metrics gives a family of \emph{harmonic maps} 
$$
h_t: \widetilde{X} \rightarrow SL_r(\cc )/SU_r
$$
which are again $\rho _t$-equivariant. 

We can define $T_{PQ}(t)$ and $\nu_{PQ}$ as before, here using $h_{t,P}$ and $h_{t,Q}$
to measure $\| T_{PQ}(t)\| $.

Gaiotto-Moore-Neitzke explain that $\nu_{PQ}$ should vary as a function of $P,Q\in X$, in a way
dictated by the \emph{spectral networks}. We would like to give a geometric framework.

\begin{remark} In the complex WKB case, one can view $T_{PQ}(t)$ in terms of Ecalle's \emph{resurgent
functions}. The Laplace transform 
$$
{\mathcal L}T_{PQ}(\zeta ):= \int _0^{\infty} T_{PQ}(t)e^{-\zeta t} dt
$$
is a holomorphic function defined for $|\zeta | \geq C$. It admits an analytic continuation having infinite, but
locally finite, branching. 
\end{remark}

One can describe the possible locations of the branch points, and this description
seems to be compatible with the discussion of Gaiotto-Moore-Neitzke, however in the present paper we look in a different direction. 
 
Namely, we would like to relate their description of WKB exponents, to harmonic mappings to buildings.  
The basic philosophy is that a WKB problem determines a valuation on ${\mathcal O}_{M_B}$ by looking at the exponential 
growth rates of functions applied to the points $\rho _t$. Therefore, $\pi _1$ should act on a Bruhat-Tits building 
and we could try to choose an
equivariant harmonic map 
following Gromov-Schoen.

Recently, Anne Parreau \cite{Parreau-compactification, Parreau-norms} has developed a very useful version of this theory, based on work of Kleiner-Leeb \cite{Kleiner-Leeb}. 
Parreau's work concentrated on the asymptotic behavior of the monodromy representations $\rho _t$,
but by thinking of the fundamental \emph{groupoid} we can extend this to the transport
functions $T_{PQ}(t)$. Look at our maps $h_t$ as being maps into a symmetric space with distance rescaled: 
$$
h_t: \widetilde{X} \rightarrow \left( SL_r(\cc )/SU_r, \frac{1}{t}d \right) . 
$$

Kleiner and Leeb, and Parreau,  now take a ``Gromov limit'' of the symmetric spaces with their rescaled distances, and show that it will be a building
modelled on the same affine space ${\mathbb A}$ as the $SL_r$ Bruhat-Tits buildings. 

The limit construction depends on the choice of \emph{ultrafilter} $\omega$, and the limit is denoted $\Cone_{\omega}$. 
We get a map 
$$
h_{\omega}: \widetilde{X} \rightarrow \Cone_{\omega}, 
$$
equivariant for the limiting action $\rho _{\omega}$ of $\pi _1$ on $\Cone_{\omega}$ which was the subject of \cite{Parreau-compactification}. 

In this situation, the 
main point for us is that we can write
$$
d_{\Cone_{\omega}}\left( h_{\omega} (P),h_{\omega}(Q)\right) = 
$$
$$
\lim _{\omega} \frac{1}{t} d_{SL_r\cc /SU_r} \left( h_t(P),h_t(Q)\right) .
$$

There are several distances on the building, and these are all related by the above
formula to the corresponding distances on $SL_r\cc /SU_r$. 

\begin{itemize}

\item The Euclidean distance $\leftrightarrow$ Usual distance on $SL_r\cc /SU_r$

\item Finsler distance $\leftrightarrow$ $\log$ of operator norm

\item Vector distance $\leftrightarrow$ dilation exponents

\end{itemize}

We are most interested in the \emph{vector distance}. In the affine space
$$
{\mathbb A} = \{ (x_1,\ldots , x_r)\in \rr^r , \;\; \sum x_i = 0\} \cong \rr ^{r-1}
$$
the vector distance is translation invariant, defined by 
$$
\overrightarrow{d}(0,x):= (x_{i_1}, \ldots , x_{i_r})
$$
where we use a Weyl group element to reorder so that $x_{i_1}\geq x_{i_2}\geq \cdots \geq x_{i_r}$. 

In $\Cone_{\omega}$, any two points are contained in a common apartment, so we can use the
vector distance defined as above in that apartment to define the vector distance in the building. 

The ``dilation exponents'' in $SL_r\cc /SU_r$ may be discussed as follows:
put
$$
\overrightarrow{d}(h,k):= (\alpha_1,\ldots , \alpha _r)
$$
where 
$$
\| e_i\| _k = e^{\alpha _i} \| e_i\| _h
$$
with $\{ e_i\}$ a simultaneously $h$ and $k$ orthonormal basis.

In terms of transport matrices,
$$
\alpha_1 = \log \| T_{PQ}(t)\| .
$$
We can furthermore get a hold of the other dilation exponents, by using exterior powers:
$$
\alpha _1+\ldots + \alpha _k = \log \| \bigwedge ^k T_{PQ}(t)\| ,
$$
using the transport matrix for the induced connection on $\bigwedge ^k\cE$. In this way, intuitively we can restrict
to mainly thinking about $\alpha _1$, which was the ``Finsler metric''. 

\begin{remark} 
For $SL_r\cc/SU_r$ we are only interested in these metrics ``in the large'' as they pass to 
the limit after rescaling. 
\end{remark}

Our rescaled distance becomes
$$
\frac{1}{t} \log \| T_{PQ}(t)\| .
$$
Define the \emph{ultrafilter exponent} 
$$
\nu ^{\omega}_{PQ}:= \lim _{\omega} \frac{1}{t} \log \| T_{PQ}(t)\| .
$$

This should be compared with the exponent $\nu _{PQ}$ considered in Definition 1.2.

\begin{proposition}
We have
$$
\nu ^{\omega}_{PQ}\leq \nu_{PQ}.
$$
Furthermore, they are equal in some cases:
\begin{enumerate}[label = \emph{(\alph*)}]
\item  for any fixed choice of $P,Q$, there exists a choice of ultrafilter $\omega$ such that 
$\nu ^{\omega}_{PQ} = \nu_{PQ}$. 

\item If $\limsup _{t} \ldots = \lim _t \ldots $ then it is the same as $\lim _{\omega} \ldots $, so $\nu ^{\omega}_{PQ}=\nu _{PQ}$.  This will apply in particular for the local WKB case to be seen below. 
\end{enumerate}
\end{proposition}

The inequality is due to the fact that the ultrafilter limit is less than the $\limsup$. Part (b) is clear, and part (a)
holds by subordinating the ultrafilter to the condition of having a sequence calculating the $\limsup$ for
that pair $P,Q$. 

It isn't \emph{a priori} clear whether we can choose the ultrafilter so that equality holds for all pairs of points $P$ and $Q$. 
Part (b) would in fact apply in the complex WKB case, for generic angles, if we knew that the Laplace transform ${\mathcal L}T_{PQ}(\zeta )$ didn't
have essential singularities. This is true, so one can show that (b) holds, for generic angles, in some cases where the spectral curve decomposes into a union
of sections, i.e. $\phi_i$ are single-valued.

\begin{theorem}[``Classical WKB'']\label{classical-WKb}
\label{localWKB}
Suppose $\gamma : [0,1]\rightarrow \widetilde{X}^{\ast}$ is \emph{noncritical} path i.e. 
$\gamma^{\ast} \re \phi _i$ are distinct for all $t\in [0,1]$. Reordering we may assume
$$
\gamma^{\ast} \re \phi _1 > \gamma^{\ast} \re \phi _2 > \ldots > \gamma^{\ast} \re \phi _r .
$$
Then, for the complex WKB problem we have
$$
\frac{1}{t} \overrightarrow{d}\left( h_t(\gamma (0)), h_t(\gamma (1)\right) \sim (\alpha _1,\ldots , \alpha _r)
$$
where 
$$
\alpha_i = \int _{0}^1 \gamma ^{\ast}\re \phi _i .
$$
\end{theorem}

\begin{corollary}
At the limit, we have
$$
\overrightarrow{d}_{\omega} \left( h_t(\gamma (0)), h_t(\gamma (1)\right) = (\alpha _1,\ldots , \alpha_r).
$$
\end{corollary}

The above theorem has been stated for the complex or Riemann-Hilbert WKB problem. It should also extend directly
to the variant considered by Gaiotto-Moore-Neitzke. We also conjecture that the same local estimate should hold in Hitchin's case. 

\begin{conjecture}
The same should be true for the Hitchin WKB problem.
\end{conjecture}

This would involve estimates on the asymptotic behavior of the harmonic metric for $(\cE,t\varphi )$ as $t\rightarrow \infty$.

We now give the main corollary of this statement. It is a corollary of the
theorem, in the complex WKB case, and would be a corollary of the conjecture in the Hitchin case.

\begin{corollary} 
\label{localWKBlongrange}
If $\gamma : [0,1]\rightarrow \widetilde{X}^{\ast}$ is any noncritical path, then 
$h_{\omega}\circ \gamma$ maps $[0,1]$ into a single apartment, and the vector distance which determines
the location in this apartment is given by the integrals: 
$$
\overrightarrow{d}_{\omega} \left( h_t(\gamma (0)), h_t(\gamma (1)\right) = (\alpha _1,\ldots , \alpha _r).
$$
\end{corollary}

This just follows from a fact about buildings: if $x,y,z$ are three points with
$$
\overrightarrow{d}(x,y) + \overrightarrow{d}(y,z) = \overrightarrow{d}(x,z) 
$$
then $x,y,z$ are in a common apartment, with $x$ and $z$ in opposite chambers centered at $y$ or 
equivalently, $y$ in the Finsler convex hull of $\{ x,z\}$.

\begin{corollary}
Our map 
$$
h_{\omega} : \widetilde{X}\rightarrow \Cone_{\omega}
$$
is a harmonic $\phi$-map in the sense of Gromov and Schoen. In other words, any point in the complement of a discrete set of points
in $\widetilde{X}$ has a neighborhood which maps into a single apartment, and the map has differential ${\rm Re}\phi $, in particular 
there is no ``folding''. 
\end{corollary}

This finishes what we can currently say about the general situation: we get a harmonic $\phi$-map 
$$
h_{\omega} : \widetilde{X}\rightarrow \Cone_{\omega}
$$
depending on choice of ultrafilter $\omega$, with 
$$
\nu ^{\omega}_{PQ}\leq \nu_{PQ},
$$
and we can assume that equality holds for one pair $P,Q$. Also equality holds
in the local case. We expect that one should be able to choose a single $\omega$
which works for all $P,Q$.

The next goal is to analyse harmonic $\phi$-maps in terms of spectral networks. 

The main observation is just to note that the reflection hyperplanes in the building,
pull back to curves on $\tilde{X}$ which are imaginary foliation curves, including
therefore the spectral network curves. 

Indeed, the reflection hyperplanes in an apartment have equations $x_{ij}={\rm const.}$
where $x_{ij}:= x_i-x_j$, and these pull back to curves in $\widetilde{X}$ with
equation ${\rm Re} \phi _{ij} = 0$. This is the equation for the ``spectral network curves'' of
Gaiotto-Moore-Neitzke.

\subsection{The Berk-Nevins-Roberts example}

In order to try to understand the role of the \emph{collision} spectral network curves 
in terms of harmonic mapings to a building, we decided to look closely at a classical example:
it was the original example of Berk-Nevins-Roberts \cite{BNR} which introduced the ``collision phenomenon'' special
to the case of higher-rank WKB problems. 

They seem to be setting $\hbar = 1$, a standard physicist's move. If we undo that, we can say
that they consider a family
of differential equations with large parameter $t$, 
of the form
$$
(\frac{1}{t^3} \frac{d^3}{dx^3}  - \frac{3}{t} \frac{d}{dx} + x ) f = 0. 
$$
When we use the companion matrix we obtain a Higgs field $\varphi$ with spectral curve given by
the equation 
$$
\Sigma \;\; : \;\;\; p^3 - 3 p + x = 0
$$
where $X=\cc$ with variable $x$, and $p$ is the variable in the cotangent direction. It is pictured in real variables in Figure \ref{Real-Spec}.

	\begin{figure}[h]
	\centering
	\includegraphics[width=.5\textwidth]{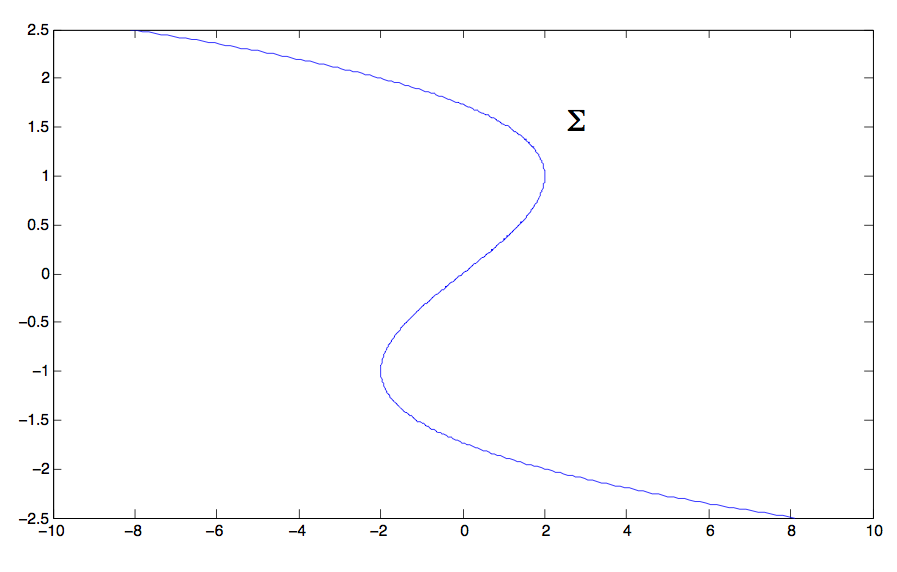}
	\caption{The real points of the BNR spectral cover}
	\label{Real-Spec}
	\end{figure}
 
The differentials $\phi _1,\phi _2$ and $\phi _3$ are of the form $p_idx$ for $p_1,p_2,p_3$
the three lifts of $p$ as a function of $x$. 

Notice that $\Sigma \rightarrow X$ has branch points 
$$
b_r= 2,\;\;\;  b_l = -2.
$$

The imaginary spectral network
is as in the accompanying picture (Figure \ref{BNR-spectral-network}), which is the same as in the Berk-Nevins-Roberts paper \cite{BNR}.

\begin{figure}[h]
\centering
\includegraphics[width=.5\textwidth]{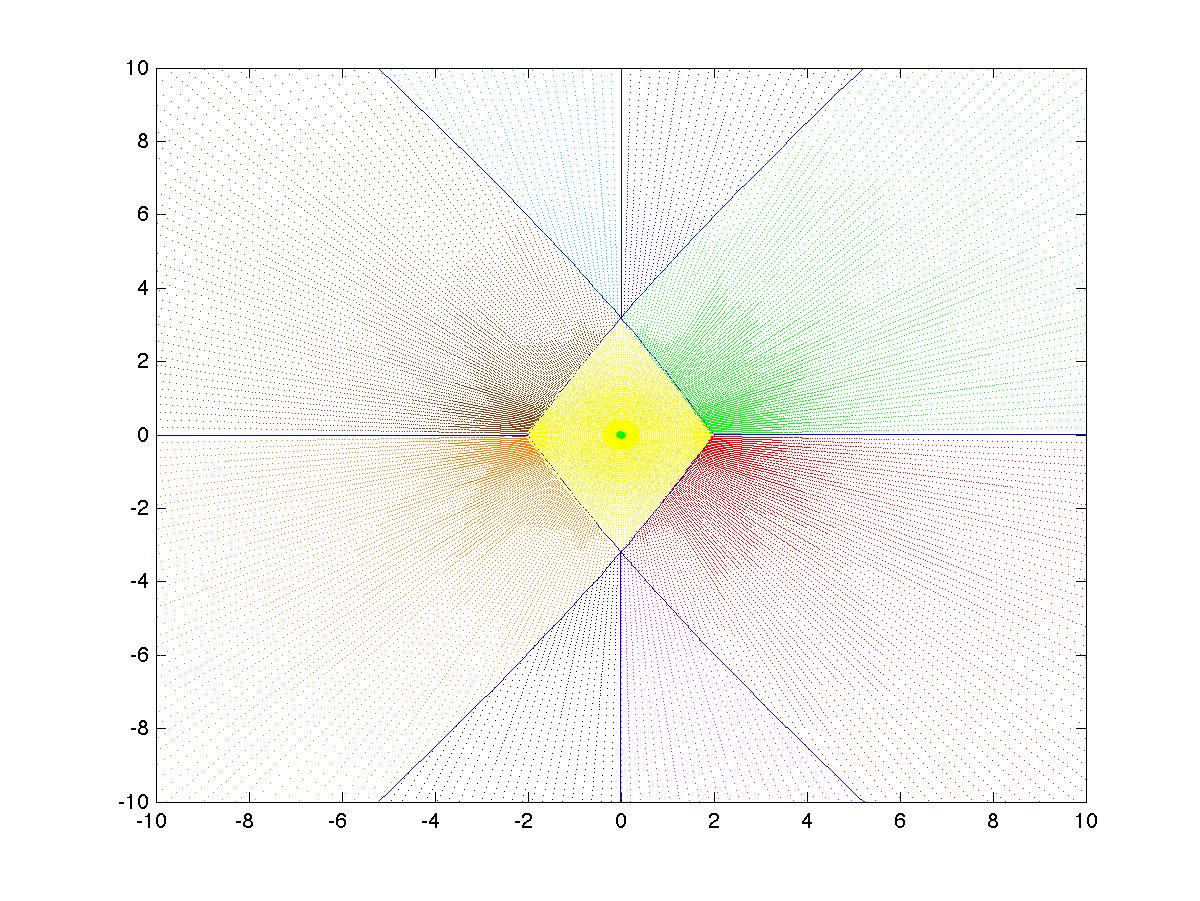}
\caption{The BNR Spectral Network}
\label{BNR-spectral-network}
\end{figure}

Notice the following salient features: 

\begin{itemize}
\item There are two collision points, which in fact lie on the same vertical collision line. 

\item The spectral network curves divide the plane into $10$ regions: 

\item $4$ regions on the outside to the right of the collision line; 

\item $4$ regions on the outside to the left of the collision line;

\item $2$ regions in the square whose vertices are the singularities and the collisions; the two regions
are separated by the interior part of the collision line. 

\end{itemize}

The first and perhaps main step of the analysis is
to say, using the local WKB approximation of Theorem \ref{localWKB} and
Corollary \ref{localWKBlongrange}, we can conclude: 

\begin{proposition}
Each region cut out by the spectral network is mapped into a single Weyl sector in
a single apartment
of the building $\Cone_{\omega}$. 
\end{proposition}

In particular, the interior square (containing in fact two regions)
maps into a single apartment, with a fold  line along the ``caustic'' joining the two singularities.
The fact that the whole region goes into one apartment comes from an argument with the axioms of the building and special
to the fold line. 

These interior regions, which will be colored ``yellow'', are somewhat special in the present discussion. 
They are the only ones which do not map surjectively onto their corresponding sectors. In this sense, they
represent perhaps most closely the behavior which is necessarily to be expected for even higher rank $r\geq 4$, when
the dimension of the building is strictly larger than the dimension of the Riemann surface. 

However, for the moment, we make use of the fact that many pieces of the Riemann surface map surjectively onto their
corresponding sectors, in order to understand this first example. 

For proofs of statements such as the proposition, we found the paper of Bennett, Schwer and Struyve \cite{Bennett-Schwer}
about axiom systems for buildings, based on Parreau's paper \cite{Parreau-norms}, to be very useful.  

Another special property also holds in this example: 

\begin{lemma}
In this case, the two collision points map to the same point in the building. 
\end{lemma}

This is shown by making
a contour integral and using the fact that the interior region goes into a single apartment. 

\begin{corollary}
The sectors in question all correspond to sectors in the building with a single vertex. 
\end{corollary} 

In view of the picture of sectors starting from a single vertex, we can switch over from affine buildings to spherical buildings.
An $A_2$ spherical building is just a graph, such that any two points have distance maximum $3$, any two edges are contained
in a hexagon, and there are no loops smaller
than a hexagon.

Our situation corresponds to an octagon: the eight exterior sectors (see Figure \ref{octagon-picture}). In this case, one can inductively construct a spherical building,
by successively completing uniquely each path of length $4$ to a hexagon. It corresponds to completing any adjacent sequence of $4$ distinct sectors,
to their convex hull which is an apartment of $6$ sectors. 

\begin{lemma}
There is a universal $A_2$ spherical building ${\B}^8$ containing an octagon $G\subset {\B}^8$.
For any other $A_2$ spherical building $\B$ and map from an octagon $G\rightarrow \B$ it extends
in a unique way to a map ${\B}^8\rightarrow \B$.  
\end{lemma}

The universal ${\B}^8$ is constructed by successively completing any chain of four segments not contained in a hexagon,
to a hexagon. In this process, one adds at most one pair of edges between any two vertices. 

The two sectors which contain the images of the two interior ``yellow'' zones of $X$, correspond to the first two
new segments which would be added. They complete both the upper and the lower sequences of four edges. These two new edges
correspond to a  pair of sectors which will contain the image of the yellow regions. After this, we obtain the following
picture in spherical terminology 
corresponding to the following picture in the affine building: there are two apartments, corresponding to the upper and
lower hexagons of the diagram, glued together along a pair of sectors into which the yellow regions map.

\begin{figure}[h]
\centering
\includegraphics[width=.5\textwidth]{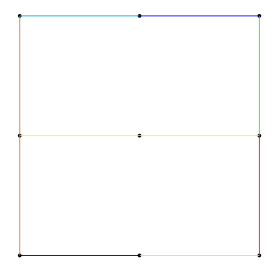}
\caption{The Octagon in the link of the building}
\label{octagon-picture}
\end{figure}

As we continue adding pairs of edges to the diagram to construct ${\B}^8$,
corresponding to adding pairs of sectors to the above picture, the main observation is that opposite edges in the octagon cannot go into a hexagon intersecting the octagon in $5$ segments. 
Rather, they have to go to a twisted hexagon reversing directions. It is here that we see the collision phenomenon.

The inverse image of the apartment corresponding to this twisted hexagon, in $X$, is disconnected. Thus, if 
$P,Q$ are points in the opposite sectors, then the distance in the building is not calculated by any integral of a single $1$-form
from $P$ to $Q$. The $1$-form has to jump when we cross a collision line. This is the collision phenomenon. 

Our main result is as follows.

\begin{theorem}
In the BNR example, there is a \emph{universal building} $\Bu$ together with a harmonic $\phi$-map 
$$
h^{\phi}:X \rightarrow \Bu
$$
such that for any other building $\B$  and harmonic 
$\phi$-map $X\rightarrow \B$ there is a unique factorization 
$$
X\rightarrow \Bu \stackrel{\psi}{\rightarrow} \B.
$$
Furthermore, on the Finsler secant subset of the image of $X$, $\psi$ is an isometry for any
of the distances (this depends on the non-folding property of $\psi$). Hence, 
distances in $\B$ between points in $X$ are the same as the distances in $\Bu$.
\end{theorem}

Applying  gives: 

\begin{corollary}
In the BNR example, for any pair $P,Q\in X$, the WKB dilation exponent is calculated as the distance in the building $\Bu$, 
$$
\overrightarrow{\nu}_{PQ} = 
\overrightarrow{d}_{\Bu} (h^{\phi}(P), h^{\phi} (Q)). 
$$
\end{corollary}

Indeed, for any $P$ and $Q$ there is an ultrafilter $\omega$ such that $\nu ^{\omega}_{PQ}= \nu_{PQ}$. Applying the theorem 
to $\B= \Cone_{\omega}$, and recalling that $\nu ^{\omega}_{PQ}$ is the distance in $\Cone_{\omega}$,
we get that $\nu_{PQ}$ is the distance between the images of $P$ and $Q$ in $\Bu$. The same discussion holds for
the vector distance and vector WKB exponent $\overrightarrow{\nu}_{PQ}$.

The isometry property is not always true. 
There exist examples, such as pullback connections, where we can see that the isometry property is certainly not true. However we conjecture
that it should be true under some genericity hypotheses. 

On the other hand, the universal property doesn't necessarily require having an isometry, and we make the following conjecture
for the general case, together with an isometry statement under genericity hypotheses. 

\begin{conjecture}
For any spectral curve with multivalued differential $\phi$,  
there is a universal harmonic $\phi$-map to a ``building'' or building-like object $\Bu$. 
Furthermore, if the spectral curve $\Sigma$ is smooth and irreducible, and $\nabla _0$ is generic, then
$$
\overrightarrow{\nu}_{PQ} = 
\overrightarrow{d}_{\Bu} (h^{\phi}(P), h^{\phi} (Q)). 
$$
\end{conjecture}

It might be necessary to restrict to some kind of target object which is somewhat smaller than a building. 
In the example we treat here, some special largeness properties hold: many of the sectors in the building are images of
regions in the Riemann surface. We don't yet know exactly what kind of building-like object we should look for
as the universal $\Bu$. 

Our universal object should be thought of as
the higher-rank analogue of the \emph{space of leaves of a foliation} which shows up in the $SL_2$ case
in classical Thurston theory.  Notice that for the case of $SL_2$, the universal tree exists: it is just the
space of leaves of the foliation determined by the quadratic differential.  Our goal in this project is
to try to obtain a generalization of the
 ``space of leaves'' picture, to the
higher-rank case. The present paper constitutes a first step in this direction. 

It is hoped that this theory  will later help with stability conditions on categories.

\subsection{Organization, notation and conventions}

This paper is organized as follows:

In $\S$ \ref{Metric-Structures} we introduce the structures and constructions from metric geometry that play a central role in the paper. Buildings are defined in $\S$ \ref{affine-spherical-buildings}, and their relationship to symmetric spaces via the asymptotic cone construction is explained in $\S$ \ref{symmetric-spaces}. We also introduce the notion of a vector valued distance in $\S$ \ref{vector-distance-section}; this notion plays an important role in formulating the WKB problem in $\S$ \ref{WKB}. 

In $\S$ \ref{universal-building-section}, we explain how harmonic maps from a Riemann surface to a building are related to spectral and cameral covers of the Riemann surface. We formulate the notion of a harmonic $\phi$-map here, and define the universal building $\B^{\phi}$ as a harmonic $\phi$-map that is initial among all harmonic $\phi$-maps in a certain sense. $\S$ \ref{phi-map-properties} collects together various useful propositions about the behavior of $\phi$-maps, which play an important role in the WKB analysis of $\S$ \ref{WKB}, as well as in proving the universality of the building constructed in $\S$ \ref{BNR-example}. In $\S$ \ref{gluing-construction} we describe a way of constructing a ``pre-building'' from certain covers of the Riemann surface by ``gluing in an apartment'' for every element of the cover. This construction is used in $\S$ \ref{BNR-example} to construct the universal building. Finally, we conclude this section in $\S$ \ref{spectral-networks-buildings} by formulating the conjecture relating spectral networks to the singularities of the universal building.

$\S$ \ref{Riemann-Hilbert-Hitchin} is devoted to formulating precisely the main WKB problem that we consider, namely the problem of determining the asymptotic behavior of the Riemann-Hilbert correspondence. In $\S$ \ref{local-WKB-section}, we study this analyze this problem for short paths. In the final subsection $\S$ \ref{cone-computes-spectrum} we prove the main theorem relating WKB to buildings, which states that the WKB spectrum for the Riemann-Hilbert problem can be expressed as the distance in a building. 

In $\S$ \ref{BNR-example}, we analyze the example of the spectral cover studied by Berk, Nevins and Roberts in their seminal paper on higher order Stokes phenomena. We prove that a universal building exists in this example, and computes the WKB spectrum. Furthermore, we prove the conjecture relating singularities of the universal building to spectral networks. We refer the reader to Outline \ref{BNR-outline} for a detailed description of the organization of $\S$ \ref{BNR-example}.  

Finally we conclude this subsection by collecting together some notation and conventions that will be used in what follows:

\begin{itemize}

\item[--] Harmonic maps to targets that are singular were introduced in \cite{Gromov-Schoen}. We will make heavy use of the theory of harmonic maps from a Riemann surface to non-negatively curved metric spaces that are more general than those considered in \cite{Gromov-Schoen}. When we speak of a harmonic map from a Riemann surface to a building, we will mean a harmonic map in the sense of \cite{Korevaar-Schoen}.

\item[--] Unless otherwise stated, we will use the term building to mean affine $\mathbb{R}$-building. When we refer to spherical buildings, we will explicity use the adjective ``spherical''.

\item[--] We will typically use the symbols:
\begin{itemize}
\item[-] $X$ for a Riemann surface
\item[-]$\tilde{X}$ for its universal cover
\item[-] $\gamma$ for a path in a Riemann surface or an arrow in a groupoid
\item[-] $\B$ for an affine building
\item[-] $G$ for a complex semisimple Lie group, $\mathfrak{g}$ for its Lie algebra
\item[-] $\phi$ for a point in the Hitchin base, $\Sigma_{\phi}$ for the corresponding cameral cover, $\mathfrak{t}$ for a Cartan subalgebra
\item[-] $\Sigma$ for a spectral cover
\item[-] $P$, $Q$,... for points on a Riemann surface
\item[-] $x$, $y$,... for points in a building
\item[-] $\waff$ for an affine Weyl group, $\wlin$ or $W$ for its spherical part, unless specified otherwise
\item[-] $\cE$ for a holomorphic vector bundle
\item[-] $T$ for the parallel transport operator ($T$ stands for ``transport'')
\item[-] $d$ for ordinary distance, and $\vecd$ for vector valued distance
\item[-] $\nu$ for the WKB exponent, and $\nuu$ for the WKB dilation spectrum
\end{itemize} 

\item[--]  Unless explicity stated otherwise, the complex semisimple Lie group $G$ occuring in the statements of theorems, propositions and lemmas will be assumed to be $\SL_r\cc$ for some $r$. This restriction does not apply to conjectures and definitions.

\item[--] Unless explicity stated otherwise, all affine buildings $\B$ occuring in the statements of Theorems, Propositions and Lemmas (except in $\S$ \ref{Metric-Structures}) will have Weyl group $\waff \simeq W \ltimes \mathbb{R}^{r-1}$ where $W$ is the spherical Weyl group of type $A_{r}$. This restriction does not apply to definitions and conjectures. 

\item[--] In Section \ref{BNR-example} we will in addition assume implicity that $r = 2$ in the previous point.

\end{itemize}

\subsection*{Acknowledgements}

The authors would like to thank D. Auroux, G. Daskalopoulos, A. Goncharov, F. Haiden, M. Kontsevich, C. Mese, A. Neitzke, T. Pantev, M. Ramachandran, R. Schoen and Y. Soibelman for several useful conversations related to the subject matter of this paper. The first, second, and third named authors were funded by NSF DMS 0854977 FRG, NSF DMS 0600800, NSF DMS 0652633 FRG, NSF DMS 0854977, NSF DMS 0901330, FWF P 24572 N25, by FWF P20778 and by an ERC Grant. The last named author's research is supported in part by the Agence Nationale de la Recherche grants ANR-BLAN-08-1-309225 (SEDIGA) and  ANR-09-BLAN-0151-02 (HODAG).


\section{Metric Structures}\label{Metric-Structures}
This section collects together the basic notions from metric geometry that will be used in the sequel. It can be skipped or skimmed on a first reading, and referred back to when necessary.


\subsection{Affine and spherical buildings}\label{affine-spherical-buildings}

Buildings were introduced by Tits as a category of objects where every Lie group could be realized as an automorphism group. In the decades that followed, buildings have found numerous applications in various areas of mathematics, ranging from geometric group theory and metric geometry to representation theory and higher Teichm\"uller theory. The reader who is new to buildings will find a leisurely and informal survey in \cite{Everitt}. The survey \cite{Rousseau} treats the $\mathbb{R}$-buildings that will be used in this paper. For a more detailed discussion of buildings, the reader might consult the textbooks \cite{Ronan} and \cite{Abramenko-Brown}. 

The ubiquity of buildings in mathematics is reflected in the plethora of different axiomatizations of buildings, ranging from the purely combinatorial chamber systems of Tits, to the purely geometric axioms of Kleiner-Leeb \cite{Kleiner-Leeb}. The point of view that most directly makes contact with the theme of this paper is that a building is a metric space that is obtained by ``gluing apartments'', with the ``gluing maps'' being given by elements of a reflection group. The apartments are either spheres or Euclidean spaces, and correspondingly there are two types of buildings -- affine and spherical. The relevant notions will be recalled below, closely following the treatment in \cite{Rousseau}. A comparison of various axiom systems can be found in  \cite{Parreau-thesis} and \cite{Bennett-Schwer}. 

Before discussing the formal definitions, the reader is invited to contemplate the following motivating examples:

\begin{figure}[h]
\centering
\includegraphics[width=.5\textwidth]{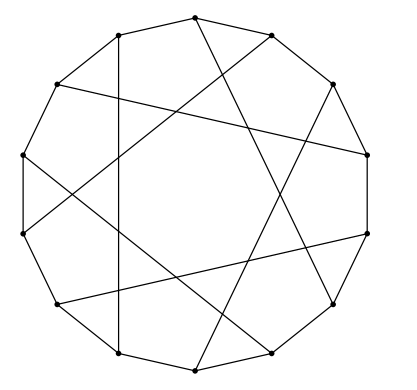}
\caption{The spherical building of $\mathrm{SL}(3, \mathbb{F}_2)$}
\label{SL3F2}
\end{figure}

\begin{example}\label{flag-complex}
\emph{(Flag complexes)}. Let $V$ be a vector space of dimension $n \in \mathbb{N}$ over a field $\mathbb{F}$. Let $W = \mathbb{S}_n$ be the symmetric group on $n$-letters, thought of as a Coxeter group with generating reflections given by the transpositions $(ij)$, with $j = i +1$. Associated to $V$, there is a spherical building $\B(V)$ modelled on $W$, whose \emph{chambers} are given by complete flags $0 = V_0 \subset V_1 \subset ... \subset V_n = V$ in $V$. To any basis $(e_1,...,e_n)$ of $V$, one can associate a flag in the obvious way by letting $V_i = \mathrm{span}_{j \leq i}\{ e_j \}$. The set $A_{\{e_1,...,e_n\}}$ of flags/chambers obtained in this way from the various permutations $(e_{\sigma(1)},...,e_{\sigma(n)})$, $\sigma \in \mathbb{S}_n$ of some given basis $(e_1,...,e_n)$ constitutes an \emph{apartment} in $\B(V)$. 

Let $F_1$ and $F_2$ be flags in a given apartment $A_{\{e_1,...,e_n\}}$ given by bases $(e_{\sigma(1)},...,e_{\sigma(n)})$ and $(e_{\eta(1)},...,e_{\eta(n)})$. Then the element $d(F_1, F_2) := \eta^{-1} \sigma$ gives a ``combinatorial distance'' from $F_1$ to $F_2$ (note that this notion of distance is not symmetric). The building axioms for $\B(V)$ can be understood, roughly, as saying that

\begin{enumerate}
\item given any two flags $F$ and $F'$ in $\B(V)$ there is an apartment $A_{\{e_1,...,e_n\}}$ containing them both.

\item The (combinatorial) distance between $F$ and $F'$ is independent of the apartment in which it is measured.

\end{enumerate}

When $\mathbb{F}$ is a finite field, there are only finitely many chambers in $\B(V)$. Figure \ref{SL3F2} shows a picture of $\B(V)$ for $V$ a 3-dimensional vector space over $\mathbb{F}_2$. The chambers are the (interiors of) edges joining two vertices. The apartments are the embedded hexagons in the picture. Thus, each apartment consists of six chambers/flags. Observe that each apartment is topologically a sphere -- the dimension of this sphere is the \emph{rank} of the building (which is one in this case).
\end{example}

\begin{example}\label{Bruhat-Tits-building}
\emph{(Bruhat-Tits buildings)}. In very crude terms, one may summarize the content of this example by saying that if one replaces the vector space in Example \ref{flag-complex} by a vector bundle on a curve, one obtains an \emph{affine building}. More precisely, let $(\mathbb{K}, \nu)$ be a discrete valuation ring with uniformizing parameter $t$, and let $\cO$ be the valuation ring. For instance, one may take $\mathbb{K} = \mathbb{F}((t))$ with the usual valuation. Let $V$ be a vector space of finite dimension $n$ over $\mathbb{K}$. Recall that a lattice $L \subseteq V$ is a free $\cO$-submodule of rank $n$. The \emph{Bruhat-Tits} building $\B(V, \mathbb{K})$ is an affine building which can be constructed as the geometric realization of a simplicial object whose $k$-simplices are flags of lattices of the form $L_0 \subset L_1 \subset ... \subset L_k \subset t^{-1} L_0$. As in the previous example, the apartments are given by fixing a basis for $\mathbb{K}$, and considering only those lattices ``generated by the given basis''. For more details, the reader is referred to \cite{Abramenko-Brown}. Figure \ref{SL3Q2} shows a picture of the case where $\mathbb{K} = \mathbb{Q}_2$, and $V$ is a $3$-dimensional vector space over $\mathbb{K}$ (the tree extends to infinity in all directions - only a part of it is visible in the picture). 

\end{example}

\begin{figure}[h]
\centering
\includegraphics[width=.5\textwidth]{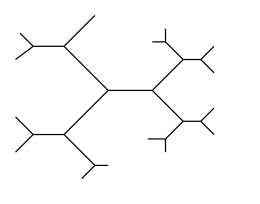}
\caption{The Bruhat-Tits building of $\mathrm{SL}(3, \mathbb{Q}_2)$}
\label{SL3Q2}
\end{figure}

\begin{example}\label{trees}
\emph(Trees). More generally, any tree $\cT$ is a rank 1 affine building. The apartments are embedded copies of the real line. In the case of an ordinary tree (as opposed to an $\mathbb{R}$-tree), the chambers are the connected open intervals in $\cT$ whose closure contains two vertices, as in Example \ref{flag-complex}. 

In the case of $\mathbb{R}$-trees, the set of points where $\cT$ branches can become dense in the tree. Thus, the notion of an edge of the tree does not make sense in this setting. However, it turns out that one can still make sense of the notion of a chamber -- a sort of ``infinitesimal edge'' -- by introducing the notion of filters (see the definitions below). 

The leaf space of a quadratic differential on a Riemann surface will in general be an $\mathbb{R}$-tree. This is one of the motivating examples of buildings from the point of view of (higher) Teichm\"uller theory, which is the point of view that inspired the current work. 
\end{example}

\begin{example}
The asymptotic cone of a symmetric space of non-compact type is an affine building. This example, which is discussed at greater length in Section \ref{symmetric-spaces}, will play an important role in the WKB problem addressed in Section \ref{WKB}. 
\end{example}

We now turn to the formal definition of a building, motivated by the examples above. At the first reading, the reader may want to keep in mind Example \ref{flag-complex}, and replace the word ``facet'' in the definition below by ``chamber''.  A good part of the remainder of this section will be devoted to explaining the precise meaning of the notion of an ``apartment'', and the various other terms that are central to the following definition. Our treatment follows \cite{Rousseau} very closely.

\begin{definition}\label{building}
An \emph{affine building} (resp. \emph{spherical building}) is a triple $(\B, \cF, \cA)$ consisting of (i) a set $\B$, (ii) a collection $\cF$ of filters on $\B$ (called the \emph{facets} of $\B$) and (iii) a collection $\cA$ of subsets $A$ of $\B$ called \emph{apartments}, each endowed with a metric $d_{A}$, satisfying the following axioms:

\begin{itemize}
\item[(B0)] For $A \in \cA$, let $\cF_{A} := \{ \sigma \in \cF | A \in \sigma \}$ be the set of filters contained in $A$. Then for each apartment $A$,  $((A, d_{A}), \cF_{A})$ is isomorphic to a Euclidean (resp. spherical) apartment.

\item[(B1)] For any two  facets $F$ and $F'$ in $\cF$, there is an apartment $A$ containing $\overline{F}$ and  $\overline{F'}$. 

\item[(B2)] For any two apartments $A$ and $A'$, their intersection is a union of facets. For any two facets $F$, $F'$ in $A \cap A'$, there exists an isometry of apartments $A \rightarrow A'$ that carries $\cF_{A}$ to $\cF_{A'}$ and fixes $\overline{F}$ and $\overline{F'}$ pointwise. 
\end{itemize}

\end{definition}

\begin{remark}
It follows from $(B1)$ that any two points in a building $\B$ are contained in a common apartment $A$. Furthermore, it follows immediately from $(B2)$ that there is a well defined distance function $d$ defined on $\B$ which coincides with $d_{A}$ on each apartment. The triangle inequality holds for $d$, although this is not quite as immediate (see e.g. \cite{Rousseau}).  
\end{remark}

We now turn to the definition of apartments, and the various terms used in Definition \ref{building}. Let $\euc$ be a Euclidean space, i.e., a real vector space $V$ together with a non-degenerate inner product $\langle \relbar ,\relbar \rangle$. Let $\aff$ be an affine space over $V$, viewed as a metric space with the natural metric induced by $\euc$.  An \emph{affine hyperplane} in $\aff$ is a codimension one affine subspace. A \emph{reflection} of  $\aff$ is an isometry  $r:\aff \rightarrow \aff$ of order 2 whose fixed point set is an affine hyperplane $H$. Given a hyperplane $H$, there is a unique reflection $r_{H}$ of $\aff$ whose fixed set is $H$.

The group $\Aff(\aff)$ of affine transformations of $\aff$ is isomorphic to $ V \rtimes \GL(V)$. For any subgroup $\waff$ of $\Aff(\aff)$, $\wlin$ denotes its image in $\GL(V)$. A subgroup $\waff$ of the group of affine isometries of $\aff$ is called an \emph{affine reflection group} if it is generated by affine reflections, and $\wlin$ is finite. The group $\wlin$ acts naturally on the unit sphere $S(\euc) := \{ v \in \euc | \langle v, v \rangle = 1 \}$; we may identify $\wlin$ with its image $\wsph$ in the isometry group of $S(\euc)$. $\waff$ is a \emph{linear reflection group} if it has a fixed point, in which case it can be identified with the subgroup $\wlin$ of $GL(V)$.

Let $\waff$ be an affine reflection group, and let $\cH$ be the collection of affine hyperplanes $H$ in $\aff$ with the property that there exists $w$ in $\waff$ such that $H$ is the fixed point set of $\waff$. Denote by $\cHlin$ the collection of vector subspaces $U$ of $V$ such that there exists $H \in \cH$ such that $U = \{ v - w | v, w \in H \}$. Denote by $\cHsph$ the collection of subsets $C \subset S(\euc)$ such that there exists $H \in \cHlin$ such that $C = H \cap S(\euc)$. An \emph{affine Coxeter complex} (resp. \emph{spherical Coxeter complex}) is a pair $(\aff, \waff)$ consisting of a affine Euclidean space $\aff$ (resp. a sphere $S(\euc)$), and an affine reflection group $\waff$ (resp. linear reflection group $\wlin \simeq \wsph$) acting on $\aff$ (resp. $S(\euc)$). Coxeter complexes are the basic building blocks of buildings. An affine Coxeter complex $(\aff, \waff)$ is completely determined by the corresponding set of hyperplanes $\cH$; a similar statement is true for a spherical Coxeter complex, with $\cH$ replaced by the corresponding set of ``spherical hyperplanes'' $\cHsph$ introduced above.

\begin{definition}\label{vector-facets}
Let $\wlin$ be a linear reflection group acting on the Euclidean space $\euc$. We identify $\wlin$ with the spherical reflection group $\wsph$ by restricting its action to the unit sphere. Let $\cHlin$ and $\cHsph$ be as defined in the paragraph above. Elements of $\cHlin$ and $\cHsph$ will be called \emph{walls}. An \emph{open half apartment} is a connected component of the complement of a wall.  Define an equivalence relation on $\euc$ (resp. $S(\euc)$) by saying that $x \sim y$ iff the set of open half apartments and walls containing $x$ coincides with the set of open half apartments and walls containing $y$. 

\begin{enumerate}
\item A vectorial \emph{facet} is an equivalence class for the equivalence relation $\sim$\item The \emph{support} of a facet is the intersection of walls containing it. The dimension of a facet is the dimension of its support as an abstract manifold.
\item A facet maximal with respect to inclusion is called a \emph{Weyl chamber} (or just a \emph{chamber} in the spherical case).
\item A \emph{panel} is a facet whose support is of codimension 1 in $\euc$ (resp. $S(\euc)$. 
\end{enumerate}

\end{definition}

\begin{definition}\label{half-apartment}
Let $(\aff, \waff)$ be an affine Coxeter complex, given by a set of reflection hyperplanes $\cH$ called \emph{walls}. 

\begin{itemize}
\item An \emph{open half apartment} in $\aff$ is a connected component of the complement of a wall. 

\item A \emph{sector} in $\aff$ is a subset of the form $x + \Delta$, where $\Delta$ is a Weyl chamber in $(\euc, \wlin)$. We say that such a sector is \emph{based at $x$}

\item A \emph{germ} of Weyl sectors based at $x$ is an equivalence class of Weyl sectors based at $x$ for the following equivalence relation: $S$ and $S'$ are equivalent if $S \cap S'$ is a neighborhood of $x$ in $S$ (resp. $S'$). The germ of a Weyl sector $S$ at $x$ is denoted $\Delta_xS$. 

\item A \emph{half-apartment} in $\aff$ is the closure of an open half apartment. 

\item An \emph{enclosure} in $\aff$ is the intersection of a collection of  half-apartments. The intersection of enclosures is an enclosure. If $Q \subset \aff$ then the intersection of enclosures containing $Q$ is called its hull, and denote $\hull(Q)$. A subset $Q$ is said to be \emph{enclosed} or \emph{Finsler convex} if it equals its hull. 
\end{itemize}
\end{definition}

Let $(\aff, \waff)$ be an affine Coxeter complex, and let $\cH$ be the set of reflection hyperplanes in $\aff$. The group $\waff$ is a semi-direct product $\waff \simeq \wlin \rtimes T$ where $T$ is a subgroup of the translation group of $\aff$. We say that $(\aff, \waff)$ is \emph{discrete} if $T$ is a discrete subgroup, otherwise we say that $(\aff, \waff)$ is dense. If $(\aff, \waff)$ is discrete, then $\aff^{o}_{\cH} := \aff - \cup_{H \in \cH} H$ is an open set. Its connected components are called \emph{alcoves} or chambers. The closures of alcoves tile the affine space. In the discrete case, the structure of a building can be described completely in terms of these alcoves. However, when $T$ is dense, the set $\aff^{o}_{\cH}$ is very badly behaved; in order to restore the notion of an alcove to its rightful place, we need to introduce the notion of a filter:

\begin{definition}\label{filter}
Let $A$ be a set. A \emph{filter} $\sigma$ on a $A$ is a collection of subsets of $A$ satisfying the following conditions:

\begin{enumerate}
\item If $P$ is in $\sigma$ and $P \subseteq Q$ then $Q$ is in $\sigma$ 
\item If $P$ and $Q$ are in $\sigma$, then so is $P \cap Q$. 

\end{enumerate}

For any subset $Z \subseteq A$, the collection $\sigma_{Z} = \{ Y \subset A | Z \subseteq Y \}$ is a filter. A filter is \emph{principal} if it is of the form $\sigma_{\{a\}}$ for some $a \in A$.  
\end{definition}

\begin{example}\label{nbd-filter}
Let $X$ be a topological space, and $x \in X$. Then the collection $N_{x}$ of subsets of $X$ which contain a neighborhood of $x$ is a filter on $X$, which is different from $\sigma_{ \{ x \} }$ in general. It is called the neighborhood filter of $\{ x \}$. 
\end{example}

\begin{example}\label{Frechet-filter}
Ley $X$ be a set, and let $F := \{ A \subset X | \ X - A$ is finite $\}$. Then $F$ is a filter on $X$, called the Frech\'{e}t filter.  Note that $F$ is not contained in any principal filter. 

\end{example}

The set of filters on a set $A$ forms a sub-poset of the opposite of the  lattice of subsets of the power set $2^{A}$ ordered by inclusion. One writes $\sigma \preceq \sigma'$ and says that $\sigma$ is contained in $\sigma'$, if for all $Z \in \sigma'$ we have $Z \in \sigma$. We will often implicitly identify a subset $Z \subset A$ with the filter $\sigma_{Z}$; in particular, we will say that a filter $\sigma$ is contained in a set $Z$ if $\sigma \preceq \sigma_{Z}$, i.e., if $Z \in \sigma$. A subset $Z \subseteq A$ is the union of a family of filters $\{ \sigma_{\alpha} \}_{\alpha}$ if each $\sigma_{\alpha}$ is contained in $Z$, and for each $x \in Z$ there exists $\alpha$ such that $\sigma_{ \{ x \} } \preceq \sigma_{\alpha} \preceq \sigma_{Z}$. Thus, for example, a subset of a topological space $X$ is open if and only if it is the union of a family of neighborhood filters (see Example \ref{nbd-filter}). The closure $\overline{\sigma}$ of a filter $\sigma$ in a topological space $X$ is the collection of subsets of $X$ that contains the closure of a set in $\sigma$.

\begin{definition}\label{apartment}
Let $x$ be a point in a Coxeter complex $\aff$ modeled on a Eucledian space $\euc$, and let $F$ be a vectorial facet in $\euc$ (Definition \ref{vector-facets}). The \emph{facet} $\sigma_{F,x}$ associated to $(x,F)$ is defined by the following condition: $Z$ belongs to $\sigma_{F,x}$ iff it contains a finite intersection of open half apartments and walls containing $U \cap (x + F)$ for some open neighborhood $U$ of $x$ in $\aff$. Let $\cF_{\aff,W} = \{ \sigma_{F,x} | x \in \aff$, and $F$ is a vectorial facet$\}$. The pair $(\aff, \cF_{(\aff,W)})$ consisting of the metric space $\aff$ together with the collection of filters $\cF_{(\aff, W)}$ is called the \emph{Euclidean apartment} associated to $(\aff, W)$. An abstract metric space $A$ endowed with a collection of filters $\cF$ is called a Euclidean apartment if it is isomorphic to $(\aff, \cF_{(\aff,W)})$ for some affine Coxeter complex $(\aff, \waff)$.

\end{definition}

This concludes our discussion of the terminilogy involved in the definition of a building. Having introduced the objects that we will study, we now introduce the maps between them.

\begin{definition}\label{chamber-system}
A \emph{(generalized) chamber system} is a set $X$ equipped with a family $\cF$ of filters on $X$. Let $(X, \cF)$ and $(X', \cF')$ be generalized chamber systems. A map of sets $f: X\rightarrow X'$ is a morphism of chambers systems if for each filter $\sigma \in \cF$ we have that $f_{\ast}(\sigma) \in \cF'$. Here $f_{\ast}(\sigma) := \{ Q' \subset X' |\  f(Q) \subset Q'$ for some $Q \in \sigma \} $

\end{definition}

\begin{definition}
A \emph{pre-building} $(\B, \cF, \cC)$ is a generalized chamber system $(\B, \cF)$ equipped with a collection $\cC$ of sub-chamber systems called cubicles, each cubicle $C \in \cC$ being equipped with a metric $d_C$, satisfying the following condition: each cubicle $(C,d_C)$ is isomorphic to an enclosure (see Definition \ref{half-apartment}) in  an apartment $\aff$ (as a chamber system and as a metric space). 

\end{definition}

\begin{definition}\label{building-maps}
Let $(\B, \cF, \cA)$ and $(\B, \cF', \cA')$ be pre-buildings. A morphism of generalized chamber systems $f: \B \rightarrow \B'$ is an \emph{isometry of pre-buildings} if it restricts to a distance preserving map $f_{|A}: C \rightarrow \B'$ for every cubicle $C$.

Suppose now that $(\B, \cF, \cA)$ and $(\B', \cF', \cA')$ are buildings. Then $f$ is 

\begin{enumerate}
\item An \emph{isometry of buildings} or \emph{strong morphism of buildings} if it is an isometry of pre- buildings. 

\item A \emph{folding map} or \emph{weak morphism of buildings} if it has the following property: for every apartment $A$ in $\B$ there exists a locally finite collection of hyperplanes $\cH$ such that $f$ restricts to an isometry on the closure of each connected component of $A - \cup_{H \in \cH} H$. 

\end{enumerate}
\end{definition}

Finally, we close this section by collecting together some properties of affine buildings that will be very useful in later sections:

\begin{proposition}\label{properties}
Let $\B$ be an affine building with Weyl group $W$. Then the following hold:

\begin{itemize}

\item[(CO)] Let $S$ and $S'$ be opposite sectors based at a common vertex. Then there is a unique apartment $A$ such that $S \cup S' \subset \B$. 

\item[(SC)] Let $S$ be a sector in $\B$, and let $A_1$ be an apartment. Suppose that $A \cap S$ is a panel in $A_1$. Let $H$ be a wall in $A_1$ that contains  $P$. Then there exist apartments $A_2 \neq A_3$ such that $A_1 \cap A_j$ is a half-apartment  and $H \cup S \subset A_j$ for $j = 2,3$. 

\item[(CG)] Let $x$ be a vertex in an affine building $\B$, and let $S$ and $S'$ be sectors based at $x$. Then there is an apartment $A$ in $\B$ containing $S$ and the germ $\Delta_xS'$. 
\end{itemize}
\end{proposition}

\begin{proof}
For the proof, we refer the reader to \cite{Bennett-Schwer}, \cite{Parreau-thesis}. 
\end{proof}

For some of the arguments of this paper, it will be necessary to restrict ourselves to buildings with a ``complete system of apartments''. As the following definition and the remark following it show, this is not a major restriction at all

\begin{definition}\label{complete-system}
Let $(\B, \cF, \cA)$ be an affine building. We will say that $(\B, \cF, \cA)$ is a building \emph{with a complete system of apartments} if for any other system of apartments $\cA'$ with $\cA \subset \cA'$ we have $\cA = \cA'$.
\end{definition}

\begin{remark}
It is easy to show (see \cite{Parreau-thesis, Rousseau}) that any system of apartments $\cA$ is contained in a unique maximal one $\overline{\cA}$. Thus, one does not lose much by restricting to buildings that have a complete system of apartments.  
Let $A \subset \cB$ be a subset of a building $(\cB, \cF, \cA)$ with a complete system of apartments. Suppose that $A$ is isometric to the standard apartment, and every bounded subset of $A$ is contained in an apartment. Then $A \in \cA$. In fact, this gives a characterization of buildings with complete systems of apartments. One can give a stronger characterization: a building has a complete system of apartments if and only if every geodesic is contained in a single apartment \cite{Parreau-thesis}. 
\end{remark}


\subsection{Asymptotic cones of symmetric spaces are buildings}\label{symmetric-spaces}

A locally symmetric space is a Riemannian manifold $M$ whose Riemannian curvature tensor is covariantly constant. This is equivalent to requiring that for any $p \in M$, there is a self-isometry $s_p$ of a neighborhood of $p$ that fixes $p$, and whose derivative at $p$ is the negative of the identity. We say that $M$ is a symmetric space if $s_p$ extends to a global isometry on $M$. A reference for the theory of symmetric spaces is \cite{Helgason}; a short survey tailored to the needs of this paper is in \cite{Maubon}.

The connected component $G$ of the identity in the isometry group of a Riemannian manifold $M$ is a Lie group. Furthermore, $M$ is a homogenous space for $G$; if $K$ denotes the isotropy group of a point $p \in M$, then we have a diffeomorphism $G/K \simeq M$. 

The action of $s_p$ induces an involution $\sigma$ of $G$ given by $\sigma(g) = s_p g s_p$, and an involution $\mathrm{Ad}(s)$ of the Lie algebra $\g$ of $G$. The decomposition of $\g \simeq \mathfrak{t}+ \mathfrak{p}$ into $+1$ and $-1$ eigenspaces for the action of $\mathrm{Ad}(s)$ is called the Cartan decomposition of $\g$. 

If the restriction of the Killing form to the $-1$ eigenspace $\mathfrak{p}$ is negative definite, then the symmetric space is of non-compact type. Symmetric spaces of non-compact type have negative sectional curvature, and, like affine buildings, are CAT(0) spaces. As we are about to see, the relationship between affine buildings and symmetric spaces runs deeper than that. But first we pause to give an example:

\begin{example}\label{slrsur}
The symmetric space $M = \SL_r\cc / SU_r$ is a symmetric space of non-compact type. It can be identified with the space of hermitian metrics on the vector space $\cc^r$. The corresponding Cartan decomposition is $\mathfrak{sl}_r \cc \simeq \mathfrak{su}_r + \mathfrak{h}$, where $\mathfrak{h}$ is the set of $r \times r$ Hermitian matrices. The exponential map gives a diffeomorphism $\exp: \mathfrak{h} \rightarrow X$. 

\end{example}

Recall that a Riemannian submanifold $N \subset M$ is \emph{totally geodesic} if it is locally convex as a metric subspace. This is equivalent to requiring that any geodesic starting in $N$ and with initial direction in $T_N$ stays in $N$. It is also equivalent to requiring that the Levi-Civita connection on $M$ restricts to the Levi-Civita connection on $N$. 

\begin{definition}
Let $M$ be a symmetric space of non-compact type. A totally geodesic submanifold $N \subset M$ is called a $k$-\emph{flat} if it is isometric to $\mathbb{R}^k$ with its standard Riemannian metric. An \emph{apartment} in $M$ is a flat that is maximal with respect to inclusion.  
\end{definition}

\begin{remark}
Under the exponential map $\exp: \mathfrak{h} \rightarrow M$ of Example \ref{slrsur}, apartments $A \subset M$ correspond to maximal abelian subalgebras $\mathfrak{a} \subset \mathfrak{h}$

\end{remark}

A symmetric space of non-compact type, endowed with its set of apartments, enjoys a weak form of the axioms of an affine building. Compare the following with Definition \ref{building}:

\begin{proposition}[\cite{Maubon}]\label{sym-spaces-apts}
Let $M$ be a Riemannian symmetric space. Then the set $\cA$ of apartments in $M$ satisfies the following axioms:
\begin{enumerate}
\item[(b1)] For any two points $p$ and $q$ in $M$, there is an apartment containing $p$ and $q$.
\item[(b2)] For any two apartments $A$ and $A'$, their intersection is a closed convex set of both. Furthermore, there is an isometry $A \rightarrow A'$ fixing $A \cap A'$. 
\end{enumerate}
\end{proposition}

This similarity between the behavior of the apartments in symmetric spaces and buildings suggests a deep relationship between the two. This is indeed the case, as is made manifest by the following beautiful theorem of Kleiner and Leeb:

\begin{theorem}[Kleiner-Leeb, \cite{Kleiner-Leeb}]\label{cone-is-building}
Fix an ultrafilter $\omega$ on $\mathbb{N}$. Let $M$ be a non-empty symmetric space of non-compact type. Then for any sequence $p = \{p_n\}_{n \in \mathbb{N}}$ of base-points in $M$, and any family of scale factors $\mu = \{\mu_{n}\}_{n \in \mathbb{N}}$, the asymptotic cone $\Cone_{\omega}(M, p, \mu)$ is a thick affine building with a complete system of apartments. Furthermore, if the Coxeter complex associated to $M$ is $(\aff, \waff)$, then $\Cone_{\omega}(M, p, \mu)$ is modelled on $(\aff, \waff)$. 
\end{theorem}

\begin{remark}
The reader who consults \cite{Kleiner-Leeb} will not find the term ``complete system of apartments'' there. The reason is that Kleiner and Leeb use a different axiomatization of buildings from the one given in Definition \ref{building}. Parreau (\cite{Parreau-thesis}) has shown that a building in the sense of Kleiner and Leeb is precisely the same thing as a building in the sense of Definition \ref{building} with a complete set of apartments.
\end{remark}

In the remainder of this subsection, we will briefly outline the notion of asymptotic cones. Our treatment will be superficial, the primary purpose being to fix notation. For a detailed discussion of this circle of ideas, the reader is referred to \cite{Gromov-metric-structures-revised, Gromov-metric-structures}.  

Loosely speaking, the asymptotic cone of a metric space $(X,d)$ is what one obtains by ``looking at $X$ from infinity'', ignoring all of its fine-grained local structure. For a convex subset $X \subset \mathbb{R}^n$ with the induced metric, it is fairly straighforward to make this idea precise. Let $p \in X$, and consider the family of subsets $X_t := (1/t)X = \{q \in \mathbb{R}^n |\  t(q -p) \in X \}$ for $t \in \mathbb{R}$. For $t_1 \leq t_2$ we have $X_{t_2} \subset X_{t_1}$, so this is a nested family of subsets of $\mathbb{R}^n$. The asymptotic cone of $X$ with respect to $p$ is defined to be the intersection $\Cone(X,p) = \cap_{t} X_{t}$. 

\begin{example}
Let $f: \mathbb{R} \rightarrow \mathbb{R}$ be the function given by $f(x) = \exp(-1/x^2)$ for $x \leq 0$, and $f(x) = 0$ for $x \geq 0$. Let $X := \{ (x,y) \in \mathbb{R}^2 |\  y \geq f(x) \}$, and let $p = (0,0)$. Then the asymptotic cone $\Cone(X,p)$ is the right upper quadrant $\{ (x,y) |\  x \geq 0, y \geq 0 \}$. Note that the tangent cone at $p$ is the entire upper half-plane. 
\end{example}

When $(X,d)$ is not a convex subset of $\mathbb{R}^n$, the sequence of metric spaces $(X, (1/n)d)$, $n \in \mathbb{N}$, cannot in general be realized as a nested sequence of subsets of some metric space. However, it is still possible, roughly speaking, to define the asymptotic cone as certain limit of this sequence in the ``metric space of all of metric spaces''. The relevant notion of distance between metric spaces is the Gromov-Hausdorff distance. Before considering the precise definition, we give one  more example:

\begin{example}
Consider the lattice $\mathbb{Z}^{2}$ with generators given by the standard basis vectors, and viewed as a metric space with the word metric. That is, $d((m,n),(m',n')) = |m - m'| + |n - n'|$. It can be viewed as a subspace of the metric space $(\mathbb{R}^2,d_{\Cone})$ with the metric $d_{\Cone}((x,y),(x',y')) = |x - x'| + |y - y'|$. Then the asymptotic cone of $(\mathbb{Z}^{2}, d)$ with respect to any base-point is  $(\mathbb{R}^2,d_{\Cone})$. Intuitively, as we rescale the metric on $\mathbb{Z}^2$, the points on the lattice get closer together in the plane, until eventually they fill out the entire plane.  

\end{example}

In order to define the asymptotic cone in general, one needs to take Gromov-Hausdorff limits of sequences of metric spaces that do not always converge. The natural thing to do is to pass to a convergent subsequence. The reader might profitably think of an ultrafilter as a black-box that picks out such a convergent subsequence for us. Serendepitously, we have already introduced the notin of a filter in the previous subsection, so we can define ultrafilters with minimal work:

\begin{definition}\label{ultrafilter}
( \emph{Ultrafilters and ultralimits}). An \emph{ultrafilter} $\sigma$ on a set $A$ is a filter on $A$ that is maximal with respect to inclusion. Let $\omega$ be a \emph{non-principal ultrafilter} on a set $A$, i.e., an ultrafilter that is not principal (see Definition \ref{filter} for the notion of a principal filter). Then we say that a  family of points $\{x_a\}_{a \in A}$ in a topological space $X$ has $\omega$-limit $x$, and write $\lim_{\omega} x_{a} = x$ if for each neighborhood $U$ of $x$, the set $\{ a \in A | \ x_{a} \in U \}$ belongs to $\omega$. 
\end{definition}

The typical examples of interest will be $A = \mathbb{N}$, and $A = \mathbb{R}$. By Zorn's lemma, there exist ultrafilters on any set $A$. Furthermore, there exists a maximal filter containing the Frech\'{e}t filter on any set (Example \ref{Frechet-filter}) --- such a filter is necessarily a non-principal ultrafilter. We summarize some of the basic facts that we will need in the form of a proposition, and refer the reader to \cite{Comfort, Leinster} for a more thorough discussion of ultrafilters:

\begin{proposition}\label{ultralimit-properties}
Let $\omega$ be a non-principal ultrafilter on $\mathbb{N}$. Let $(X,d)$ be a metric space, and let $\{ x_n \}_{n \in \mathbb{N}}$ be a sequence in $X$.  Then we have the following:
\newcounter{saveenum}
\begin{enumerate}
\item If $\{ x_n \}_{n \in \mathbb{N}}$ is a bounded sequence in $X$, then it has an $\omega$-limit in $X$.

\item If $\{x_n \}$ is a convergent sequence (in the ordinary sense)  then $\lim_{\omega} x_n$ exists, and we have $\lim_{\omega} x_n = \lim_{n \to \infty} x_n$. 

\setcounter{saveenum}{\value{enumi}}
\end{enumerate} 

Let $\tilde{\omega}$ be a non-principal ultrafilter on $\mathbb{R}$ and let $f: \mathbb{R} \rightarrow X$ be a bounded map. Then

\begin{enumerate}
\setcounter{enumi}{\value{saveenum}}
\item  $\lim_{\omega} f \leq \limsup_{t \to \infty} f(t)$.
\item If $\lim_{t \to \infty} f(t)$ exists, then we have $\lim_{t \to \infty} f(t) = \lim_{\omega} f = \limsup_{t \to \infty} f(t)$.

\end{enumerate}
\end{proposition}

 A \emph{family of scale factors} $\{ \mu_{n}\}_{n \in \mathbb{N}}$ (resp. $\{\mu_t\}_{t \in \mathbb{R}}$) is a sequence (resp. family) of positive real numbers such that $\mu_{n} \to \infty$ as $n \to \infty$ (resp. $\mu_t \to \infty$ as $t \to \infty$). 

\begin{definition}\label{asymptotic-cone}
Let $(X,d)$ be a metric space, let  $\{ p_n \}$ be a sequence of points in $X$ and let  $\{ \mu_{n}\}_{n \in \mathbb{N}}$ be a family of scale factors. Fix a non-principal ultrafilter $\omega$ on $\mathbb{N}$. The asymptotic cone $\Cone_{\omega} := \Cone_{\omega}(X, \{p_n\}, \{ \mu_n \})$ is the metric space associated to the pre-metric space $\Cone'_{\omega}$:

\begin{itemize}
\item[-] the points of $\Cone'_{\omega}$ are sequences $\{ x_n \}$ in $X$ such that $(1/n)d(x_n, p_n)$ is bounded.

\item[-] The distance is given by:

\[
d_{\Cone_{\omega}}(\{x_n\}, \{y_n\}) = \lim_{\omega} \frac{1}{n} d(x_n, y_n)
\]

\end{itemize}

\end{definition}

Suppose that  $(X,d)$ is a symmetric space of non-compact type. Let $d_n$ denote the metric $(1/n)d$. Suppose that for each $n$ we have a totally geodesic embedding $f_n: \aff \rightarrow X$ sending $0$ to $p_n$. Here $\aff$ is an apartment whose rank equals the rank of the symmetric space. Then we have an induced map $[f]: \aff \rightarrow \Cone_{\omega}(X,p, \mu)$, $[f](a) = \{ f_n(a) \}$. These maps define the apartments in $\Cone_{\omega}$ in Theorem \ref{cone-is-building}.


\subsection{Refined distance functions}\label{vector-distance-section}

Let $M = G/K$ be a symmetric space. If $M$ is of rank $1$, then the length  is a complete invariant of a geodesic segment, upto the action of $G$. However, in the case of higher rank symmetric spaces, a more refined invariant is needed to distinguish geodesic segments. 

Let $(\aff, \waff)$ be a Coxeter complex, with $\aff$ modelled on a Euclidean vector space $\euc$. Let $\wsph$ be the spherical part of $\waff$. The subtraction map $\aff \times \aff \rightarrow \euc$ descends to a natural map:

\[
\aff \times \aff \rightarrow (\aff \times \aff)/\waff \rightarrow \euc / \wsph
\]

$\euc/ \wsph$ can be identified with a fundamental domain for the action of $\wsph$ on $\euc$: so we get a distance function $\vec{d}: \aff \times \aff \rightarrow \chamber$ with values in the fundamental Weyl chamber. 

\begin{definition}\label{vector-distance}
Let $X$ be a symmetric space or an affine building, and let $x$, $y$ be points in $X$. In light of the definition of buildings (Definition \ref{building}) and the properties of apartments in symmetric spaces (Proposition \ref{sym-spaces-apts}), we know that there exists an apartment $A \subset X$ such that $x$, $y \in A$. Define the \emph{vector valued distance} $\vec{d}(x,y)$ to be the vector valued distance computed in this apartment $A$. It follows from Def \ref{building} and Prop \ref{sym-spaces-apts} that this distance is independent of the apartment it is computed in. 

\end{definition}

When $M = G/K$ is a symmetric space, the vector valued distance has a familiar algebraic interpretation: it is the Cartan projection. Recall that there is the Cartan decomposition $G = K A^{+} K$. The Cartan projection is just the natural map $\log: A^{+} \rightarrow \chamber$ (see e.g. \cite{Helgason}). There is a special case of great interest to us in this paper, where this has an even more elementary interpretation:

\begin{example}
Let $M = \SL_r\cc/SU_r$. Then the Cartan decomposition is just the singular value decomposition of a complex matrix: every $T \in \SL_r\cc$ can be written $T = U D V$, where $U$ and $V$ are unitary, and $D$ is diagonal with real entries. If $A, B$ are two elements of $M$ and $A = TB$, then $\vec{d}(A,B)$ is the vector of logarithms of the diagonal entries of $D$. 

\end{example}

The following proposition says that the vector distance is well-behaved under passage to the asymptotic cone. It will play an important role in Section \ref{WKB}. 

\begin{proposition}[Parreau, \cite{Parreau-compactification}]\label{limit-vector-distance}
Let $M$ be a Riemannian symmetric space of non-compact type, and let $\vec{d}$ denote its vector distance function. Let $\Cone_{\omega}$ denote the affine building obtained from $M$ by passing to the asymptotic cone with respect to some family of base-points and scale factors, and an ultrafilter $\omega$ on $\mathbb{N}$. Let $[x_n]$ and $[y_n]$ be two points in $\Cone_{\omega}$, represented by sequences $\{x_n\}$ and $\{y_n\}$ in $M$. Then we have

\[
\vec{d}_{\Cone_{\omega}}([x_n],[y_n])  = \lim_{\omega} \vec{d}(x_n, y_n)
\]

\end{proposition}


\section{Spectral Covers and Buildings}\label{Spectral-Covers-and-Buildings}

The asymptotics of the differential equations studied in this paper are controlled by a multi-valued holomorphic 1-form associated to the differential equation. The Riemann surface of this 1-form is called the spectral curve, and defines a ramified cover (the spectral cover) of the space on which the differential equation is defined.

\subsection{The universal building}\label{universal-building-section}

The purpose of this subsection is to (i) recall the construction of a spectral cover starting from a harmonic map to a building, and (ii) formulate the notion of a universal buildings associated to a spectral cover (Definition \ref{Bu-def}) and formulate a conjecture regarding its existence (Conjecture \ref{Bu-exists}). 

Let $\varphi$ be a regular semisimple endomorphism of a complex vector space $V$ of dimension $n$. One can alternatively describe $\varphi$ in terms of its spectral data:  a family $\{L_1,...,L_n\}$ of lines (the eigenlines of $\varphi$) in $V$, each line $L_i$ being decorated with a complex number $\lambda_i$ (the corresponding eigenvalue). The notion of a spectral cover (see e.g. \cite{Donagi-spectral, Carlos-Higgs-Local, BNR-spectral}) is obtained by allowing the endomorphism to vary in a family $\{\varphi_{x}\}_{x \in X}$ where $X$ is a complex manifold, and allowing the endomorphisms to take values in some ``coefficient object'' $K$. Much of what we say in this subsection makes sense for varieties of arbitrary dimension, arbitrary reductive groups $G$, and very general coefficient objects $K$.  For the sake of simplicity and clarity, we restrict our exposition to the case where $X$ is of dimension $1$, $G = \SL_r\cc$ and $K$ is a line bundle,  leaving the obvious generalizations to the reader. 

\begin{definition}\label{spectral-cover-Higgs}
Let $K$ be a holomorphic line bundle on a smooth complex algebraic curve $X$.  
\begin{enumerate}
\item A \emph{K-valued spectral cover} $\phi$ of $X$ is the data of a finite ramified cover $\pi: \Sigma \rightarrow X$ together with a morphism $i: \Sigma \hookrightarrow \tot(K)$ realizing $\Sigma$ as a closed subscheme of the total space $\tot(K)$.
\item A \emph{K-valued Higgs coherent sheaf} on $X$ is a coherent sheaf $\cE$ on $X$ together with a section $\varphi$ of the sheaf $\Shend{\cE} \otimes K$ satisfying the condition $\varphi \wedge \varphi = 0$.
\end{enumerate} 
When $K$ is not specified, it is to be understood that $K = \omega_{X}$, the holomorphic cotangent bundle of $X$. A spectral cover $\phi$ is \emph{smooth} if $\Sigma$ is a smooth. We  will denote by $p: \tot(T^{\vee}_{X}) \rightarrow X$ the natural projection, and by $\lambda$ the restriction to $\Sigma$ of the tautological (Liouville) section $\alpha \in H^{0}(\tot(T^{\vee}_{X}), p^{\ast}(T^{\vee}_{X}))$. 
\end{definition}

A spectral cover $\phi$ can be thought of as a ``multi-valued 1-form'' on $X$. We may suggestively write ``$\phi = \{\phi_1,...,\phi_r\}$'', where $r$ is the degree of $\pi: \Sigma \rightarrow X$, to emphasize this point of view. Just as one associates to a linear operator on a vector space its eigenvalue spectrum, one can associate to a Higgs field $\varphi$ a multi-valued 1-form which plays the role of the eigenvalue spectrum. More, precisely to a $K$-valued Higgs bundle $(\cE, \varphi)$, one can associate its characteristic polynomial $\Char \varphi := \det(s \id  - \varphi) = \lambda^{r} + a_1 s^{r-1} + ...+a_r$. The coefficient $a_k$ of this polynomial is a section of the line bundle $K^{\otimes k}$. Note that if $(\cE, \varphi)$ is an $\SL_r$-Higgs bundle, then the sum the eigenforms is zero: $a_1 = 0$. The zeroes of the characteristic polynomial define a multi-valued homolorphic 1-form $\phi= \{ \phi_1,..., \phi_r \}$, which is the \emph{spectral cover associated to the Higgs field} $\varphi$. More precisely, we have the following:

\begin{definition}\label{associated-spectral-cover}
Let $(\cE, \varphi)$ be a $K$-valued rank $r$ Higgs bundle on a smooth complex algebraic curve $X$. Let $\lambda \in H^0( \tot(K), \pi^{\ast} K)$ be the tautological section.
\begin{enumerate}

\item The \emph{characteristic polynomial} of $(\cE, \varphi)$ is the section $\Char \varphi \in H^0( \tot(K), \pi^{\ast} K^{\otimes r})$ defined by the following condition: for any local section $s$ of $\pi^{\ast}K^{\vee}$, we have $(\Char \varphi, s^{\otimes r} ) = \det( \id_{\pi^{\ast} \cE} \otimes s - \pi^{\ast}\varphi \otimes s)$ as functions on $\tot(K)$. Here $(\relbar, \relbar)$ is the natural pairing between $K$ and $K^{\vee}$. 

\item The \emph{spectral cover associated to the Higgs field} $(\cE, \varphi)$ is the subscheme of $\tot(K)$ defined as the zero locus $Z( \Char \varphi )$ of the characteristic polynomial. 

\end{enumerate}

\end{definition}

Throughout the rest of this section, we let $G$ be the semisimple complex algebraic group $\mathrm{SL}_r\cc$, and let $W$ denote its Weyl group, which is the symmetric group on $r$-letters. Fix a Cartan subalgebra $\cartan \subset \g = \sll_r$. We will restrict ourselves to a discussion of $G$-Higgs bundles. We have seen that to a $G$-Higgs bundle one can associate a point in the vector space $\oplus_{i=2}^r H^0(X, K^{\otimes i})$; this vector space is called the \emph{Hitchin base} and the map is called \emph{Hitchin map}. 

The adjoint action of $G$ on $\g$ restricts to an action of $W$ on $\cartan$. According to Chevalley's theorem, $\cc[\cartan]^W$ is a polynomial algebra $\cc[\sigma_2,...,\sigma_r]$ on $r-1$ generators, and is naturally isomorphic to $\cc[\g]^G$. This is true for any semisimple lie group of rank $r -1 $; in our case, $G = \SL_r$, and this is just Newton's theorem on symmetric functions, and the polynomials $\sigma_i$ are the symmetric polynomials. More precisely if we identify $\cartan$ with the subset $\{\sum_{i=1}^r x_i = 0\}$ in $\cc^r$ with coordinates $(x_1,...,x_r)$, then $\sigma_k = \sum x_{i_1}...x_{i_k}$. 

The $\sigma_i$'s give a well-defined map of spaces over $X$,
$\sigma := (\sigma_2,...,\sigma_r): \tot(K \otimes \cartan) \rightarrow \tot (\oplus_{i=2}^{r} K^{\otimes i})$, realizing the target as the quotent of the total space of the vector bundle $K \otimes \cartan$ by the action of the Weyl group. This immediately leads to the following

\begin{definition}\label{cameral}
Let $X$ be a smooth complex algebraic curve, $K$ a line bundle on $X$, and let $\phi \in \oplus_{i=2}^{r}H^0(X, K^{\otimes i})$ be a point in the corresponding Hitchin base. The \emph{cameral cover} associated to $\phi$ is the cover $\pi_{\phi}: \Sigma_{\phi} \rightarrow X$ defined by the following pullback square:

\[
\xymatrix{
\Sigma_{\phi} \ar[r]^{\phii} \ar[d]_{\pi_{\phi}} & \tot(K \otimes \cartan) \ar[d]^{\sigma} \\
X \ar[r]^(0.3){\phi} & \tot( \oplus_{i=2}^r K^{\otimes i}) \\
}
\] 

The \emph{tautological section} on the cameral cover is the section $\phii \in H^{0}(\Sigma_{\phi}, \pi_{\phi}^{\ast}(K \otimes \cartan))$ defined by the commutative diagram above. If we fix a linear coordinate system $(x_1,...,x_r)$ on $\cartan$ as above, then we may identify $\bar{\phi}$ with a sequence of $r$ differential forms $\phii = (\phii_1,..., \phii_r)$ such that $\sum_{i=1}^r \phii_i = 0$. 
\end{definition}

\begin{remark}
The cameral cover $\Sigma_{\phi}$ is generically a $W$-Galois cover of $X$: its function field is a $W$-Galois extension of the function field of $X$. If the point $\phi$ in the Hitchin base is associated to a spectral cover $\Sigma$, then the function field of $\Sigma_{\phi}$ is the Galois closure of the function field of $\Sigma$ (as extensions of the function field of $X$). 
\end{remark}

We have seen that to any Higgs bundle $(\cE, \varphi)$ on $X$ one can associate a point $\phi$ in the Hitchin base which parametrizes cameral covers of $X$. Now we will describe another mathematical object that produces a point $\phi$ in the Hitchin base --- namely, a $\pi_{1}(X,x)$ equivariant harmonic map $h: \tilde{X} \rightarrow \B$ from the universal cover of $X$ to an affine building. This construction was first described in \cite{Ludmil-thesis}.

\begin{definition}\label{phi-map}
Let $\B$ be a rank $r-1$ affine building with Weyl group $W_{\mathrm{aff}} \simeq W \ltimes T$ where $T$ is a subgroup of the group of the group of translations of $\mathbb{R}^r$, and let $\{f: \aff \rightarrow \B\}_{f \in\cA}$ be an atlas for $\B$. A differential $k$-form $\eta$ on an open subset $U \subset \B$ is a collection $\{\eta_f\}_{f \in \cA}$ of differential $k$-forms on $f^{-1}(U) \subset \aff \simeq \mathbb{R}^{r-1}$ such that $(g^{-1} \circ f)^{\ast} \eta_{g} = \eta_{f}$ on $f^{-1}(g(\aff) \cap U)$. Denote by $T^{\vee}_{\B}$ the sheaf of $\mathbb{R}$-vector spaces on $\B$ whose sections over an open set $U$ are the differential $1$-forms, and by $T^{\vee}_{\B, \cc}$ its complexification. 
\end{definition}

Let $\{x_1,...x_r\}$ be linear coordinate functions on $\aff$ whose zero loci give the reflection hyperplanes for the action of the Weyl group $W$. Let $dx_1,...,dx_r$ denote the differentials of the coordinate functions $x_i$, viewed as sections of the complexified tangent bundle of $\aff$. As above, let $(\sigma_{2},...,\sigma_{r})$ be generators for the polynomial algbera $\cc[\cartan]^{W}$ (the symmetric polynomials). Then, for each $2 \leq k \leq r$, $\sigma_k(dx_1,...,dx_r)$ is invariant under the action of the affine Weyl group, and therefore the ``local differentials'' $\{\sigma_k(dx_1,...,dx_r) \}_{f \in \cA}$ define a section $\xi_k \in H^{0}(\B, \Sym^{k}(T^{\vee}_{\B, \cc}))$. 

Now let $X$ be a Riemannian manifold, and suppose that we are given a map $h: X \rightarrow \B$. Recall that a point $x \in X$ is $h$-\emph{regular} if there is a neighborhood $U$ of $x$ in $X$ such that $h(U)$ is contained in a single apartment $A$. Let $X_{\mathrm{reg}} \subset X$ denote the set of $h$-regular points. Then for any complex differential form $\eta$ on $\B$, there is a well-defined pullback $h^{\ast} \eta$ which is a section of the complexified cotangent bundle of $X$. To define the pullback, let $\{U_{\alpha}\}_{\alpha}$ be an open cover of $X_{\mathrm{reg}}$, such that for every $\alpha$, there exists a chart $f_{\alpha}: \aff \rightarrow \B$ with $h(U) \subset f(\aff)$. Then $h^{\ast}\eta$ is uniquely determined by the requirement that $(h^{\ast}\eta)_{|U_{\alpha}} = (f_{\alpha}^{-1} \circ h)^{\ast}(\eta_{f_{\alpha}})$ for all $\alpha$. 

\begin{lemma}
Let $X$ be a smooth complex algebraic curve, and suppose that $h: X \rightarrow \B$ is a \emph{harmonic} map to a building. Let $\xi_{k}$ be the harmonic symmetric tensor on $\B$ given locally by $\xi_{k} = \sigma_{k}(dx_1,...,dx_r)$ (see the previous paragraph), where $\sigma_1,...,\sigma_r$ are the standard symmetric polynomials. Then there is a \emph{unique} holomorphic section $\phi_k \in H^{0}(X, \omega_{X}^{\otimes k})$ that restricts to $h^{\ast}(\xi_k)$ on $X_{\mathrm{reg}}$. 
\end{lemma}

\begin{proof}
We will just sketch the idea of the proof; for the details, the reader is referred to \cite{Ludmil-thesis}. Since $h$ is a harmonic map, the complexified pullbacks $h^{\ast}(dx_i)$ are locally defined harmonic 1-forms on $X_{\mathrm reg}$ whose $(1,0)$ parts are holomorphic. Thus, $h^{\ast}(\xi_k)$ is a holomorphic sections of $\Sym^{k}(\Omega^{1}_{X})$. Since the singularities of a harmonic map are in codimension 2, these holomorphic sections extend to all of $X$. 
\end{proof}

\begin{definition}\label{phi-map}
Let $X$ be a smooth complex algebraic curve, and let $\phi = (\phi_2,...,\phi_r) \in \oplus_{i=2}^r H^{0}(X,\omega_{X}^{\otimes i})$ be a point in the $\mathrm{SL}_r\cc$ Hitchin base. Let $\B$ be an affine building with Weyl group $W_{\aff} \simeq W \ltimes T$ where $W$ is the Weyl group of $\mathrm{SL}_r\cc$, and $T$ is a group of translations. A $\pi_{1}(X,x)$-equivariant harmonic map $h: \tilde{X} \rightarrow \B$ is a \emph{harmonic $\phi$-map} if $\pi^{\ast} \phi_k$ coincides with $h^{\ast}\xi_k$ on $X_{\mathrm{reg}}$ for all $k$. 
\end{definition}

Thus far, we have seen two constructions of cameral covers. The first is the cameral cover associated to a Higgs bundle. The second is the one just described, which associates a cameral cover to a harmonic map to a building. One of the main ideas of this paper is to pass from a (family of) Higgs bundles to a harmonic map taking values in a building by ``reversing'' the second construction. It is easy to see that there can be many harmonic maps with the same underlying cameral cover;  we might hope to associate a harmonic map to a cameral cover by picking out one that is minimal in some sense. First we need a preliminary definition:

\begin{definition}\label{image-in-building}
Let $X$ be a Riemann surface, let $(\B, \cF, \cA)$ be a building, and $h: X \rightarrow \B$ be a map. The image of $h$, written $\mathrm{im}(h)$ is the pre-building $\cup_{A \in \cA} \hull(A \cap h(X))$ with the natural pre-building structure induced from $\B$.

\end{definition}

\begin{definition}\label{Bu-def}
Let the notation be as in Definition \ref{phi-map}. A map $\hu: \tilde{X} \rightarrow \Bu$ is a \emph{universal harmonic $\phi$-map} if it is a harmonic $\phi$-map and enjoys the follows property: given any building $\B$ with a complete system of apartments (and with the same vectorial Weyl group), and a harmonic $\phi$-map $h: \tilde{X} \rightarrow \B$ 
\begin{enumerate}

\item there exists a folding map of buildings (see Definition \ref{building-maps}) $\psi: \Bu \rightarrow \B$ such that the following diagram commutes:
\[
\xymatrix{
\tilde{X} \ar[r]^{\hu} \ar[dr]_{h} & \Bu \ar[d]^{\psi} \\
& \B \\
}
\]

\item If $\psi': \Bu \rightarrow \B$ is another morphism such that $\psi \circ \hu = h$, then $\psi_{|\mathrm{im}(h)} = \psi'_{|\mathrm{im}(h)}$. 

\end{enumerate}
\end{definition}

\begin{remark}\label{Bu-for-trees}
When $r=2$, points in the Hitchin base are given by quadratic differentials $\phi$ on $X$. The leaf space of the induced foliation on $\tilde{X}$ is an $\mathbb{R}$-tree $T^{\phi}$. It is easy to see that in this case the natural quotient map $\hu: \tilde{X} \rightarrow  T^{\phi}$ is the universal harmonic $\phi$-map. 
\end{remark}

\begin{remark}\label{Bu-in-simple-cases}
If the cameral cover is totally decomposed, then the universal building is just a single apartment. In the neighborhood of a simple branch point (see Figure \ref{SpecBranch}), the universal building is the product of a trivalent vertex with an apartment of one dimensional lower. 
\end{remark}

We can now formulate one of the main conjectures of this paper:

\begin{conjecture}\label{Bu-exists}
Let $X$ be a Riemann surface, and let $\phi$ be a smooth spectral cover of $X$. Then there exists a universal $\phi$-map $\hu: \tilde{X} \rightarrow \Bu$.
\end{conjecture}

The remark above shows that the conjecture is true in the rank 1 case (i.e., when $r = 2$). Some evidence supporting the conjecture when $r \geq 3$ will be presented in Section \ref{BNR-example}, where an example of universal building in the rank 2 case is constructed.


\subsection{Some properties of $\phi$-maps}\label{phi-map-properties}

In order to understand the behavior of $\phi$-maps, it is useful to develop criteria for when a given region $\Omega$ is carried into a single apartment by a $\phi$-map. In this subsection, we develop such a criterion (Proposition \ref{phi-maps-non-crit} and Corollary \ref{non-crit-regions-to-apts}) which will be used heavily in Section \ref{BNR-example}. The notion of a \emph{non-critical path} is the key ingredient in what follows. 

\begin{definition}\label{non-critical}
Let $X$ be a Riemann surface and let $\phi = (\phi_2,...,\phi_r) \in \oplus_{k=2}^{r}H^0(X, \omega_{X}^{\otimes k})$. Let $\gamma: [0,1] \rightarrow X$ be a smooth path, and let $a_k: [0,1] \rightarrow \mathbb{C}$ be the function $a_k(t) = (\gamma^{\ast}(\phi_k),\partial_t^{\odot k})$. Then we say that $\gamma$ is a $\phi$-\emph{non-critical} path if the real parts of the roots of the polynomial $\sum_k a_k(t) z^k$ are distinct for every $t \in [0,1]$. 
\end{definition}

\begin{remark}
Let $\pi: \Sigma \rightarrow X$ be the spectral cover defined by $\phi$. We can think of $\pi$ as a multi-valued differential form $\{\lambda_1,..., \lambda_r\}$ on $X$. Then asking that $\gamma$ be non-critical is equivalent to requiring the real $1$-forms $\re \gamma^{\ast} \lambda_1$,...,$\re \gamma^{\ast} \lambda_r$ to be to distinct. We may re-order them, and assume without loss of generaility that  $(\re \gamma^{\ast} \lambda_1, \partial_t) > ... > (\re \gamma^{\ast} \lambda_r, \partial_t)$. 
\end{remark}

\begin{lemma}\label{non-critical-opposite-sectors}
Let $X$ be a Riemann surface, and let $\phi \in \oplus_{k=2}^r H^0(X, \omega_{X}^{\otimes k})$. Let $h: X \rightarrow \B$ be a harmonic $\phi$-map, and let $X_{\mathrm{reg}}\subset X$ denote the locus where $h$ is regular. Let $\gamma$ be a non-critical path in $X_{\mathrm{reg}}$, and let $s \in (0,1)$. Then there exist $\epsilon > 0$, and sectors $S^+$ and $S^-$ in $\B$ based at $x := h(\gamma(s))$ such that 

\begin{enumerate}

\item The germs $\Delta_x S^+$ and $\Delta_x S^-$ are opposed.

\item  $h(\gamma((s-\epsilon, s]) \subset S^-$ and $h(\gamma([s,s+ \epsilon)) \subset S^+$. 

\end{enumerate}
\end{lemma}

\begin{proof}
Since $\gamma(s)$ is a regular point, there exists a neighborhood $U$ of $\gamma(s)$, a chart $f: \aff \rightarrow \B$ (where $\aff$ is the standard apartment on which $\B$ is modelled), and a commutative diagram

\[
\xymatrix{
U \ar[r]^{\tilde{h}} \ar[d] & \aff \ar[d]^{f} \\
X \ar[r]_{h} & \B \\
}
\]
with $\tilde{h}$ a differentiable map, and $f(a) = x$ for some $a \in \aff$.  Let $(x_1,...,x_r)$ be the standard coordinates on the apartment $\aff$, and let $\{ \lambda_1,..., \lambda_r \}$ be the roots of the characteristic polynomial defined by $\phi$. Since $h$ is a $\phi$-map, we may choose $f$ (precomposing with an element of the vectorial Weyl group if necessary) such that $\tilde{h}^{\ast}(dx_i) = \re \lambda_i$.  By continuity, there exists an $\epsilon > 0$ such that $\gamma(J) \subset U$ where $J := (s - \epsilon, s + \epsilon)$.

Let $\euc$ be the Euclidean space on which $\aff$ is modelled. For any $y \in \aff$, we can identify the tangent space $T_y \aff$ with $\euc$. Let $\cHlin$ be the set of reflection hyperplanes in $\euc$ for the vectorial part of the Weyl group of $\B$. Identifying $\euc$ with $\aff$ via the map $v \mapsto a + v$, we may think of the $x_i$ as coordinates on $\euc$. With respect to these coordinates, $\cH$ is the set of hyperplanes $\{ \re (x_i - x_j) = 0 \}$. 

The condition that  $\gamma_{|J}$ is $\phi$-non-critical is equivalent to the condition that, for all $t \in J$, $(\tilde{h} \circ \gamma)_{\ast} (\partial_t) \notin H$ for any $H \in \cHlin$. Since $J$ is connected and $\psi$ is continuous, this implies that $\psi(J)$ is contained in a single Weyl chamber in $\mathfrak{C}$ in $\euc$, i.e., in a single connected component of $\euc - \cup_{H \in \cHlin}H$. Let $\mathfrak{C}^{\mathrm{op}} = \{ v | -v \in \mathfrak{C} \}$ be the opposite chamber. We have the corresponding sectors based at $a$: $S^+  := a + \mathfrak{C}$ and $S^- = a + \mathfrak{C}^{\mathrm{op}}$. 

Let $f_i = x_i \circ \tilde{h} \circ \gamma: J \rightarrow \mathbb{R}$. Then the conclusion of the previous paragraph says that, after re-ordering coordinates if necessary, we may assume that $f_1'(t) > ...> f_r'(t)$, where $f_i'$ is the derivative of $f_i$. From the formula

\[
f_i(t) = x_i(a) + \int_s^t f_i'(t) dt
\]

we see that $(\tilde{h} \circ \gamma) (t) \in S^+$ (resp. $(\tilde{h} \circ \gamma (t)) \in S^-$) for all $t \in [s, s+\epsilon)$ (resp. $t \in (s - \epsilon,s]$). 
\end{proof}

In the course of proving Lemma \ref{non-critical-opposite-sectors} we actually proved the following statement:

\begin{lemma}
Let $X$, $\phi$ be as in Lemma \ref{non-critical-opposite-sectors}. Let $J \subset \mathbb{R}$ be an interval (open, closed or half-open), and let $\gamma: J \rightarrow X_{\mathrm{reg}}$ be a non-critical path. Suppose that there is an apartment $A \subset \B$ such that $\gamma(J) \subset A$.  Let $s \in J$, $x = h(\gamma(s))$, and let $J^- = \{ t \in J | \ t \leq s\}$, and $J^+ = \{ t \in J | \  t \geq s\}$. Then

\begin{enumerate}\label{non-crit-convexity}
\item There are a pair of opposite sectors $S^+$, $S^-$ in $A$ such that $\gamma(J^-) \subset S^-$, and $\gamma(J^+) \subset S^+$. 

\item For any $t_0 \in J^-$ and $t_1 \in J^+$ we have 

\[
\vecd(y,z) = \vecd(y,x) + \vecd(x, z)
\]

where $y = h(\gamma(t_0))$ and $z = h(\gamma(t_1))$. That is, $x$ is in the Finsler convex hull of $y$ and $z$. 
\end{enumerate}
\end{lemma}

\begin{proof}
Only the second statement needs proof. It follows from first statement. Let $\mathfrak{C}'$ be the Weyl chamber in $\euc$ that contains $z -x$. Since $S^+$ and $S^-$ are opposed, we have that $x - y \in \mathfrak{C}'$. It follows that $(z - y) = (z - x) + (x - y) \in \mathfrak{C}'$. Let $w$ be the element of the spherical Weyl group that carries $\mathfrak{C}'$ to the fundamental Weyl chamber $\mathfrak{C}$. Then, by definition of the vector distance, we have $\vecd(x,z) = w(z-x)$, $\vecd(y,x) = w(x-y)$ and $\vecd(y,z) = w(y - z)$. Since  $w(z - y) = w(z - x) + w(x - y)$, the claim follows.
\end{proof}

We can deduce a sort of partial converse to Lemma \ref{non-critical-opposite-sectors}.

\begin{lemma}\label{opposedness}
Let $X$, $\phi$ be as in Lemma \ref{non-critical-opposite-sectors}. Let $J \subset \mathbb{R}$ be an interval (open, closed or half-open), and let $\gamma: J \rightarrow X_{\mathrm{reg}}$ be a non-critical path. Let $s \in J$, and let $J^- = \{ t \in J | \ t \leq s\}$, and $J^+ = \{ t \in J | \  t \geq s\}$. Suppose that there exist sectors $S^-$ and $S^+$ based at $x := h(\gamma(s))$ such that $h(\gamma(J^-)) \subset S^-$ and $h((\gamma(J^+)) \subset S^+$. Then the germs $\Delta_x S^-$ and $\Delta_x S^+$ are opposed.
\end{lemma}

\begin{proof}
According to Proposition \ref{properties}, there exists an apartment $A \subset \B$ containing $S^+$ and the germ $\Delta_xS^-$. Let $S_A^-$ be the sector in $A$ whose germ at $x$ is $\Delta_xS^-$. Then, by definition of germs, $S_A^- \cap S^-$ is an open neighborhood of $x$ in $S^-$. It follows that there exists $\epsilon > 0$ such that $h(\gamma(s-\epsilon, s]) \subset S_A^-$. By Lemma \ref{non-crit-convexity}, $S_A^-$ and $S^+$ must be opposite sectors in $A$. Since $\Delta_xS_A^- = \Delta_xS^-$, it follows that $\Delta_xS^-$ and $\Delta_xS^+$ are opposed. 
\end{proof}

We can now prove the main proposition of this section:

\begin{proposition}\label{phi-maps-non-crit}
Let $X$ be a Riemann surface, and let $\phi \in \oplus_{k=2}^r H^0(X, \omega_{X}^{\otimes k})$. Let $h: X \rightarrow \B$ be a harmonic $\phi$-map, and let $X_{\mathrm{reg}}\subset X$ denote the locus where $h$ is regular. Let $I = [0,1]$, and let $\gamma: I \rightarrow X_{\mathrm{reg}}$ be a non-critical path. Then there exists an apartment $A \subset \B$ such that $\mathrm{im} (h \circ \gamma) \subset A$. 
\end{proposition}

\begin{proof}
Define a subset $K : = \{ t \in I |$ there exists an apartment $A \subset \B$ such that $h(\gamma([0,t])) \subset A \} \subset I$.  We must prove that $K = I$. First, note that $K \neq \emptyset$: there exists an apartment containing the point $h(\gamma(0))$, by the axiom (B2) for buildings, so we have $0 \in K$. Furthermore, note that if $t \in K$, and $s \leq t$, then $ s \in K$. We conclude that $K$ is a non-empty interval. Let $t_0 = \sup K$, and let $Q := \gamma(t_0)$. We must prove that $t_0 = 1$. We will do so by contradiction. So assume that $t_0 < 1$. 

Since $\gamma(I) \subset X_{\mathrm{reg}}$, there is a neighborhood $U_Q$ of $Q$ and an apartment $A_Q \subset \B$ such that $h(U_Q) \subset A_Q$. Passing to a smaller open set if necessary, we may assume that $U_Q \cap \gamma(I) = \gamma(J)$ for some interval $J$ containing $t_0$.

Since $t_0 = \sup K$, and $J$ is an open neighborhood of $t_0$, there exists $t_1 \in J \cap K$. Let $P = \gamma(t_1)$. By definition of $K$, there exists an apartment $A_P \subset \B$ such that $h(\gamma([0,t_1])) \subset A_P$. Applying Lemma \ref{non-crit-convexity}, we conclude that there is a sector $S_P^-$ in $A_P$ based at $x := h(P)$ such that $h(\gamma[0,t_1]) \subset S_P^-$. 

Since  $h(\gamma(J)) \subset A_Q$, we can apply Lemma \ref{non-crit-convexity} again to conclude that there exist opposed sectors $S^+_Q$ and $S^-_Q$ in $A_Q$ such that $h(\gamma(J^-)) \subset S_Q^-$ and $h(\gamma(J^+)) \subset S_Q^+$. Here $J^+ = \{t \in J | \ t \geq t_1 \}$ and $J^- = \{t \in J | \ t \leq t_1 \}$. Note that the sectors $S^+_Q$ and $S^-_Q$ are based at $x := h(P)$. 

From Lemma \ref{opposedness}, it follows that the germs $\Delta_xS_P^-$ and $\Delta_xS_Q^+$ are opposed. Property (CO) (Proposition \ref{properties}) of buildings states that there is a unique apartment containing a pair of opposed sectors in a building. Let $A$ be the unique apartment in $\B$ containing $S_P^-$ and $S_Q^+$. 

Since $J^+$ contains an open neighborhood of $t_0$ in $[0,1]$, and $t_0 < 1$,  there exists $t_2 > t_1$ such that $t_2 \in J^+$, and hence there exists $t_2 > t_0$ such that $h(\gamma([t_1, t_2]) \subset S_Q^+$. Since $h(\gamma([0,t_1])) \subset S_P^-$, and $h (\gamma ([t_1, t_2])) \subset S_Q^+$, we have that $h(\gamma([0,t_2]) \subset A$. Hence, $t_2 \in K$. But $t_2 > t_0 := \sup K$ by construction, so we have a contradiction.
\end{proof}

\begin{corollary}\label{non-crit-regions-to-apts}
Let $X$, $\phi$ and $h$ be as in the statement of the propostion. Let $P$ and $Q$ be points in $X_{\mathrm{reg}}$, and $\cP_{PQ}$ be the set of $\phi$-non-critical paths in $X_{\mathrm{reg}}$ starting at $P$ and ending at $Q$. Let $\Omega_{PQ} := \cup_{\gamma \in \cP_{PQ}} \mathrm{im} \gamma$ be the union of the images of elements of $\cP_{PQ}$. Then $h(\Omega_{PQ})$ is contained in the Finsler convex hull $[h(P), h(Q)]_{\mathrm{Fins}}$ of $h(P)$ and $h(Q)$. In particular, there exists an apartment containing $h(\Omega_{PQ})$. 
\end{corollary}

\begin{proof}
Let $R \in \Omega_{PQ}$. By definition of $\Omega_{PQ}$, there exists a $\phi$-non-critical path $\gamma$ starting at $P$ and ending at $Q$ such that $R = \gamma(s)$ for some $s \in I$. By the proposition, there is an apartment $A$ such that $h(\gamma(I)) \subset A$. So we can apply Lemma \ref{non-crit-convexity} to conclude that $h(R) \in [h(P), h(Q)]_{\mathrm{Fins}}$. We know that $[h(P), h(Q)]_{\mathrm{Fins}}$ is contained in the intersection of all apartments that contain both $h(P)$ and $h(Q)$. Since this is a non-empty intesection by the axiom (B1) for buildings, the proof is complete. 
\end{proof}

\begin{lemma}\label{local-non-crit}
Let $h: X \rightarrow \B$ be a harmonic $\phi$-map, and let $X_{\mathrm{reg}}\subset X$ be the set of points at which $h$ is regular.  Let $Z$ denote the ramification divisor of the spectral cover defined by $\phi$. Then for any $R \in X-Z$ there exists a pair of points $P$, $Q$, and two paths $\gamma$ and $\gamma'$ such that

\begin{enumerate}
\item $\gamma$ and $\gamma'$ start at $P$ and end at $Q$.
\item $\gamma^{-1} \circ \gamma'$ bounds a region $D$ containing $R$ that is toplogically isomorphic to a disc.
\item  $\gamma$ and $\gamma'$ are $\phi$-non-critical and contained in $X_{\mathrm{reg}}$.
\end{enumerate}
\end{lemma}

\begin{proof}
First of all, note that the singular set $X - X_{\mathrm{reg}}$ is given by the zeroes of the Hopf differential (see e.g. \cite{Daskalopoulos-Wentworth}). Since the Hopf differential on a Riemann surface is holomorphic, the singular set is discrete.

Now let $R \in X -Z$. Since $R$ is not a ramification point, there exists a neighborhood $U$ of $R$ on which we can write $\phi = (\phi_1, \ldots, \phi_r)$. Let $\cF_{ij}$ be the foliation on $U$ defined by $\re (\phi_i - \phi_j) = 0$. Then we can find a short curve $\gamma_0: [-1,1] \rightarrow X$ that is everywhere transversal to each $\cF_{ij}$, and such that $\gamma_0(0) = R$.  (E.g., take any vector $v$ in $T_RX$ that is not in any of the foliations. Fix a Riemannian metric on $X$. Then there is an $\epsilon > 0$ such that for $|t| < \epsilon$ the path $t \mapsto \exp(tv)$ has the required property.)

Now let $P = \gamma_0(-1)$ and $Q = \gamma_0(1)$. Let $\cX$ be a vector field along $\gamma_0(I)$ that is everywhere normal to $d\gamma(\partial_t)$. Then for $\epsilon > 0$ small enough the paths $\gamma(t) = \exp_{\gamma_0(t)}(\epsilon \cX(\gamma_0(t)))$ and  $\gamma'(t) = \exp_{\gamma_0(t)}(-\epsilon \cX(\gamma_0(t)))$ are non-critical, and satisfy the requirements $1.$, $2.$ and $3.$. (Here we have used the fact that the singularities are discrete). 
\end{proof}

\begin{remark}
The proof of the previous lemma relied on having control over the behavior of the singular set of the harmonic map. In our case, the proof is considerably simplified by the fact that the source is a Riemann surface. However, strong restrictions on the singular sets of harmonic maps hold in much greater generality. See, for example, \cite{Sun} for the case of harmonic maps from Riemannian manifolds to $\mathbb{R}$-trees. 

\end{remark}

\begin{lemma}\label{Isihara}
Let $h: X \rightarrow \B$ be a harmonic map from a Riemann surface $X$ to an affine building. Let $D \subset X$ be a closed disc in $X$, and $A$ be an apartment in $\B$. Suppose that $h(\partial D) \subset A$. Then $h(D) \subset A$.
\end{lemma}

\begin{proof}
Consider the function $f: \B \rightarrow \mathbb{R}$ which associates to a point $X$ its distance from the apartment $A$. i.e.,   $f(x) = d(x,A)$. Since $\B$ is a non-positively curved space, and $A$ is a convex subset, the function $f$ is convex. Since $h$ is harmonic, and $f$ is convex, $f \circ h$ is a subharmonic function on $X$ (see \cite{Farb-Wolf} for a proof in the tree case). It vanishes on $\partial D$, since $h(\partial D) \subset A$. It follows that $f$ vanishes on $D$, i.e., that $h(D) \subset A$.

\end{proof}

\begin{proposition}\label{phi-map-singularities}

Let $h: X \rightarrow \B$ be a harmonic $\phi$-map, and let $X_{\mathrm{reg}}\subset X$ be the set of points at which $h$ is regular.  Let $Z$ denote the ramification divisor of the spectral cover defined by $\phi$. Then $X - Z \subset X_{\mathrm{reg}}$. 
\end{proposition}

\begin{proof}
Let $R \in X - Z$. Then by  Lemma \ref{local-non-crit}, there is a disc $D$ containing $R$ such that $\partial D = \gamma^{-1} \circ \gamma'$, where $\gamma, \gamma': I \rightarrow X_{\mathrm{reg}}$ are $\phi$-non-critical paths. Applying Corollary \ref{non-crit-regions-to-apts}, we see that there is an apartment $A$ in $\B$ such that $h(\partial D) \subset A$. Applying Lemma \ref{Isihara}, we conclude that $h(D) \subset A$. This completes the proof.

\end{proof}

\subsection{A gluing construction}\label{gluing-construction}

Remark \ref{Bu-in-simple-cases} suggests an approach to constructing the universal building: since we know the universal building for simple building blocks from which the Riemann surface is glued together, we might try to produce the universal building by gluing together the universal buildings for each of these pieces. The construction we are about to describe is motivation by this idea, and will be used to in Section \ref{BNR-example} (and specifically in \S \ref{MAR-section}) as a step in the construction of universal building in a particular example. Since the construction itself is of much wider applicability, we decided to include it here. The input for this construction is an open cover that is ``adapted'' to a fixed cameral cover $\phi$ of the curve $X$. The output is a pre-building, which should be the image $\mathrm{im}(\hu)$ (Definition \ref{image-in-building}) of $X$ in the universal $\phi$-building $\Bu$.

\begin{definition}\label{adapted}
Let $\pi_{\phi}: \Sigma_{\phi} \rightarrow X$ be a cameral cover of a smooth algebraic curve over $\cc$ with Weyl group $\mathbb{S}_r$. A subset $U \subset X$ is $\phi$-adapted if it is simply-connected and if $\Sigma_{\phi}$ decomposes over $U$, i.e.,
\[
\Sigma_{\phi} \times_{U} X \simeq W \times U
\]
where $W$ is the Weyl group.  A locally finite cover $\{ U_{\alpha} \}_{\alpha}$ of $X$ is $\phi$-adapted if all pair-wise intersections of elements of the cover are $\phi$-adapted, and the interiors of the $U_{\alpha}$ cover the complement of the ramification locus of $\phi$. 
\end{definition}

On the cameral cover, we have a well defined section $\phii \in H^{0}(\Sigma_{\phi}, \mathfrak{t} \otimes \Omega^1_{\Sigma_{\phi}})$, which we may think of as a sequence of forms $\phii = (\phii_{1},..., \phii_r)$. For any $P \in \Sigma_{\phi}$ the map
\[
Q \mapsto \int_{P}^{Q} \phii
\]
is only well defined upto periods. However, if $U \subset X$ is $\phi$-adapted, then for each $P \in U$ we have a well defined map $\theta_P: \pi^{-1}_{\phi}(U) \rightarrow \cc^r$ given by
\[
Q \mapsto \int_{\gamma} \phii
\]

where $\gamma: I \rightarrow \Sigma_{\phi}$ is a path (automatically unique upto homotopy) with $\pi(\gamma(0)) = P$ and $\gamma(1) = Q$. Let $\psi_{P} = \re (\theta_P)$. Since $\sum_{i} \phii_i = 0$, we have that $\psi_P$ factors through the subset $\aff \subset \mathbb{R}^r$ defined by $\aff  = \{ (x_1,...,x_r) | \ \sum_{i} x_i = 0 \}$. We equip $\aff$ with the structure of an apartment for the affine Weyl group $W \ltimes \mathbb{R}^r$, where $W$ is the Weyl group of $\SL_r\cc$.

\begin{construction}\label{gluing-pre-building}
Let $\pi_{\phi}: \Sigma_{\phi} \rightarrow X$ be a cameral cover of a smooth algebraic curve over $\cc$. Let $\mathfrak{U}_{\phi}$ denote the category whose objects are $\phi$-adapted sets in $X$, and morphisms are inclusions. We are going to define a functor $\cT$ on $\mathfrak{U}_{\phi}$ taking values in the category of generalized chamber systems. On objects, $\cT$ is defined by the formula

\[
\cT(U) \simeq \left (\coprod_{P \in \pi^{-1}_{\phi}(U)} \{P \} \times \hull \psi_P(U) \right)/ \sim
\]

where $\sim$ is given by two types of relations:

\begin{enumerate}
\item $(P,v) \sim (wP,  w^{-1} v)$ for all $v$, and for all $P \in \pi^{-1}_{\phi}(U)$, and all $w \in W$. 
\item $(P, v) = (Q, \psi_{\pi_{\phi}(P)}(Q) + v)$  whenever $P$ and $Q$ are in the same connected component of $\pi^{-1}_{\phi}(U)$. 

\end{enumerate}

If $U \hookrightarrow V$ is an inclusion, then there is an induced inclusion $\pi^{-1}_{\phi}(U) \hookrightarrow \pi^{-1}_{\phi}(V)$, which in turn induces a map $\cT(U) \hookrightarrow \cT(V)$ in the obvious way. It is easy to see that this is a map of chamber systems.  We extend $\cF$ to a functor on disjoint unions of elements of $\mathfrak{U}_{\phi}$ by requiring that it preserve coproducts.

Now suppose that $\cU = \{ U_{\alpha} \}_{\alpha}$ is an $\phi$-adapted open cover of $X$. Then we can define a pre-building $(\Bpre, \cF, \cC)(\cU)$ associated to this open cover in the following way. As a chamber system, $\Bpre$ is the coequalizer of the natural diagram:

\[
\xymatrix{
\cT(U \times_{X} U) \ar@<1ex>[r] \ar@<-1ex>[r]  & \cT(U)  \\
}
\]

The cubicles in $\Bpre$ are given by the images of the sets $\{P\} \times \hull (f_P(U))$ under the natural projection $\cT(U) \rightarrow \B$. 

\end{construction}

This construction relied, of course, on the following lemma:

\begin{lemma}
The forgetful functor from the category of generalized chamber systems to the category of sets creates all small colimits.
\end{lemma}

\begin{proof}
The proof is almost identical to the proof that the forgetful functor from topological spaces to sets creates small colimits. We leave the details to the reader.
\end{proof}

The key proposition of this section is the following:

\begin{proposition}\label{gluing-proposition}
Let $X$ be a Riemann surface, $\phi \in \oplus_{i=2}^{r}H^{0}(X, \omega_{X}^{\otimes i})$ be a point in the $\SL_r\cc$ Hitchin base, and $h: X \rightarrow \B$ be a harmonic $\phi$-map to a building $\B$.  Let $\{ U_{\alpha} \}$ be a $\phi$-adapted cover of $X$, and suppose that for each $\alpha$ there exists an apartment $A_{\alpha}$ such that $h(U_{\alpha}) \subset A_{\alpha}$. Then there exists a unique isometry of pre-buildings that makes the following diagram commute:
\[
\xymatrix{
X \ar[r] \ar[dr]_{h} & \Bpre(\cU) \ar[d] \\
& \B \\
}
\]

\end{proposition}

\begin{proof}
Let $U^{o}_{\alpha}$ denote the interior of $U_{\alpha}$. Let $f: \aff \rightarrow A_{\alpha}$ be a chart. Then, since $h$ is a $\phi$-map, and $U^{o}_{\alpha}$ does not intersect the ramification locus, on $U_{\alpha}^{o}$ we have $d(f^{-1} \circ h) = (\phi_1,...,\phi_r)$, where ``$\phi =
\{ \phi_1,...,\phi_r \}$''. More precisely, there exists a section $s$ of $\Sigma_{\phi}$ over $U^{o}_{\alpha}$ such that for any $P \in U^{o}_{\alpha}$ we have $d(f^{-1} \circ h) = d(\psi_P \circ s)$.

Now suppose that we are given a section $s$ and a point $x \in U^{o}_{\alpha}$. Integrating both sides of $d(f^{-1} \circ h) = d(\psi_P \circ s)$, we see that there is a unique chart $f: \aff \rightarrow A_{\alpha}$ such that the following diagram commutes:

\[
\xymatrix{
\pi_{\phi}^{-1}(U^{o}_{\alpha}) \ar[r]^{\psi_P}  & \aff \\
U^{0}_{\alpha} \ar[r]_{h} \ar[u] ^{s} & A \ar[u]_{f^{-1}} \\
}
\]

Thus we have an injective map associating to a pair $(s,P)$ the corresponding chart $f = f_{(s,P)}$. The map $f_{(s,P)}: \mathrm{im}(\psi_{P}(U^{o}_{\alpha})) \rightarrow A$ extends naturally to a map on $\hull(\psi_{P}(U^{o}_{\alpha}))$.

This correspondence is compatible with the action of the affine Weyl group. More precisely, for any $w$ in the sperical Weyl group, and any pair of points $P$, $Q$ in $U^{o}_{\alpha}$, we have that $f_{(ws, Q)}^{-1} \circ f_{(s,P)}$ is given by the action of the element $(w, \psi_{ws(P)}(Q)) \in W \ltimes \mathbb{R}^{r-1}$.  It follows that the collection of maps
\[
f_{(s,P)}: \hull(\psi_{P}(U^{o}_{\alpha})) \rightarrow A
\]

descend to a map 

\[
\cT(U_{\alpha}) \hookrightarrow A \hookrightarrow \B
\]

Note that this map is just the inclusion of an enclosure into the apartment. Furthermore, it is clear from the construction that the maps $ \cT(U_{\alpha}) \rightarrow \B$ coequalize the two morphisms

\[
\xymatrix{
\cT(U \times_{X} U) \ar@<1ex>[r] \ar@<-1ex>[r]  & \cT(U)  \\
}
\]

where $U = \coprod U_{\alpha}$. The proposition follows. 

\end{proof}


\subsection{Spectral networks from the universal building}\label{spectral-networks-buildings}

Spectral networks, which were introduced in \cite{GMN-Spectral-Networks}, are certain decorated graphs drawn on a Riemann surface that encode part of the geometry of a spectral cover. The goal of this subsection is to remind the reader about the notion of spectral networks, and to formulate Conjecture \ref{spectral-networks-Bu}, which relates the spectral network of a spectral cover $\phi$ to the singularities of the universal $\phi$-building.

  Assume that the spectral cover $\Sigma$ is smooth and let $P \in X$.   

\begin{definition}\label{BPS-network-definition}
  We say that $P \in X$ \textnormal{supports an open BPS-state of phase $\theta$} if there exist lifts $P_0\neq P_1$ of $P$ to $\Sigma$ and a path $\gamma:[0,1]\rightarrow \Sigma$ such that:
  \begin{itemize}
  	\item $\gamma(0)=P_0, \gamma(1)=P_1$
	\item $\int_\gamma \lambda \in e^{i\theta} \mathbb R_+$
	\item $\pi_{\phi *}(\partial_{\gamma(1)})=-\pi_{\phi *}(\partial_{\gamma(0)})$
	\item $\langle\lambda(\gamma(1)), \partial_{\gamma(1)} \rangle + \langle\lambda(\gamma(0)),\partial_{\gamma(0)}\rangle \in e^{i\theta} \mathbb R_+$
  \end{itemize} 
		The \emph{spectral network of phase} $\theta$, $W_\theta$, is the union of all $P\in C$ such that $P$ supports an open BPS-state of phase $\theta$.
	\end{definition}
  
  Let $P$ be a branch point of $\Sigma \rightarrow X$. Then the spectral network near $P$ looks as in Figure \ref{SpecBranch}: for every angle $\theta$ there are three rays starting at the branch point, because of the $\mathbb Z/2$ monodromy of $\lambda$ around a branch point.
  	\begin{figure}[h]
	\centering
	\includegraphics[width=.5\textwidth]{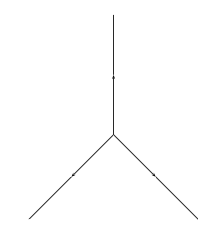}
	\caption{Spectral Network near a branch point}
	\label{SpecBranch}
	\end{figure}
  
The spectral network curves emerging from different branch points $P_{1,2}\in X$ may intersect. Let $s_{1,2}$ be spectral network lines starting at $P_{1,2}$ with label $(ij)$ and $(jk)$ respectively. Here $i$, $j$ and $k$ are labels for the sheets of the spectral cover. Assume that $s_1$ and $s_2$ intersect at a point $Q$, as in Figure \ref{Collision}. In this situation there is a \textit{collision line} $s_3$ starting at $Q$: integrate along the green line in Figure $\ref{Collision}$

    	\begin{figure}[h]
	\centering
	\includegraphics[width=.7\textwidth]{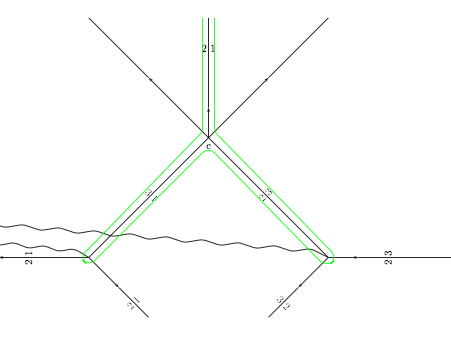}
	\caption{Colliding Spectral Network lines}
	\label{Collision}
	\end{figure}

\begin{conjecture}\label{spectral-networks-Bu}
Let $X$ be a Riemann surface, and let $\phi$ be a spectral cover. Let $\cW$ be the spectral network associated to $\phi$. Then the support of $\tilde{\cW}$ coincides with the inverse image of the singular set of $\Bu$ under the universal $\phi$ map $\hu: \tilde{X} \rightarrow \Bu$. See Caveat \ref{spec-net-caveat}
\end{conjecture}

\begin{caveat}\label{spec-net-caveat}
In the statement of Conjecture \ref{spectral-networks-Bu}, the singular set of the building is the set of all singular points. A point in a building is singular if no neighborhood of the point is contained in a single apartment. It is not hard to see that with this definition, the conjecture cannot be true exactly as stated. The reason for this is that our current definition of spectral networks generates collision lines going off in \emph{only one direction} from the collision point, whereas the singular set of the building does not distinguish between the forward and backward directions (see for example Proposition \ref{Bu-sing-BNR}). There are potentially two ways to change our definitions to make the conjecture true:

\begin{itemize}

\item Modify the definition of spectral networks to include the ``backward collision lines''. Proposition \ref{Bu-sing-BNR} gives evidence that the conjecture is then true.

\item Modify the definition of singular set of the building to somehow distinguish ``forward and backward directions''. It is not clear at this point how to do this. 

\end{itemize}
\end{caveat}


\section{Singular Perturbation Theory}\label{WKB}


\subsection{The Riemann-Hilbert WKB problem}\label{Riemann-Hilbert-Hitchin}

A finite dimensional complex representation of a groupoid $\Gamma$ is a functor $\rho_{(\relbar)}: \Gamma \rightarrow \Vectc$, where $\Vectc$ is the category of finite dimensional $\mathbb{C}$-vector spaces. A representation is of rank $r$ if $\rho_x$ is an $r$-dimensional vector space for all objects $x$ in $\Gamma$. Let $\Rep(\Gamma, r)$ denote the set of representations of $\Gamma$ of rank $r$. By a hermitian metric on a representation $\rho$, we mean simply a hermitian metric on the vector space $\rho_x$ for every object $x$ in $\Gamma$. 

Recall that  a family $\{\mu_{t}\}_{t \in \mathbb{R}}$ of scale factors is a function $\mu: \mathbb{R} \rightarrow \mathbb{R}$ with the property that $\mu_{t} \rightarrow \infty$ as $t \rightarrow \infty$. Let $\{\mu_{t}\}$ be a family of scale facts. A function $f: \mathbb{R} \rightarrow \mathbb{R}$ is of \emph{exponential growth at infinity with respect to} $\{\mu_{t}\}$ if there exist constants $C$ and $\eta$ such that $|f(t)| \leq C \exp (\eta \mu_{t})$ for $t \gg 0$. The set of functions of exponential growth with respect to $\{ \mu_t \}$ form a ring with valuation. For a function $f$, the valuation $\nu = \nu(\mu)$ is the infimum of the set of real numbers $\eta$ for which there exists a constant $C$ such that $f(t) \leq C \exp (\eta \mu_t)$ for $t \gg 0$:

\[
\nu(f) := \limsup_{t \to \infty} \frac{1}{\mu_{t}} \log |f(t)|
\]

We will use this valuation to measure the rate of growth of the size of a representation varying in a family depending on a large parameter. First we need to say how we will measure the size of a representation, and which families of representations are admissible:

\begin{definition}\label{exp-type}
Let $\Gamma$ be a groupoid, and let $\rho: \mathbb{R} \rightarrow \Rep( \Gamma, r)$ be a family of representations. For each $t \in \mathbb{R}$, let $h(t)$ be a  hermitian metric on $\rho(t)$. Denote by $\|\relbar\|_{t}$ the operator norm on $\Hom(\rho_x(t), \rho_y(t))$ associated to $h(t)$. Let $\{\mu_{t}\}_{t \in \mathbb{R}}$ be a family of scale factors. Then $\rho$ is \emph{of exponential type} with respect to $(\{h(\relbar)\}, \{\mu_{t}\})$ if for each arrow $\gamma$ in $\Gamma$, the function $t \mapsto \|\rho_{\gamma}(t)\|$ is of exponential growth. 
\end{definition}

\begin{definition}\label{dilation-spectrum}
Let $\Gamma$ be a groupoid, and let $\rho: \mathbb{R}  \rightarrow \Rep(\Gamma, r)$, be a family of representations that is of exponential growth with respect to a family of metrics $h(t)$ and a family of scale factors $\mu_{t}$. Let $\gamma$ be an arrow in $\Gamma$. The \emph{exponent} $\nu(\gamma)$ of $\gamma$ with respect to $(\rho, h, \mu)$ is the number
\[
\nu(\gamma):= \limsup_{t \to \infty}\frac{1}{\mu_{t}} \log \| \rho_{\gamma}(t) \|_{t}
\] 
The \emph{dilation spectrum} of $\gamma$ is the vector $\nuu (\gamma) = (\nu_1(\gamma),...,\nu_r(\gamma))$ uniquely determined by the condition that $\sum_{i=1}^{k} \nu_i(\gamma)$ is the WKB exponent of $\gamma$ with respect to $(\bigwedge^{k} \rho, h, \mu)$ for each $k$ with $1 \leq k \leq r$.   

\end{definition}

One of the main goals of this paper is to give a geometric interpretation of the WKB dilation spectrum in situations where the groupoid and the family of representations are of geometric origin. The rest of this subsection is devoted to formulating two such ``WKB problems''. 

Let $X$ be a topological space. A basic invariant of the homotopy type of $X$ is its 1-truncation: the fundamental groupoid $\pi_{\leq 1}(X)$. When $X$ admits some additional geometric structure, representations of $\pi_{\leq 1}(X)$ often admit geometric interpretations. For instance, on a smooth manifold, the category of representations of the fundamental groupoid is equivalent, via the Riemann-Hilbert correspondence to the category of vector bundles equipped with a flat connection. 

\begin{definition}\label{WKB-dilation-spectrum}
Let $X$ be a complex manifold, and let $\{(\cE_{t}, \nabla_{t})\}_{t \in \mathbb{R}}$ be a family of holomorphic vector bundles on $X$ equipped with integrable connections. Fix a family of scale factors $\{\mu_{t}\}$ and a hermitian metric $h_{t}$ on $\cE_{t}$ for each $t$. 

\begin{enumerate}

\item The family $\{(\cE_{t}, \nabla_{t})\}_{t \in \mathbb{R}}$ is said to be of \emph{exponential type} with respect to $(h, \mu)$ if the associated monodromy representation $\trans_{t}: \pi_{\leq 1}(X) \rightarrow \Vectc$ is of exponential type in the sense of Definition \ref{exp-type}. 

\item Suppse  $\{(\cE_{t}, \nabla_{t})\}_{t \in \mathbb{R}}$ is of exponential type, and let $\gamma$ be a path in $X$.  Then the \emph{WKB dilation spectrum} of $\gamma$ with respect to   $\{(\cE_{t}, \nabla_{t}), h_{t}, \mu_{t}\}_{t \in \mathbb{R}}$ is defined to be the dilation spectrum of $\gamma$ with respect to $\{ \trans_{t}, h, \mu \}_{t \in \mathbb{R}}$ in the sense of Definition \ref{dilation-spectrum}.

\end{enumerate}

\end{definition}

The ``\emph{Riemann-Hilbert WKB problem}'' is the problem of determining the WKB dilation spectrum of a family $(\cE_{t}, \nabla_{t})$ of holomorphic vector bundles with flat holomorphic connection. This problem has been treated in great depth in monograph \cite{Carlos-book}, and in the paper \cite{Carlos-infinity}. The main novelty of the treatment here is the geometric interpretation of WKB exponents in terms of harmonic maps to buildings.

For simplicity, this paper will only treat the case where $\cE = \cE_{t}$ is independent of $t$. The space of holomorphic connections on $\cE$ is a torsor for $\Shend(\cE)$ valued 1-forms, so we have $\varphi_{t}:= \nabla_{t} - \nabla_{0} \in H^{0}(X, \Shend(\cE) \otimes \Omega^{1}_{X})$. The flatness of the connections implies the integrability of $\varphi_{t}$ in the sense of Definition \ref{spectral-cover-Higgs}. So to specify the family of connections is equivalent to specifying an initial connection $\nabla_0$, and a family of Higgs bundles $(\cE, \varphi_{t})$. Note that if there exists a constant $C$ such that $\| \varphi_{t} \| \leq C \mu_{t}$ for $t \gg 0$, then the associated family $( \cE, \nabla_{t})$ is of exponential type. For instance, we may take $\mu_{t} = t$ for all $t$, and let $\varphi_{t} = t \varphi$ for some fixed Higgs field $\varphi$. All of the interesting features of the theory are already visible in this situation.


\subsection{The local WKB approximation}\label{local-WKB-section}

This subsection treats the ``local Riemann-Hilbert WKB approximation problem'', i.e., the problem of determining the growth rates of the transport operators for paths of ``sufficiently short length''. This result is a version of the ``classical WKB approximation''. As we could not find a reference in the literature with a statement and proof of this result that is tailored to our current needs, we include a self-contained treatment of the statement we need here. For a detailed discussion of ``local WKB type problems'' the reader is referred to \cite{Wasow}.

Suppose we have an $r\times r$ matrix of functions $a_{ij}(x)$ defined for $x\in (-c,c)\subset \rr$. 
Suppose $t>0$ is a large real number, and $\epsilon$ and $C$ are constants. Suppose we have the estimates
$$
\re a_{11}(x) > \epsilon t + \re a_{ii}(x)
$$
for $i=2,\ldots , r$, 
and
$$
|a_{ij}(x)| < C t^{1/2}
$$
for $i\neq j$. 

Suppose for the moment that $a_{11}(x)=0$ for all $x$. The first condition then says
$\re a_{ii}(x)\leq -\epsilon t$. We will remove this assumption later by tensoring with a rank one system. 

For a function $f(x)$ let $f'(x)$ denote its derivative with respect to $x$. 
Consider the differential equation
$$
f'_i(x)= \sum _{j=1}^r a_{ij} f_j(x).
$$
Let $F_i(x)$ denote a solution with initial conditions $F_1(0)=1$, $F_i(0)=0$ for $i\geq 2$.

For $2\leq i\leq r$ put $G_i(x):= F_i(x)/F_1(x)$. Then $G_i(x)$ satisfy 
$$
G'_i(x) = \frac{F_1(x) F'_i(x) - F'_1(x) F_i(x)}{F_1(x)^2}
$$
hence
$$
G'_i(x) = a_{i1}(x) + \sum _{j=2}^r a_{ij}(x)G_j(x) - a_{1j}(x) G_j(x)G_i(x) .
$$

Put $M(x):= \sum _{i=2}^r |G_i(x)|^2$. We have
$$
M'(x) =  \re \sum _{i=2}^r G'_i(x) \overline{G}_i(x)  
$$
$$
= \re \sum_i  a_{i1}(x)\overline{G}_i(x) + \sum _{i=2}^r \sum _{j=2}^r a_{ij}(x)G_j(x)\overline{G}_i(x) - a_{1j}(x) G_j(x)G_i(x)\overline{G}_i(x) 
$$
$$
= \re \sum _i a_{i1}(x)\overline{G}_i(x) + a_{ii} |G_i(x)|^2 + \re \left(  \sum _{i\neq j} a_{ij}(x)G_j(x)\overline{G}_i(x) - \sum _{i,j} a_{1j}(x) G_j(x)|G_i(x)|^2 \right) .
$$ 
We now use $\re a_{ii} \leq -\epsilon t$ and $|a_{ij}(x)|\leq C t^{1/2}$ for $i\neq j$. Furthermore,
$$
|G_j(x)\overline{G}_i(x)| \leq M(x)
$$
and 
$$
|G_j(x)|G_i(x)|^2| \leq M(x)^{3/2}. 
$$

Thus, we get, after increasing the constant, 
$$
M'(x) \leq  -\epsilon t M(x) + Ct^{1/2} (M(x)^{1/2} + M(x) +  M(x) ^{3/2}). 
$$

Suppose $x$ is a point such that $1/2\geq M(x)\leq 1$. Then, if $t$ is big enough that 
$$
\epsilon t / 2 > 3 Ct^{1/2}
$$
which is equivalent to $t>  (4C/\epsilon ) ^2$, we conclude that 
$$
\epsilon t M(x) >   Ct^{1/2} (M(x)^{1/2} + M(x) +  M(x) ^{3/2}).
$$
Therefore, at any such point $x$, we have $M'(x) <0$. As $M(0)=0$, it follows that $M(x)\leq 1/2$ for all $x$. 

Now, since $M(x)\leq 1/2$ we get $M(x)^{3/2} + M(x)  \leq 2 M(x)^{1/2}$ and our estimate becomes 
$$
M'(x) \leq -\epsilon t M(x) + Ct^{1/2} M(x)^{1/2}. 
$$

(Notice that our estimate becomes singular at $M(x)=0$, corresponding to the fact that the initial condition $M(0)=0$ 
doesn't persist.)

Suppose $\alpha \leq M(x)\leq 2 \alpha$. Then 
$$
-\epsilon t M(x) + Ct^{1/2} M(x)^{1/2} < 0
$$
if we have $\epsilon t \alpha > Ct^{1/2}\alpha ^{1/2}$ (increasing $C$ here). 
Thus, this equation tells us that $M(x)\leq \alpha$ for all $x$. 
We can solve for $\alpha$: whenever 
$$
\alpha > (Ct^{-1/2} /\epsilon )^2 ,
$$
or in other words $\alpha > C_1 t^{-1}$, we get this condition. 

Therefore, we may finish by concluding that
$$
M(x) \leq C_1 t^{-1}
$$
or 
$$
| G_{\cdot}(x)|\leq C_2 t^{-1/2}
$$
for all $x$. Translating back in terms of $F_i$ this says that
$$
|F_i(x)|\leq C_2 t^{-1/2} |F_1(x)|
$$
for $2\leq i\leq r$. 

Next we note that
$$
F'_1(x) = \sum _{j\geq 2} a_{1j}(x) F_j(x),
$$
so using the previous estimate and $|a_{1j}(x)|\leq Ct^{1/2}$ for $j\neq 1$, we have
$$
|F'_1(x)| \leq C_3 |F_1(x)|.
$$
We get
$$
\frac{d}{dx} |F_1(x)|^2 = \re F'_1(x) \overline{F}_1(x).
$$
Hence
$$
-C_3 |F_1(x)|^2 \leq \frac{d}{dx} |F_1(x)|^2 \leq C_3 |F_1(x)|^2.
$$
Setting $h(x):= \log |F_1(x)|^2$
this gives
$$
-C_3 \leq h'(x) \leq C_3.
$$
Therefore, for a small enough choice of $c$ depending on our constants up to now but independent of $t$,
we have
$$
-\ln (2)/2 \leq h(x) \leq \ln (2)/2
$$
for all $x\in [0,c]$. This gives
$$
1/2\leq |F_1(x)| \leq 2
$$
for $x\in [0,c]$. 

We are now ready to apply this discussion. In the following statement, which is our local WKB approximation, 
we remove the restriction $a_{11}=0$. 

\begin{proposition}\label{local-WKB}
Suppose $a^v_{ij}(x)$ is a matrix depending on some parameter $v\in V$ as well as $x\in [0,1]$, and suppose we have a positive real valued function $v\mapsto t^v$.
Suppose $\epsilon$, $C$ and $t_0$ are constants such that for all $v\in V$ with $t^v\geq t_0$, we have
$$
\re a^v_{ii}(x)\leq \re a^v_{11}(x)-\epsilon t^v \;\;\;\; i\geq 2,
$$
and 
$$
|a^v_{ij}(x)| \leq C(t^v)^{1/2}\;\;\;\; i\neq j.
$$
Let $T^v(x)$ be the fundamental solution matrix for transport from $0$ to $x$. Set 
$$
\alpha ^v(x):= \int _0^x a_{11}(x) dx.
$$
Then there is a small constant $c$ depending on $\epsilon$ and $C$, and a value $t_1$, such that
for any $v\in V$ with $t^v\geq t_1$ and any $x\in [0,c]$ we have 
$$
\frac{1}{2} e^{\alpha ^v(x)t^v} \leq \| T^v(x) \| \leq C_r e^{\alpha ^v(x)t^v} 
$$
and more precisely the same bound also holds for the upper left coefficient $T^v_{11}(x)$ of the transport matrix. 
Here $C_r$ is a constant depending only on $r$. 
\end{proposition}
\begin{proof}
By multiplying the solution by $e^{-\alpha ^v(x)t^v}$ it suffices to treat the case $a_{11}(x)=0$ which we now assume. 

Suppose $f_i(x)$ is a solution with $f_1(0)=1$ and $|f_i(0)| \leq 1$.  Proceed as before, introducing $g_i(x):= f_i(x)/f_1(x)$. 
We again get
$$
g'_i(x) = a_{i1}(x) + \sum _{j=2}^r a_{ij}(x)g_j(x) - a_{1j}(x) g_j(x)g_i(x) .
$$
Again, put
$m(x):= \sum _{i=2}^r |g_i(x)|^2$. We have
$$
m'(x) =  \re \sum _{i=2}^r g'_i(x) \overline{g}_i(x)  
$$
$$
= \re \sum _i a_{i1}(x)\overline{g}_i(x) + a_{ii} |g_i(x)|^2 + \re \left(  \sum _{i\neq j} a_{ij}(x)g_j(x)\overline{g}_i(x) - \sum _{i,j} a_{1j}(x) g_j(x)|g_i(x)|^2 \right) .
$$ 
Using as before $\re a_{ii} \leq -\epsilon t$ and $|a_{ij}(x)|\leq C t^{1/2}$ for $i\neq j$, together with
$$
|g_j(x)\overline{g}_i(x)| \leq m(x)
$$
and 
$$
|g_j(x)|g_i(x)|^2|\leq m(x)^{3/2},
$$
we get after increasing the constant that 
$$
m'(x) \leq  -\epsilon t m(x) + Ct^{1/2} (m(x)^{1/2} + m(x) +  m(x) ^{3/2}). 
$$
We have $m(0)\leq r$. Consider a point $x$ with $r\leq m(x)\leq 2r$.
Then for $t$ big enough, the second term will have smaller size than the first term and we get $m'(x)\leq 0$.
Therefore, $m(x)\leq r$ for all $x$.

Notice, in particular, that if $f_1(x)$ becomes very small then all of the $f_i(x)$ become small. They can't all go to zero,
by considering the corresponding differential equation on the determinant bundle which constrains $\det T(x)$ to satisfy
a linear equation. This justifies the division by $f_1(x)$ so the $g_i(x)$ are well defined for all $x$. 

Our discussion applies to the sum of the first and any other column of the transport matrix. This gives a bound on the 
sizes of the other columns of the transport matrix, since we have treated the first column previously. We conclude
$$
\| T^v(x) \| \leq C_r 
$$
which gives back the required estimate of the form
$$
\| T^v(x) \| \leq C_r e^{\alpha ^v(x)t^v} 
$$
in the case $a_{11}$ arbitrary. 
The previous discussion provided the estimate
$$
\frac{1}{2} e^{\alpha ^v(x)t^v} \leq T^v_{11}(x)
$$
for $x\in [0,c]$. This completes the proof. 
\end{proof}


\subsection{The asymptotic cone computes the dilation spectrum}\label{cone-computes-spectrum}

For a vector space $V$, equipped with a trivialization $\bigwedge^rV \simeq \cc$, let $\Met(V)$ denote the space of hermitian metrics on $V$ for which the induced metric on $\bigwedge^r V$ agrees with the standard metric on $\cc$. $\Met$ defines a functor from the groupoid of finite dimensional vector spaces to the  groupoid $\mathrm{Man}^{iso}$ of manifolds. A choice of an isomorphism $V \simeq \cc^r$ determines an isomorphism $\Met(V) \simeq \SL_r\cc/SU_r$.

Let $\rho(t): \Gamma \rightarrow \Vectc$, $t \in \mathbb{R}$ be a family of representations taking values in $r$-dimensional vector spaces, and equipped with a trivialization of $\bigwedge^r \rho$. Composing with $\Met$, we have an induced family of functors
\[
\Met \circ \rho_t: \Gamma \rightarrow \mathrm{Man}^{iso}
\]

The key observation behind the main theorem of this subsection is the following simple observation

\begin{lemma}\label{distance-to-exponent}
Let $\Gamma$ be a groupoid and $\rho: \mathbb{R} \rightarrow \Rep(\Gamma,r)$ be a family of representations with trivialized determinant, and let $h$ be a metric on this family. Let $\gamma$ be a morphism from $x$ to $y$ in $\Gamma$ .  Let $\|A\|_t$ denote the operator norm of a morphism $A: \rho_x \rightarrow \rho_y$ with respect to the metrics  $h_x(t)$ and $h_y(t)$. Define $\alpha_{t}(\gamma) := (\alpha_{t,1}(\gamma),...,\alpha_{t,r}(\gamma))$ recursively by the formula $\sum_{i=1}^k \alpha_{t,k}(\gamma) = \log \| \bigwedge^{k} \rho_{\gamma}(t) \|$. Then we have that $\vecd(\rho_{\gamma}(t)_{\ast} (h_{x}(t)), h_{y}(t)) = \alpha_{t}(\gamma)$. 
\end{lemma}

\begin{proof}
Recall the notion of vector distance (Definition \ref{vector-distance}). Let $h$ and $k$ be hermitian metrics on an $r$-dimensional vector space, which we think of as points in the symmetric space $\Met(V)$. Choose a basis $\{ e_i \}$ of $\rho_y$ that is orthonormal for both $h$ and $k$ -- such a basis exists by the spectral theorem. Let $\|e_i\|_k = e^{\alpha_i} \| e_i \|_{h}$. Reorder the vectors if necessary, so that $\alpha_1 \geq \alpha_2 \geq ... \geq \alpha_r$. Then the vector distance $\vecd(h,k) =  (\alpha_1,...,\alpha_r)$. 

Now let $A: W \rightarrow V$ be any linear operator, and suppose that $k = A_{\ast} h'$ for some hermitian metric $h'$ on $W$. Then by definition, the operator norm of $\|A \|$ with respect to $h'$ and $h$  is 

\[
e^{\alpha_1} = \sup_{\| w \|_{h'} = 1} \|Aw\|_{h}
\]

Since the eigenvalues of $\| \bigwedge^k A \|$ are computed from the $k \times k$ minors of $A$, we see that 

\[
e^{\sum_{i=1}^{k} \alpha_i} =  \sup_{\| w \|_{h'} = 1} \|\bigwedge^k Aw\|_{h}
\]

In particular, we see that $\sum_{i=1}^{k} \alpha_i = \log \| \bigwedge^k A\|$. The lemma follows by applying this discussion to the case where $W = \rho_x$, $h' = h_x(t)$, $h = h_y(t)$, $V = \rho_y$ and $A = \rho_{\gamma}(t)$. 
\end{proof}

Now let $\{\mu_t\}$ be a family of scale factors, i.e., a family of real numbers going to $\infty$. Fix an ultrafilter $\omega$ on $\mathbb{R}$ whose support is a countable set.  Then we can pass to the asymptotic cone levelwise, with the basepoints in each fiber $\Met(\rho_x)$ being given by the family of points $h_x(t)$. This produces a  family of functors 
\[
\Cone_{\omega, \mu_t} \circ \Met \circ \rho(t): \Gamma \rightarrow \cM
\]

where $\cM$ is the category of metric spaces and isometries. By the Kleiner-Leeb theorem \cite{Kleiner-Leeb}, this functor actually factors through the subcategory of affine buildings and isometries. We can now appeal to \cite[Proposition 4.4]{Parreau-compactification}, which tells us that

\[
\vecd_{\Cone_{\omega}}([\rho_{\gamma}(t)]_{\ast} [h_x(t)], [h_y(t)]) = \lim_{\omega} \frac{1}{\mu_t} \vecd(\rho_{\gamma}(t)_{\ast}h_x(t), h_y(t))
\]

where $[h_x(t)]$ denotes the point in the asymptotic cone corresponding to the family of points $h_x(t) \in \Met(\rho_x(t))$.

\begin{definition}\label{ultrafilter-exponent}
Fix an ultrafilter $\omega$ on $\mathbb{R}$ whose support is a countable set, and a family of scale factors $\{\mu_t\}$.  
\begin{enumerate}
\item Let $\rho: \mathbb{R} \rightarrow \Rep(\Gamma, r)$ be a family of representations, and let $\gamma$ be a morphism in $\Gamma$. The \emph{ultrafilter exponent} of $\gamma$, denoted $\nu_{\gamma}^{\omega}$ is defined by the formula

\[
\nu_{\gamma}^{\omega} = \lim_{\omega} \frac{1}{\mu_t} \log \| \rho_{\gamma}(t) \|
\]

\item Let $X$ be a Riemann surface, and $(\cE, \nabla_0 + t \varphi)$ be a family of  integrable connections on a fixed holomorphic vector bundle $\cE$ (with trivialized determinant connection). The \emph{WKB ultrafilter exponent} of this family with respect to the ultrafilter $\omega$ is the ultrafilter exponent of the family of monodromy representations $T: \mathbb{R} \rightarrow \Rep(\pi_{\leq 1}(X), \SL_r\cc)$. One similarly defines the ultrafilter WKB dilation spectrum $\nuu^{\omega}$. 

\end{enumerate}
\end{definition}

What we have just described is the argument due to Parreau \cite{Parreau-compactification} that the ultrafilter exponents of a family  $\rho: \mathbb{R} \rightarrow \Rep(\Gamma, r)$ of representations (with determinant trivialized) are computed by the ``\emph{translation vectors}  $\vecd_{\Cone_{\omega}}([\rho_{\gamma}(t)]_{\ast} [h_x(t)], [h_y(t)])$. In this paper, we are primarily interested in the ``geometric situation'', where these representations are representations of the fundamental group arising from a family of integrable connections. In this case, the translation vector can be interpreted as the distance between two points under a harmonic map. The rest of this subsection is devoted to explaining this statement. The main theorem is the following:

Let $(\cE, \varphi)$ be a rank $r$ Higgs bundle on a Riemann surface $X$, and let $\nabla_0$ be an integrable connection on $\cE$. Fix a hermitian metric on $\cE$, and a base point $P \in X$. Then $\tilde{X}$ is identified with homotopy classes of paths ending at $P\in X$, and we have a family of maps 
\[
h_t: \tilde{X} \rightarrow \Met(\cE_P)
\]

for $t \in \mathbb{R}$, sending a homotopy class of paths $\gamma$ to the metric $T_{\gamma}(t)_{\ast}(h_{\gamma(0)})$. Choosing a basis of $\cE_P$ gives an identification $\Met(\cE_P) \simeq \SL_r\cc/SU_r$.

\begin{theorem}\label{cone-wkb}
Let $(\cE, \nabla_0 + t \varphi)$  and $h_t$ be as in the previous paragraph. Let $\phi = \Char \varphi$ be the corresponding point in the Hitchin base. Fix some set of scale factors $\{ \mu_t \}$. Then 
\begin{enumerate}
\item for each path $\gamma \in \pi_{\leq 1}(X)$ there is an ultrafilter $\omega$ on $\mathbb{R}$, whose support is a countable set, such that $\nuu^{\omega}_{\gamma} = \nuu_{\gamma}$ (i.e., the ultrafilter WKB exponent coincides with the WKB exponent).

\item if $\gamma$ is a $\phi$-non-critical path and $\mu _t=t$, then for any ultrafilter $\omega$ we have $\nuu^{\omega}_{\gamma} = \nuu_{\gamma}$. In particular
\[
 \nuu_{\gamma} = \vecd_{\Cone_{\omega}}(h(P), h(Q))
\]
where $\Cone_{\omega}$ is the asymptotic cone of $\Met(\cE_P) \simeq \SL_r\cc/\SU_r$ with respect to $(\omega, \{\mu_t\})$, and $P = \tilde{\gamma}(0)$, $Q = \tilde{\gamma}(1)$ for any lift $\tilde{\gamma}$ of $\gamma$ to the universal cover $\tilde{X}$. 

\item if $\hw : \tilde{X}\rightarrow \Cone_{\omega}$ denotes the limiting map in the previous item, then $\hw$ maps $\gamma ([0,1])$ into
the Finsler convex hull of $\{ P,Q\}$, in particular into any apartment containing $P$ and $Q$.  
\end{enumerate}
\end{theorem}

\begin{proof}[Sketch of proof.]
Fix $\gamma  \in \pi_{\leq 1}(X)$. For the first statement, not that there exists a sequence of points $\{t_n\}$ in $\mathbb{R}$ such that

\[
\lim_{n \to \infty} \frac{1}{\mu_{t_n}} \log \| T_{\gamma}(t) \| = \limsup \frac{1}{\mu_t} \log \| T_{\gamma}(t) \|
\]
Then it is clear that any ultrafilter $\omega$ supported on $\{t_{n}\}$ has the required property.

The second and third parts are true for sufficiently short non-critical paths by the local WKB theorem (Proposition \ref{local-WKB}), applied to exterior powers as described in Lemma \ref{distance-to-exponent}. 
One now argues in a similar manner to Proposition \ref{phi-maps-non-crit} to conclude that theorem is true for arbitrary non-critical paths: the idea is that any non-critical path admits a finite open cover consisting of segments to which local WKB applies, and one can use this to extend the result to longer paths.  More precisely, one proves this by induction. We assume that it is covered by open intervals each of
which go into a single apartment. (as a noncritical path in that apartment).
Choose a sequence of points $0=t_1$, ... , $t_n=1$ along the path, so that each triple is in a single neighborhood. Then we show that the path from $t_1$ to $t_i$ is in a single apartment $A$. Suppose we have done it up to $t_{i-1}$. Then there is also an apartment $A'$ containing
$f([t_{i-2}, t_i])$ . Furthermore, the path here is noncritical. Now, take the sector $S$ in $A$ which is based at $x= f(t_{i-1})$ and contains the path up to there. Then, take the sector $T$ in $A'$ which is based at $x= f(t_{i-1})$ and contains the segment $f([t_{i-1},t_i)$. The claim is that these two sectors have germs at $x$ which are opposite. Therefore, $S\cup T$ is in a single apartment $A^{\flat}$,  and now this one contains
$f([t_1, t_{i}])$ completing the induction the induction step. This shows part three. 

The vector distance is now the sum of the vector distances between the successive points.
These sums of vector distances are also sums of vectors of integrals, because the noncritical condition insures that the order of sheets of $\phi$ by decreasing
real part is
the same all along the path. 
Therefore, the vector distance is the same as the vector of integrals, giving the second part. 
\end{proof}

The asymptotic cone occuring in Theorem \ref{cone-wkb} is a very complicated object to contemplate: it is a thick $\mathbb{R}$-building, and its singularities form a dense set. While Theorem \ref{cone-wkb} gives an abstract characterization of the WKB dilation spectrum as the distance between two points in a certain metric space, it  provides neither an effective procedure for computing, nor does it furnish an interpretation of the dilation spectrum in terms of the geometry of the spectral covering $\pi: \Sigma \rightarrow X$. 

The universal building (Definition \ref{Bu-def}) serves to remedy this situation, since conjecturally it can be constructed from the geometry of the spectral networks determined by the spectral cover. The following proposition reduces the calculation of the WKB dilation spectrum to a question about $\Bu$:

\begin{proposition}\label{cone-map-is-phi-map}
Let the notation be as in Theorem \ref{cone-wkb}, and let $\phi$ denote the spectral cover of the Higgs bundle $(\cE, \varphi)$. Then the  map $\hw: \tilde{X} \rightarrow \Bc$ of Theorem \ref{cone-wkb} is a $\phi$-map regular outside of the branch locus of $\phi$.  In particular it is harmonic. 

We may conclude that if a universal $\phi$-building exists, then there is a unique folding map of buildings $g: \Bu \rightarrow \Bc$ that makes the following diagram commute:
\[
\xymatrix{
\tilde{X} \ar[r]^{\hu} \ar[dr]_{\hw} & \Bu \ar@{-->}[d]^{g} \\
& \Bc \\
}
\]
\end{proposition}

\begin{proof}[Sketch of proof.]
The second statement clearly follows from the first and the definition of a $\phi$-map. To prove the first statement, we first claim that $\hw$ is regular away from the ramification locus of $\phi$. For this we argue as in Lemma \ref{local-non-crit}. Let $R$ be a point that is not on the ramification divisor. Then once can find points $P$ and $Q$ and non-critical paths $\gamma$ and $\gamma'$ as in Lemma \ref{local-non-crit}. Since we do not require that $\gamma$ and $\gamma'$ take values in the regular locus of any harmonic $\phi$-map, we can in fact arrange a homotopy through non-critical paths from $\gamma$ to $\gamma'$ that sweeps out a disc containing $R$. Now arguing as in the proof of Corollary \ref{non-crit-regions-to-apts}, and using Theorem \ref{cone-wkb}, we see that $D$ is mapped by $\hw$ into a single apartment. Then, the map on this disc $D$, into a single apartment, is given by the integrals of the real parts of the $1$-forms, because the integrals calculate the vector distance. In particular, our map is a $\phi$-map. 
\end{proof}


\section{The Berk-Nevins-Roberts example}\label{BNR-example}

In the seminal paper \cite{BNR} the authors introduced the notion of a ``new Stokes line'' emanating from the intersection of two (ordinary) Stokes lines, and demonstrated the significance of these new Stokes lines to study of the WKB asymptotics associated with a certain third order differential equation defined on  $X = \mathbb{A}^{1}_{\mathbb{C}}$. These new Stokes lines, and the equation studied in \emph{loc. cit}, have since been studied by various mathematicians under the names ``virtual Stokes lines'', ``Stokes geometry'' and ``virtual turning point theory'' (see \cite{Aoki-fresh, Honda} and the references therein). These virtual Stokes lines are identical with the ``collision curves'' of the spectral networks that have introduced in the physics literature (\cite{GMN-Spectral-Networks}, see also Definition \ref{BPS-network-definition}). In this section, we revisit the BNR-example from the fresh perspective afforded by this paper.

Let $x$ be a coordinate on $\mathbb{A}^{1}_{\mathbb{C}}$ and let $p$ denote the coordinate along the fiber of the cotangent bundle $T^{\vee}_{X} \simeq \mathbb{A}^{2}_{\mathbb{C}}$. The spectral cover associated to the differential equation studied in \emph{loc. cit.} is the subscheme  $\Sigma \hookrightarrow T^{\vee}_{X}$ given by the equation:

\[
p^{3} - 3p + x = 0 \\
\]

The main goal of this section is to construct the universal building $\Bu$ associated to this spectral cover, and to show that the WKB dilation spectrum for any Riemann-Hilbert WKB problem with spectral curve $\Sigma$ is computed by a vector distance in this universal building. Furthermore, we will show that the spectral network associated to $\Sigma$ can be recovered from the singularities of $\Bu$ and the map $\hu: \mathbb{A}^{1}_{\mathbb{C}} \rightarrow \Bu$.  The spectral network is shown in Figure \ref{BNR-spectral-network}.

To obtain a differential equation whose spectral curve is $\Sigma$, we ``quantize'' the spectral curve, by replacing the $p$ with the differential operator $d/dx$. That is, we deform $\Sym(T_{X})$ to $\cD_{X}$ via the Rees construction, and simultaneously deform the Higgs bundle (= $\Sym(T_{X})$-module) corresponding to the trivial line on $\Sigma$ to the $\cD_{X}$-module $\cD_{X}/P_{1}\cD_{X}$ where $P_{t}$ is the family of differential operators given by:

\[
P_{t} := \frac{1}{t^3}\frac{d^{3}}{dx^3} - \frac{3}{t}\frac{d}{dx} + x \\
\]

Let $f$ be a holomorphic function on $\mathbb{C}$, and let $g = f'$ and $h = g'$, and let $s \in H^0(X, \cO^{3}_{X})$ be the section given by $s = (f,g,h)$. Then the family of differential equation $P_{t}f = 0$ is equivalent to the family of first order ODEs given by $\nabla_{t}s = 0$, where $\nabla_{t}$ is the family of (automatically integrable) connections on $\cE \simeq \cO^3_{X}$ given by 

\[
\nabla_{t} = d - t \left( \begin{array}{ccc}
0 & dx & 0 \\
0 & 0 & dx \\
-xdx & 3dx & 0 \\ \end{array} \right)
\]

\begin{remark}
The differential equation studied in \cite{BNR} is slightly different from the one above. Instead of replacing $p$ by $d/dx$ in the equation of the spectral curve as above, the authors quantize the spectral curve by replacing $p$ by $id/dx$, as is customary in Quantum Mechanics.
\end{remark}

The main goal of this section is to construct the universal $\phi$-building, and to show that it computes the WKB dilation spectrum of the family of connections defined above. The construction of this building and the proof of its universality is divided over the next three subsections. Before diving into details of the construction/proof, here is an outline of the strategy which may help orient the reader:

\begin{outline}\label{BNR-outline} The construction of $\B$ consists of the following steps:

\begin{itemize}

\item[\emph{Step 1.}] Maximal abelian regions in $X$ are the maximal regions where the local WKB approximation holds. We use heuristic reasoning to identify these regions, and then construct a pre-building $\Bpre$, by ``gluing in an apartment for every maximal abelian region''. The pre-building $\Bpre$ comes equipped with a natural map $\hpre: X \rightarrow \Bpre$. There is unique vertex $\{ o \}$ in $\Bpre$ that lies in the 0-dimensional stratum of its singular locus. This step is carried out in $\S$ \ref{MAR-section}.

\item[\emph{Step 2.}] We construct a building $\Bu$ and an isometry $i: \Bpre \rightarrow \Bu$. Recall that a map of pre-buildings is an isometry if it restricts to an isometry on each apartment. The building is constructed as a cone over a certain spherical building. This spherical building is in turn constructed as the ``free spherical building'' generated by the link of the vertex $\{ o \}$ in $\Bpre$. This step is the content of $\S$ \ref{octogon}.

\item[\emph{Step 3.}] We show that $i: \Bpre \rightarrow \Bu$ has the following universal property: given any building $\B$ and an isometric embedding $j: \Bpre \hookrightarrow \B$ there exists a unique folding map of buildings $\psi: \Bu \rightarrow \B$ such that $\psi \circ i = j$. 

\item[\emph{Step 4.}] Finally, we show that for any $\phi$-map $f: X \rightarrow \Bu$ there exists a unique isometric embedding $j: \Bpre \rightarrow \B$ such that $j \circ \hpre = f$. Combining this with the previous steps, it follows immediately that $\Bu$ is the universal $\phi$-building, and furthermore that the WKB exponents for any $\phi$ WKB problem are computed by $\Bu$. Steps 3. and 4. are carried out in $\S$ \ref{BNR-WKB}. 

\end{itemize}
\end{outline}

\subsection{Maximal abelian regions: constructing the pre-building}\label{MAR-section}

This section is devoted to completing \emph{Step 1} of the program outlined above: the construction of the pre-building $\Bpre$. Let $\waff = W \ltimes \mathbb{R}^2$ be the affine Weyl group whose spherical part $W$ is the $A_2$ Weyl group, and let $(\aff, \waff)$ be the corresponding standard apartment. 

Figures \ref{MAR1} - \ref{MAR10} show a copy of $X \simeq \aff^{1}_{\cc}$ on the left hand side, and a copy of $\aff \simeq \mathbb{R}^{2}$ on the right hand side. The shaded area in $X$ is called a maximal abelian region (MAR for short). The reason for this terminology will be explained below. Each connected component of the complement of the spectral network has been assigned a different color.

Fix an MAR -- say MAR1. Over the interior of the MAR, we choose a section of the cameral cover. Pulling back the tautological forms on the cameral cover to the MAR, we get three $1$-forms $\phi_1$, $\phi_2$, $\phi_3$. Choose a basepoint $P$ in the MAR.  Then we get a well-defined map from the MAR, taking values in $\mathbb{R}^2$ defined by

\[
Q \rightarrow ( \int_P^Q \re \phi_1,  \int_P^Q \re \phi_2)
\]

The shaded regions in the right hand side represent the images of the corresponding regions in $X$ under the map above. Recall that $\sum_i \phi = 0$. Thus the data contained in these pictures and the calculations used to produce them is equivalent to the data contained in the natural map from the MAR to $\aff$ defined by 

\[
Q \rightarrow ( \int_P^Q \re \phi_1, \int_P^Q \re \phi_2, \int_P^Q \re \phi_3)
\]

The idea of the construction is to ``glue in a copy of $\aff \simeq \mathbb{R}^2$'' for each of the MARs shown in Figures \ref{MAR1} -- \ref{MAR10}. The apartments are glued by identifying points if they come from the same point on $X$. This construction was described in a more precise and formal setting in Construction \ref{gluing-pre-building}. 

\begin{definition}\label{bnr-pre-building}
The BNR-pre-building $\Bpre$ and the map $\hpre: X \rightarrow \Bpre$ are obtained by applying Construction \ref{gluing-pre-building} to the admissible cover $\cU$ consisting of the 10 regions $\{\mathrm{MAR1}, ..., \mathrm{MAR10} \}$ shown in Figures \ref{MAR1} - \ref{MAR10}. 

\end{definition}

The rest of this section is devoted to explaining the heuristics underlying the construction of the 10 regions MAR1 -- MAR10, which has as its basis ``WKB considerations''. 

	Let $P,Q \in X$, let $\gamma:[0,1]\rightarrow X$ be a path connecting $P$ and $Q$ and let $\cW_{\pi/2}$ be the imaginary-spectral network. For any point $\gamma(t_k) \in \gamma([0,1])\cap \cW_{\pi/2}$, there is a corresponding \textit{detour integral} 
	\[ D_k=\re\left(\int_P^{\gamma(t_k)}\lambda_1 + \int_{\gamma(t_k)}^Q \lambda_2\right)  \:, \]
	with $\lambda_1\left(\gamma\left(t_k\right)\right)=\lambda\left(\gamma'\left(1\right)\right), \lambda_2\left(\gamma\left(t_k\right)\right)=\lambda\left(\gamma'\left(0\right)\right)$, $\gamma'$ being the path from the Definition \ref{BPS-network-definition}.

	
	  We say that the detour $D_k$ \textit{dominates along $\gamma$} if $\re (D_k)\geq  \re I_{1,2}=\re \int_{\gamma}\lambda_{1,2}$.
	We say $P$ and $Q$ are \textit{not simply WKB related by $\gamma$} if there exists a detour $D_k$ that dominates along $\gamma$. Otherwise we say that $P$ and $Q$ are \textit{simply WKB-related by $\gamma$}.\\
	Furthermore, we say that a connected subset $U\subset X$ is an \textit{abelian region} if for any path $\gamma$ contained in $U$, the endpoints are simply WKB-related by $\gamma$. An abelian region $M$, maximal with respect to inclusion is called a \textit{Maximal Abelian Region} (MAR).

	 \begin{remark}
	 	The detours defined above are related to the detours in \cite{GMN-Spectral-Networks} as follows: $D_k$ can be rewritten as
		\begin{equation}\label{Detour}
		 D_k = \re\left( \int_P^{\gamma(t_k)} \lambda_1 + \int_{\gamma'} \lambda + \int_{\gamma(t_k)}^Q \lambda_2   \right) \:,
		 \end{equation}
		where $\gamma'$ is a path as in Definition \ref{BPS-network-definition}. Note that (\ref{Detour}) makes sense for all spectral networks $\mathcal W_\theta$, thus we obtain $D_k^\theta$ . This looks qualitatively as in Figure (\ref{Dtheta}). 
			\begin{figure}[h]
	\centering
	\includegraphics[width=.5\textwidth]{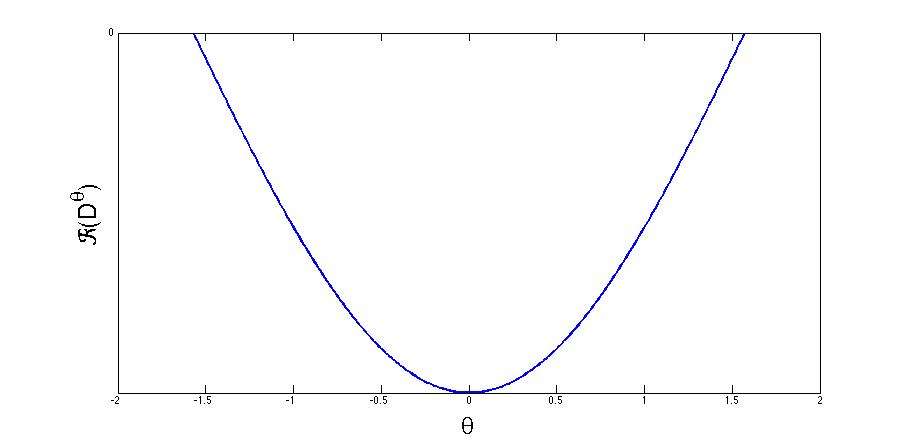}
	\caption{Detour integral}
	\label{Dtheta}
	\end{figure}
	On the other hand side, the spectral networks $\mathcal W_{\pi/2}$ and $\mathcal W_{-\pi/2}$ are the same up to orientation. This leads to\\
	 \end{remark}
	 
	 \begin{hproposition}\label{Sweeping}
	 Let $\gamma:[0,1]\rightarrow X$ be a path. Then no detour dominates along $\gamma$ if and only if $\gamma$ is homotopy equivalent to a path that does not intersect two lines $s_{1}\neq s_2 \subset \mathcal W_{\pi/2}$ such that $s_1$ is rotated into $s_2$ as $\theta$ goes from $\pi/2$ to $-\pi/2$.\
	 \end{hproposition}
	 
	 \begin{remark}
	 	Some care has to be taken when a spectral network line $s$ sweeps across a branch point. See Example \ref{ExMar}.
	 \end{remark}
	 	\begin{corollary}
		Let $M$ be an MAR. Then $M$ is a union of connected components of the complement of the imaginary spectral network.
	\end{corollary}
	
	\begin{example}[MAR without collision point]
		Consider the situation of Figure(\ref{SL2Mar}). We use the following convention for spectral network lines: if a line is labeled (ij), one letter on each side of the line, this means that $\re \int_{P'}^{Q'} (\lambda_j-\lambda_i)>0$ if $P'$ lies on the side containing $i$ and $Q'$ lies on the side containing $j$. The curly line represents a branch cut.\\
		In this case, $P$ and $Q$ are not simply WKB related by the red path since the detour dominates: the detour is given by
		\begin{equation}\label{Detour2}
			\re \int_P^R \lambda_1+ \re \int_R^Q\lambda_2 = \re\int_P^S\lambda_1 + \re \int_S^Q \lambda_2\:.
		\end{equation}
		The equality comes from $\re \int_R ^S\lambda_1 = \re \int_R^S \lambda_2$ which follows from the definition of spectral networks. \\
		We have to compare (\ref{Detour2}) with both $\int_P^Q \lambda_{1,2}$:
		\[ 
           \re \left(\int_P^R \lambda_1 + \int_R^Q\lambda_2\right) - \re \int_P^Q \lambda_1
		=\re\int_S^Q(\lambda_2-\lambda_1)>0
		\]
		Similarly,
		\[\re \left(\int_P^R \lambda_1 + \int_R^Q \lambda_2\right)-\re \int_P^S \lambda_2 =\re \int_P^R(\lambda_1-\lambda_2)>0\:.\]
		
		On the other hand, $R$ and $S$ (and in fact every pair of points on the red path lying between $R$ and $S$) are simply WKB-related. \\
			
	\begin{figure}[h]
	\centering
	\includegraphics[width=.3\textwidth]{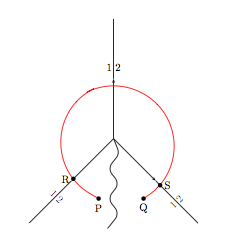}
	\caption{Estimating the WKB exponents using spectral networks}
	\label{SL2Mar}
	\end{figure}
	
	This can also be easily seen from the criterion in Remark $\ref{Sweeping}$.
	
	\end{example}
	
	\begin{example}[Example with collision points]\label{ExMar}
	Consider the situation of Figure \ref{SL3Mar}: we are interested in comparing the detour $D=\int_P^R\lambda_1 + \int_R^Q \lambda_2$ with $I=\int_P^Q \lambda_1$. Firstly, note that (with $\lambda_{12}=\lambda_1-\lambda_2$)
	\[ \re \int_R^S \lambda_{12}=0 \:,\]
	since it can be deformed to the green contour which moves only along spectral networks: along I, IV and V, the integral vanishes; the contributions from II and III cancel each other.
	Now,
	\[\re (D-I) =\re \int_R^Q \lambda_{12}=\re \int_S^Q\lambda_{12}>0 \:.\]
	Similar computations can be made for the other detours.\\
	The conclusion is that $P$ and $Q$ are not simply WKB related via the red path (and in fact are not simply WKB related at all).
	
	\begin{figure}[h]
	\centering
	\includegraphics[width=.8\textwidth]{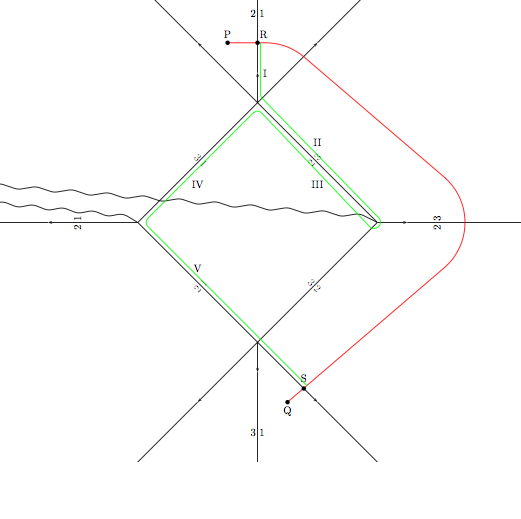}
	\caption{Reconstructing Maximal Abelian Regions (MARs) from spectral networks}
	\label{SL3Mar}
	\end{figure}

	This can also be seen from Proposition \ref{Sweeping}: the spectral network line containing $c$ is rotated into the spectral network line containing $d$, as we let $\theta$ vary from $\pi/2$ to $-\pi/2$ (see Figure \ref{Wmovie}). Note that the collision line rotates into the red line in the last picture. 
	
	\begin{figure}[h]
	\centering
	\includegraphics[width=.8\textwidth]{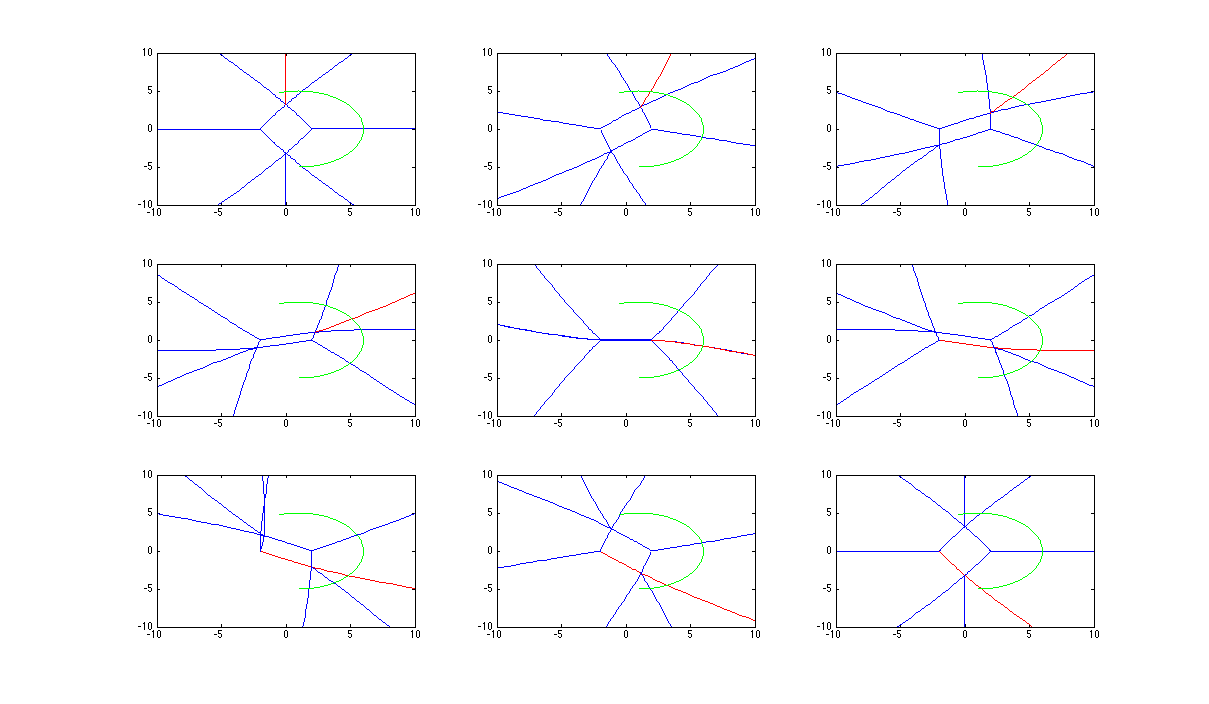}
	\caption{Spectral networks for varying values of $\theta$}
	\label{Wmovie}
	\end{figure}

	\end{example}

\begin{figure}[h]
\centering
\includegraphics[width=.95\textwidth]{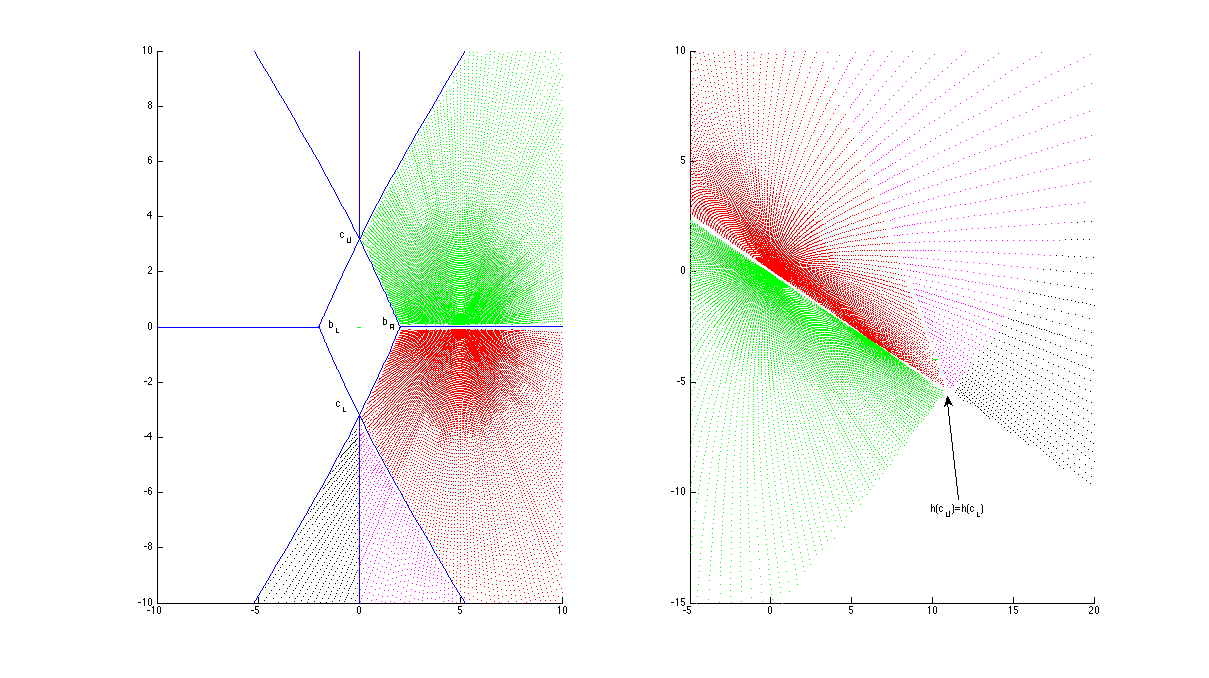}
\caption{MAR1}
\label{MAR1}
\end{figure}

\begin{figure}[h]
\centering
\includegraphics[width=.5\textwidth]{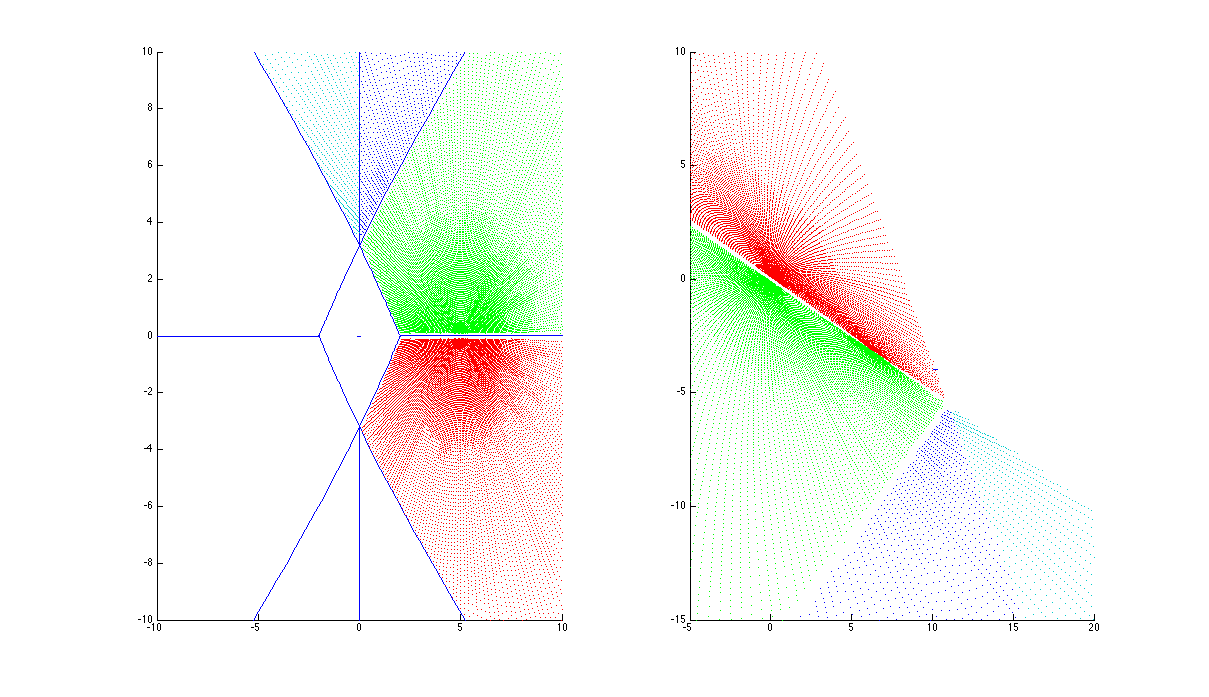}
\caption{MAR2}
\label{MAR2}
\end{figure}

\begin{figure}[h]
\centering
\includegraphics[width=.5\textwidth]{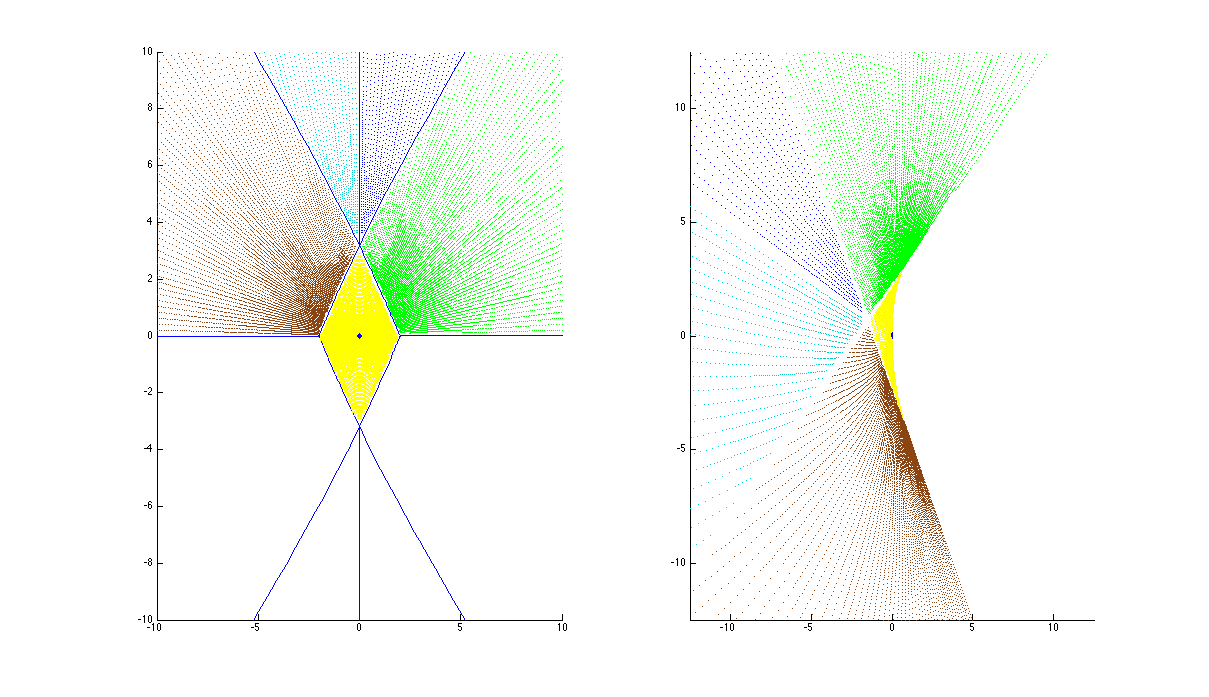}
\caption{MAR3}
\label{MAR3}
\end{figure}

\begin{figure}[h]
\centering
\includegraphics[width=.5\textwidth]{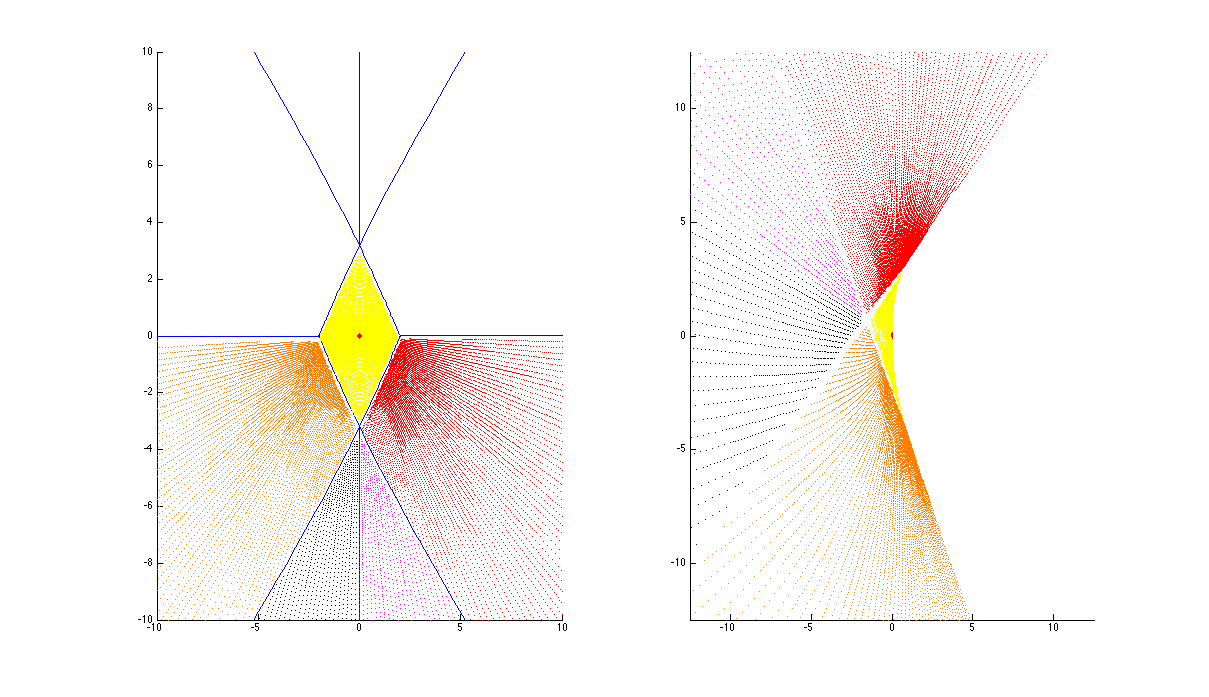}
\caption{MAR4}
\label{MAR4}
\end{figure}\begin{figure}[h]

\centering
\includegraphics[width=.5\textwidth]{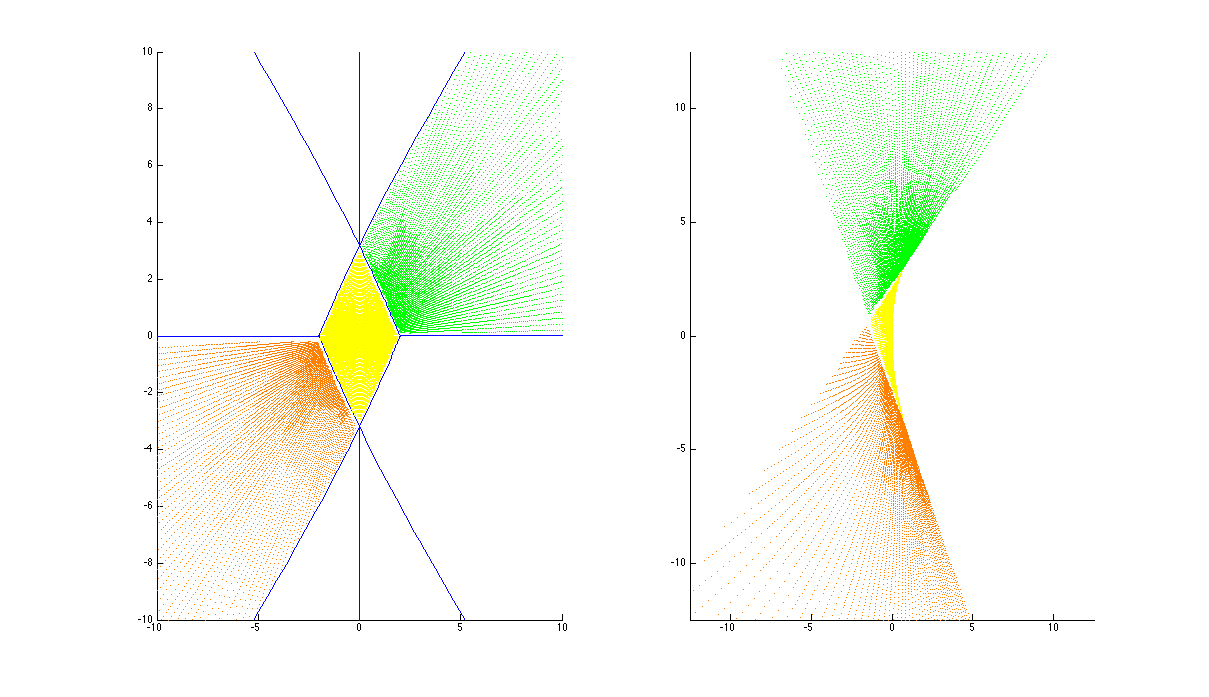}
\caption{MAR5}
\label{MAR5}
\end{figure}

\begin{figure}[h]
\centering
\includegraphics[width=.5\textwidth]{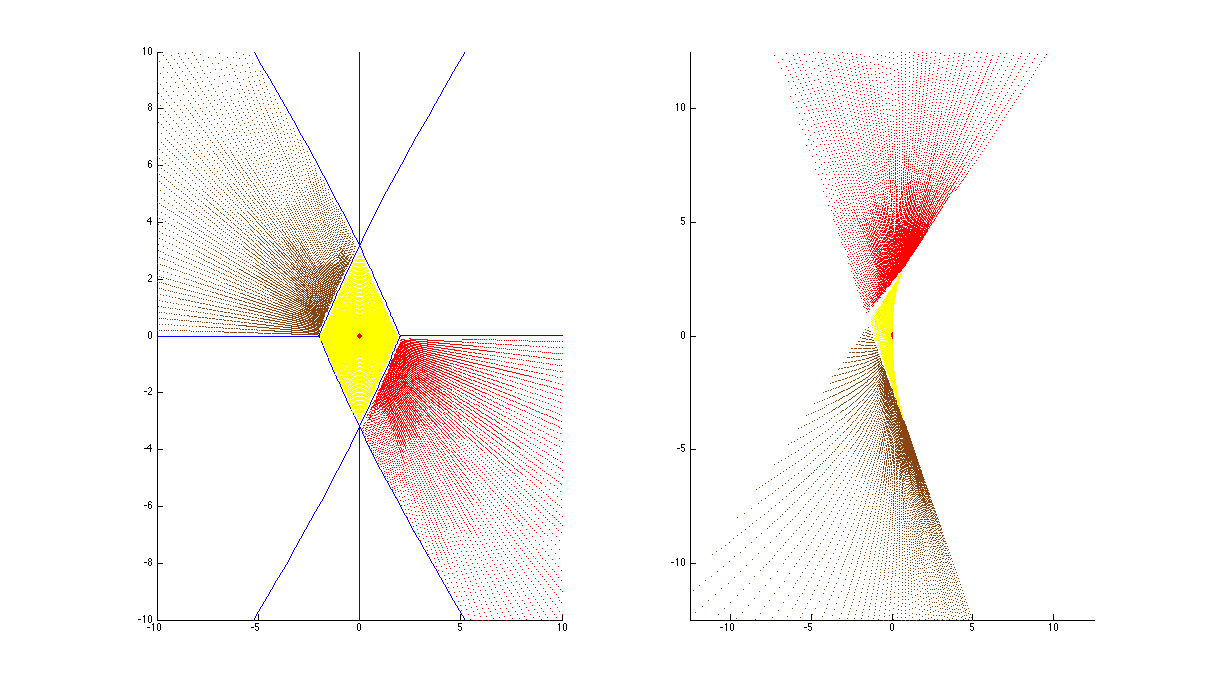}
\caption{MAR6}
\label{MAR6}
\end{figure}

\begin{figure}[h]
\centering
\includegraphics[width=.5\textwidth]{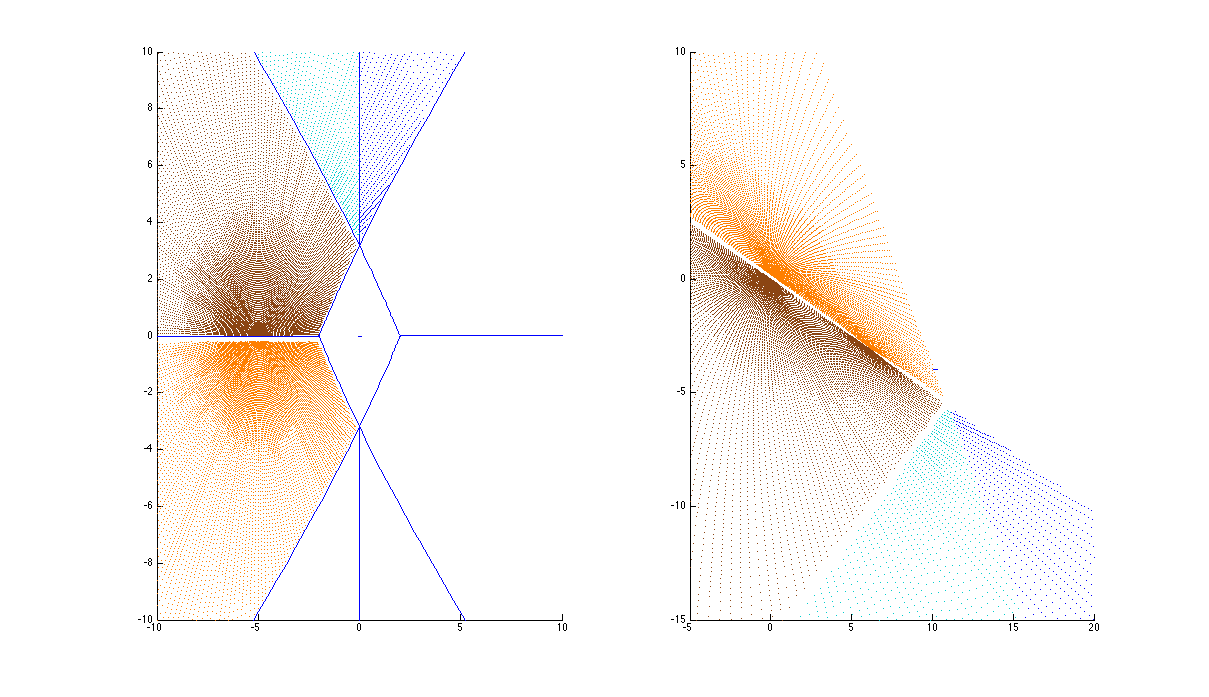}
\caption{MAR7}
\label{MAR7}
\end{figure}

\begin{figure}[h]
\centering
\includegraphics[width=.5\textwidth]{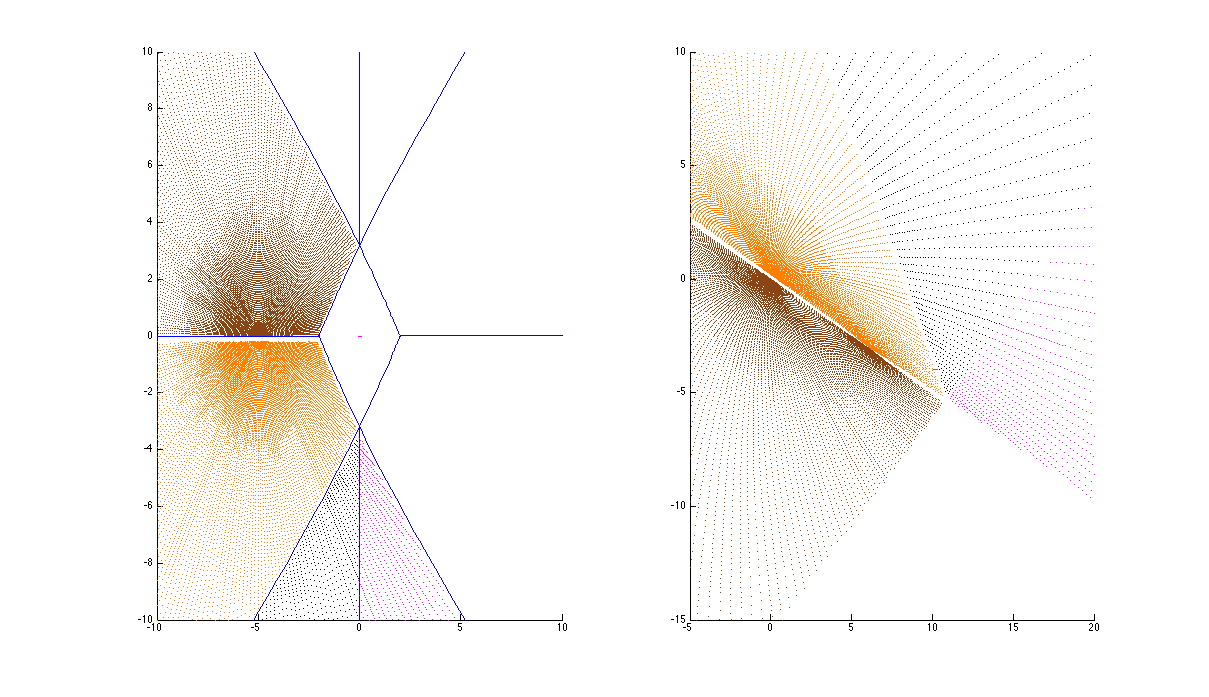}
\caption{MAR8}
\label{MAR8}
\end{figure}

\begin{figure}[h]
\centering
\includegraphics[width=.5\textwidth]{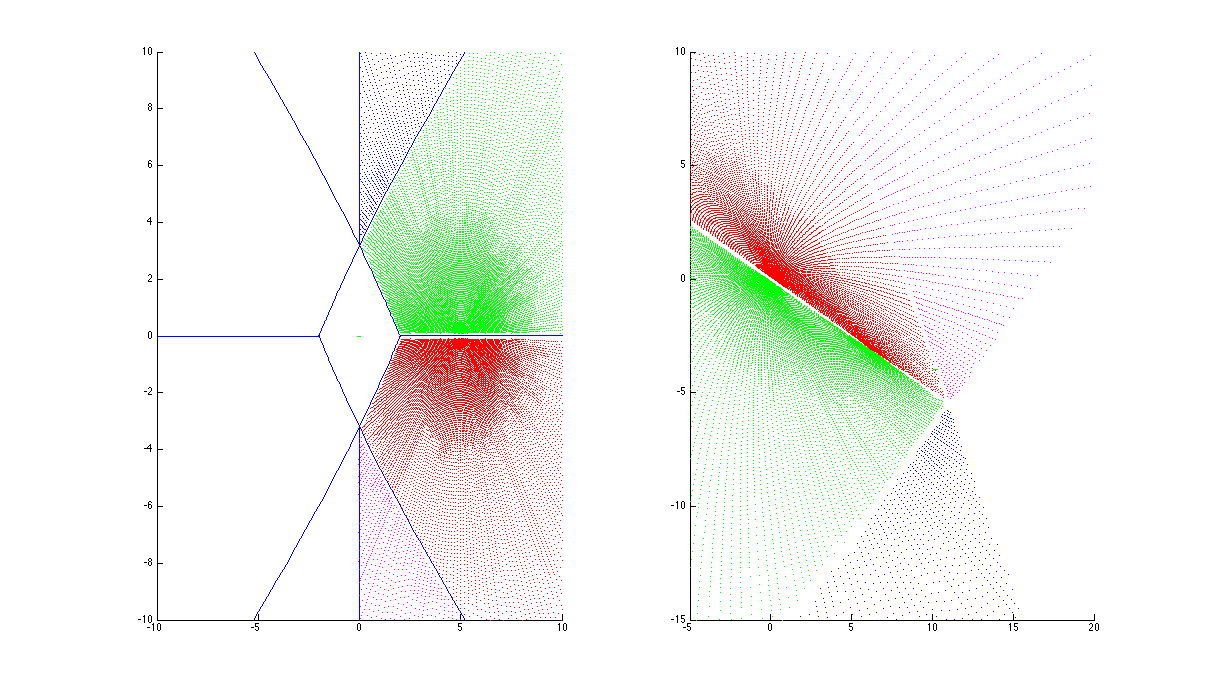}
\caption{MAR9}
\label{MAR9}
\end{figure}

\begin{figure}[h]
\centering
\includegraphics[width=.5\textwidth]{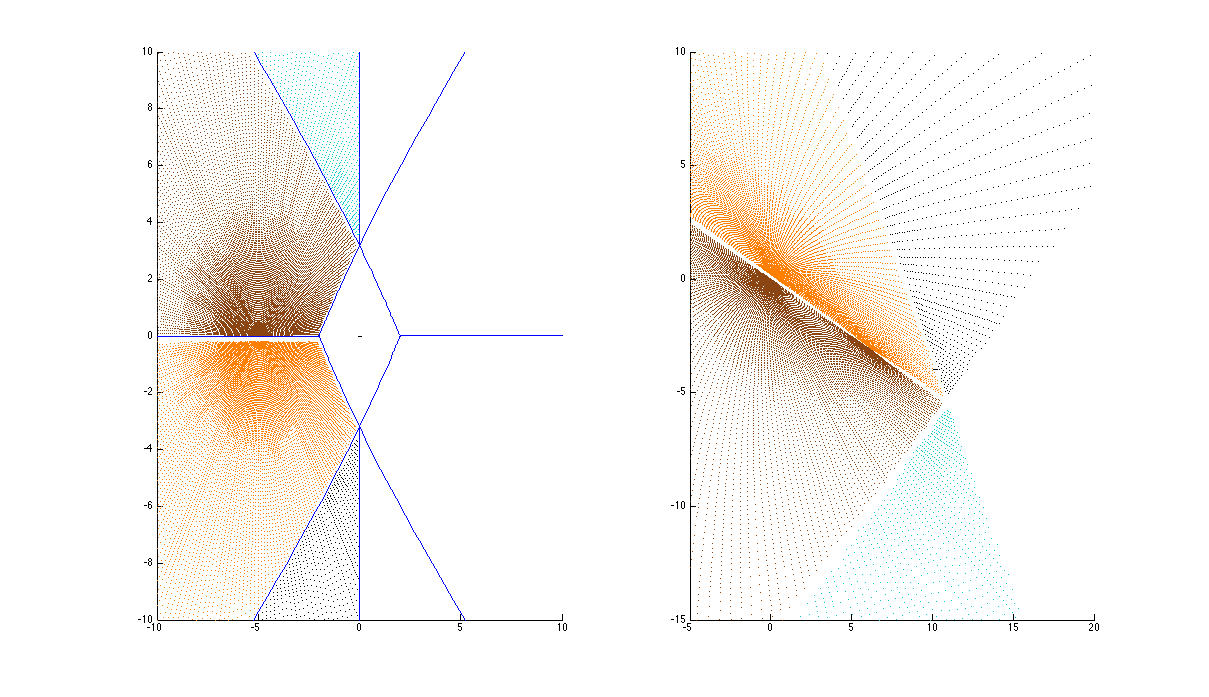}
\caption{MAR10}
\label{MAR10}
\end{figure}

\FloatBarrier

\subsection{The universal building as a cone}\label{octogon}

In this subsection, we will complete the pre-building $\Bpre$ constructed in the previous subsection to a building, following the strategy outlined in the introduction. The main proposition of this subsection is Proposition \ref{completion-universal-property}. First, we need some preliminary lemmas, which will be used in this subsection as well as the next. 

\begin{lemma}\label{attaching-sectors-prelim}
		Let $A_+$ be a closed half apartment in $B$ bounded by a wall $H$, and let $S$ be a sector with vertex $x$ such that $S\cap H=P$ is a panel of $S$ (in particular $x \in H$) ,
and such that the germ $\Delta _xS$ is not contained in $A_+$. Then $S$ is opposite to some sector $S^-$ of $A_+$.
	\end{lemma}
	\begin{proof}
	Let $A$ be an apartment containing $A_+$. Consider the link $G$ of $B$ at $x$. This is a spherical building. Then (keeping only $A$ and $S$) the situation is as in Figure \ref{link-attachingsectors}, where $S^-\subset A_+$. From this, it is clear that $S$ has to be opposite to $S^-$, otherwise the girth would be less than six.
	
	\begin{figure}[h]
	\centering
	\includegraphics[width=.5\textwidth]{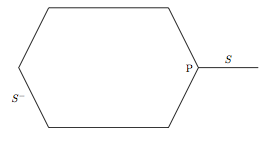}
	\caption{Part of link near $x$}
	\label{link-attachingsectors}
	\end{figure}
	
	\end{proof}	

  	\begin{lemma}\label{attaching-sectors}
Let $A_+$, $H$,$P$ and $S$ be as in  Lemma \ref{attaching-sectors-prelim}. Then there exists
an apartment containing $A_+\cup S$. 
\end{lemma}

	\begin{proof}
		Let $A$ be an apartment containing $A_+$ as a half-apartment. By Lemma \ref{attaching-sectors-prelim}, we can use (CO) from (\ref{properties})  to conclude that $S$ and $S^-$ are contained in a common apartment $A'$. Then $P\cup S^- \subset A\cap A'$, thus $A_+ \subset A\cap A'$, by convexity.
	\end{proof}

\begin{corollary}\label{four-sectors}
In the $A_2$ case i.e. for $SL_3$ buildings, suppose $S_1,S_2,S_3,S_4$ are sectors based at a single point $x$, such that
$S_i$ and $S_{i+1}$ share a common face for $i=1,2,3$. Suppose that these successive common faces are distinct.
Then $S_1,S_2,S_3,S_4$ are contained in a common apartment.
\end{corollary}
\begin{proof}
We show this by induction on $i$. For $i=1$ it is easy. Suppose $i\leq 4$ and we have shown it up to $i-1$, that is to
say we have an apartment $A'$ containing $S_1,\ldots , S_{i-1}$. These sectors satisfy the same adjacency condition within $A'$,
from which it follows that they are successive sectors aranged around the vertex $x$. Now, $R:= S_i\cap S_{i-1}$ is the
face of $S_{i-1}$ (ray) which is different from $S_{i-2}\cap S_{i-1}$. Let $H$ be the half-apartment of $A'$ whose boundary
contains $R$, and which contains $S_{i-1}$. Then, $H$ contains $S_1,\ldots , S_{i-1}$. Indeed, if some previous $S_j$ were
not in $H$ then its boundary would have to contain $R$. (Here is where we use $i-1\leq 3$, to say that a previous $S_j$ cannot
leave $H$ along the other ray in $\partial H$.) Now apply the previous lemma: we get an apartment $A$ containing $H\cup S_i$,
so $A$ contains $S_1,\ldots , S_i$. This completes the inductive step.
\end{proof}

We now turn to the construction of the universal building. Consider the pe-building constructed in the previous subsection by gluing MARs. Note that it has a special vertex $\{o\}$, namely, the image of both the collision points. Now consider the link $G$ of the pre-building $\Bpre$ at the vertex $\{o\}$. This is shown in Figures \ref{octagon-picture} and \ref{adding-sectors-octagon}

\begin{construction}\label{BNR-Bu-construction}
Let $G$ be the link at the vertex $\{o\}$ of the building. We can view $G$ as a bipartite graph, by coloring its alternate vertices black and white. Write $bwbwbw..bw$ for a sequence of black and white vertices connected by edges. Now define a bipartite graph inductively as follows. Let $G = G_0$. Having defined $G_{n-1}$, we define $G_n$ by gluing in a pair of edges $bwb$ (resp. $wbw$) to every sequence of $4$ edges $wbwbw$ (resp. $bwbwb$) that is not contained in a hexagon. See Figure \ref{add-to-octagon}. The idea is that the link of the affine building that we want to construct should be a spherical building of type $A_2$, and the apartments in a spherical building of type $A_2$ are hexagons. The first stages of this process are shown in Figure \ref{adding-sectors-octagon}. Define
\[
\B^{8}: = \cup_{n=0}^{\infty} G_n
\]

\end{construction}

	\begin{figure}[h]
	\centering
	\includegraphics[width=.5\textwidth]{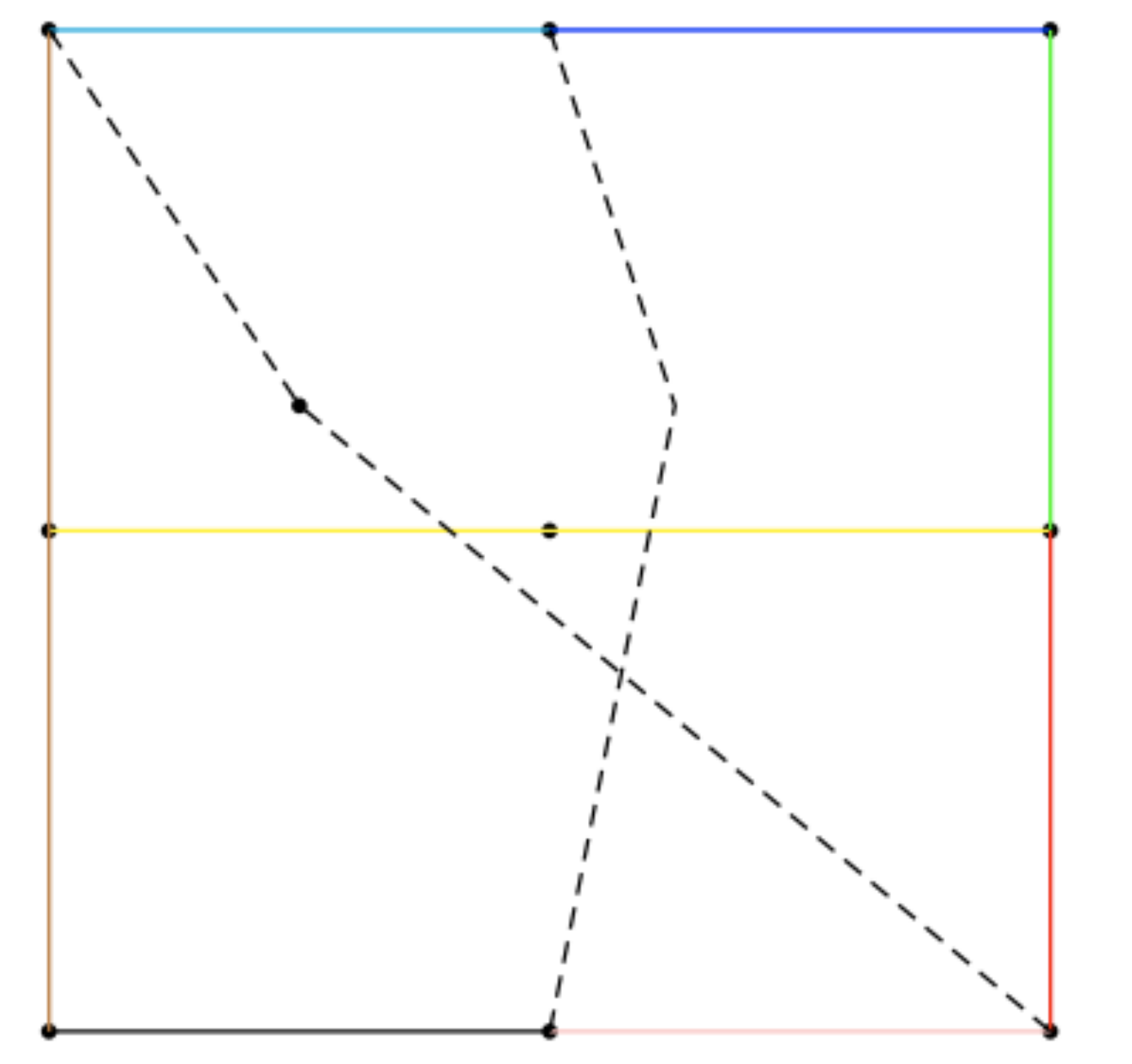}
	\caption{Adding sectors to the link of the BNR pre-building}
	\label{adding-sectors-octagon}
	\end{figure}

	\begin{figure}[h]
	\centering
	\includegraphics[width=.5\textwidth]{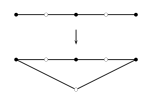}
	\caption{Completing hexagons in the bipartite graph}
	\label{add-to-octagon}
	\end{figure}

Then we have the following proposition

\begin{proposition}
There is a natural structure of a spherical building of type $A_2$ on $\B^8$.
\end{proposition}

\begin{corollary}
There  is a natural structure of a building with Weyl group $W \rtimes \mathbb{R}^2$ on $\Bu : = \Cone(\B^8)$, the cone over the spherical building $\B^8$. Here $W$ is the Weyl group of $\SL_3\cc$. 
\end{corollary}

The proposition follows immediately from the following \cite[Proposition 4.44]{Abramenko-Brown}, which says that a connected biparitite graph in which every vertex is the face of at least two edges is a building if and only if it has diameter $m$ and girth $2m$ for some $m$ with $2 \leq m \leq \infty$, in which case it is a building of type.  Ours is the case $m = 3$.

\begin{proposition}\label{completion-universal-property}
The isometry of pre-buildings $i: \Bpre \rightarrow \Bu$ has the following universal property: given any building $\B$ and an isometric embedding $j: \Bpre \hookrightarrow \B$ there exists a folding map of buildings $\theta: \Bu \rightarrow \B$ such that the following diagram commutes:
\[
\xymatrix{
\Bpre \ar[r]^{i} \ar[dr]_{j} & \Bu \ar@{.>}[d]^{ \theta}\\
& \B
}
\]
\end{proposition}

\begin{proof}
The edges in the graph (spherical building) correspond to sectors in the affine building $\Bu$. Therefore, the proposition follows easily from Corollary \ref{four-sectors}. 
\end{proof}

\subsection{The universal property and the WKB spectrum}\label{BNR-WKB}

The main theorem of this section is the following:

\begin{theorem}\label{BNR-theorem}
Let $h: X \rightarrow \B$ be a harmonic $\phi$-map to a building. Then there exists a folding map of buildings $\psi: \Bu \rightarrow \B$ that restricts to an isometry of pre-buildings on $\Bpre$, and is such that the following diagram commutes:
\[
\xymatrix{
X \ar[r]^{\hu} \ar[dr]_{h} & \Bu \ar@{.>}[d]^{\psi} \\
& \B
}
\]
Furthermore, if $\psi': \Bu \rightarrow \B$ is another folding map such that $\psi \circ \hu = h$, then $\psi'_{|\Bpre} = \psi_{|\Bpre}$. 
\end{theorem}

\begin{proof}
The theorem follows immediately from the universal property of $i: \Bpre \rightarrow \Bu$ (Proposition \ref{completion-universal-property}), and the universal property of the map $\hpre: X \rightarrow \Bpre$ described in Proposition \ref{Bpre-universal-property} below. 
\end{proof}

\begin{corollary}
The map $\hu: X \rightarrow \Bu$ computes the dilation spectrum for any WKB problem with spectral cover $\phi$.
\end{corollary}

\begin{proof}
This follows immediately from the Theorem \ref{cone-wkb}, Proposition \ref{cone-map-is-phi-map} and Theorem \ref{BNR-theorem}.
\end{proof}

We now turn to the universal property of $\hpre: X \rightarrow \Bpre$. The proof of the following proposition will occupy the rest of this section:

\begin{proposition}\label{Bpre-universal-property}
Let $\B$ be a building, and let $h: X \rightarrow \B$ be any harmonic $\phi$-map. Then there exists a unique isometry of pre-buildings $\psi: \Bpre \rightarrow \B$ such that the following diagram commutes:
\[
\xymatrix{
X \ar[r]^{\hpre} \ar[dr]_{h} & \Bpre \ar[d]^{\exists ! \psi} \\
& \B \\
}
\]
\end{proposition}

 The strategy of the proof of the proposition is as follows. We will first argue that every connected component of the complement of the spectral network is mapped into a single sector by any $\phi$-map from $X$ to a building. We will then observe that in any building with vectorial Weyl group of type $A_{2}$, any four adjacent sectors are contained in a single apartment. Combining these observations, it follows easily that each of the ten regions shown in Figures \ref{MAR1} -- \ref{MAR10} is mapped into a single apartment by any $\phi$-map. Using this, together with the fact that $\Bpre$ is constructed by gluing in an apartment for each of the ten regions MAR1 - MAR10, it is straightforward to show that there exists a unique factorization $\psi$ as claimed in the proposition. We now turn to the details of this argument:

\begin{lemma}\label{caustic-argument}
Let $R_{0}$ denote the closure of the connected component of the complement of the BNR spectral network that is at the center of the diagram (see Figure \ref{MAR1}). Let $h: X \rightarrow \B$ be a harmonic $\phi$-map to a building $\B$. Then there is a sector $S$ in $\B$ such that $h(U) \subseteq S$.
\end{lemma}
\begin{proof}
Let $\epsilon$ be a small positive real number, let $P_{\epsilon} = -2 + \epsilon$ and $Q_{\epsilon} = 2 - \epsilon$, and let $\Omega_{\epsilon} := \Omega_{P_{\epsilon} Q_{\epsilon}}$ be the region defined in Corollary \ref{non-crit-regions-to-apts} (see Figure \ref{caustic}).  Then by  Corollary \ref{non-crit-regions-to-apts}, each of the regions $\Omega_{\epsilon}$ is mapped into a single apartment. Since we can make $\epsilon$ arbitrarily small, it follows that $R_0$ is mapped into a single apartment $A$.

Since $h$ is a $\phi$-map, away from the ramification point we can choose coordinates on the apartment $A$ such that $dh = (\phi_1,..., \phi_r)$ where ``$\phi = (\phi_1,...,\phi_r)$ upto some permutation''. It follows that the foliation lines are the pullbacks of the hyperplanes defining the apartment. From this one sees easily that $R_0$ maps into the intersection of sectors $S_l$ and $S_r$ in $A$ based at $h(b_l)$ and $h(b_r)$ respectively. Here $b_l$ is the branch point $-2$ and $b_r$ is the branch point $2$. Furthermore, the segments of the spectral network lines that consitute the boundary of $R_0$ map to the boundaries of these sectors. 
\end{proof}

	\begin{figure}[h]
	\centering
	\includegraphics[width=.5\textwidth]{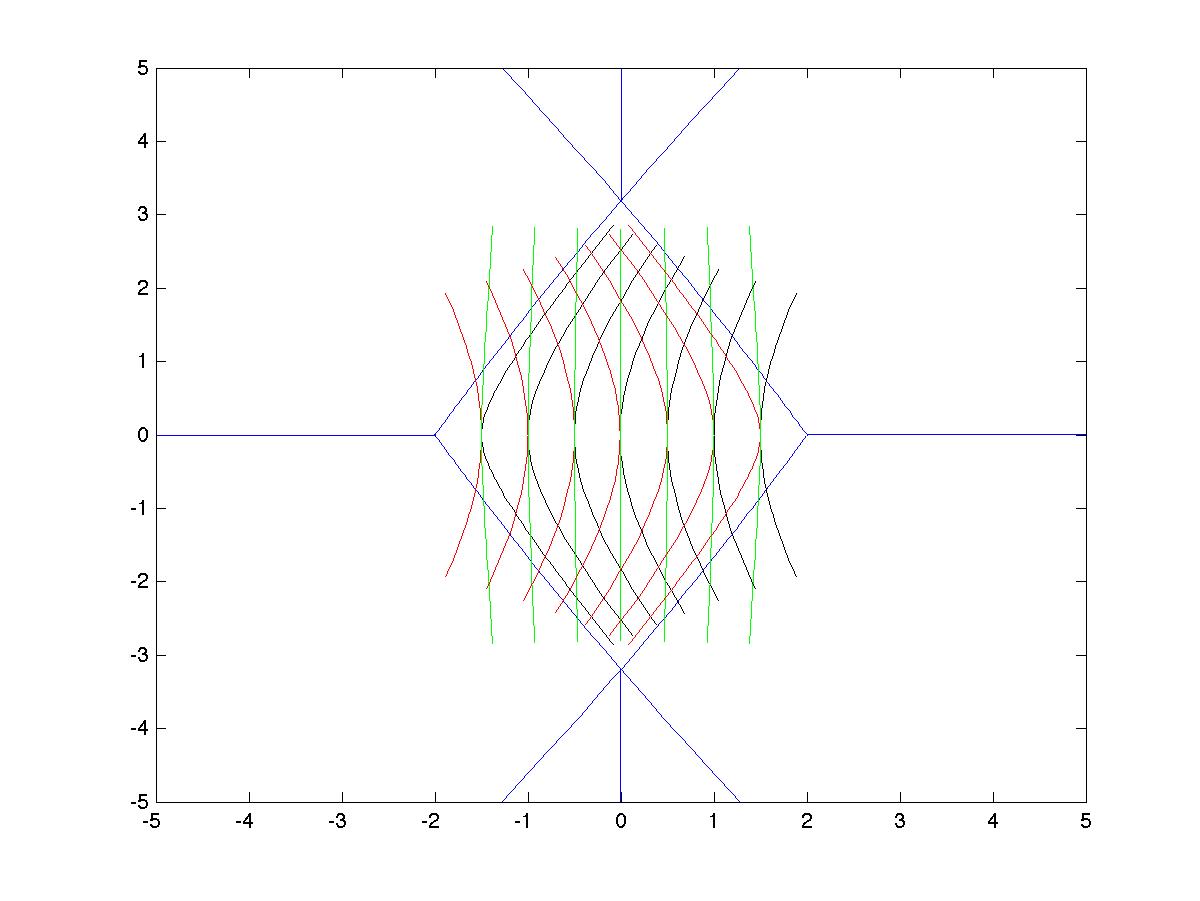}
	\caption{The Caustic Line}
	\label{caustic}
	\end{figure}

\begin{lemma}\label{collision-points-equalize}
Let $h: X \rightarrow \cB$ be a harmonic $\phi$-map to a building. Let $c_u$ and $c_l$ denote the collision points of the BNR spectral network. Then $h(c_u) = h(c_l)$.  
\end{lemma}

\begin{proof}
We know that the central yellow region maps to a single apartment. So to prove the lemma it suffice to compute the integrals of the $1$-forms along a contour (that stays within the central region) from one collision point to the other. More precisely, we can choose a section of the cameral cover over $R_0$, and use this to write $\phi = (\phi_1, \phi_2, \phi_3)$, and then integrate the three $1$-forms $\phi_1$, $\phi_2$, $\phi_3$ along a contour joining the collision points. One checks that this integral is $(0,0,0)$ (This follows easily from the definition of spectral networks, and the fact that $\sum_i \phi = 0$). Thus, the vector distance between $h(c_u)$ and $h(c_l)$ is zero. The lemma follows. 
\end{proof}

\begin{remark}
The proof of the lemma actually shows that $h(x) = h(\overline{x})$ for all $x$ in $R_0$, where $\overline{x}$ denotes the complex conjugate of $x$. Thus, the segment of the $[-2,2] \times \{0\}$ is a ``fold-line'' or caustic for any $\phi$-map. 
\end{remark}

\begin{lemma}\label{regions-sectors}
Let $R_{0}$,...,$R_{9}$ denote the closures of connected components of the complement of the BNR spectral network, and let $h: X \rightarrow \cB$ be a $\phi$-map to a building. Then there exists apartments $A_{1}$,...,$A_{10}$, and sectors $S_{i} \subseteq A_{i}$ such that $h(R_i) \subset S_i$. Furthermore, if two regions $R_i$ and $R_j$ are adjacent (in the sense that their intersection contains an open segment in a spectral network curve), then the corresponding sectors $S_i$ and $S_j$ are adjacent in $\B$, i.e., they share a panel. 

\end{lemma}

\begin{proof}
We have already proven the lemma for $R_0$ (Lemma \ref{caustic-argument}). We will describe the proof for one of the regions $R_1$ through $R_8$; the other cases are similar. Consider the region $R_1$ which contains the point $q$ shown in Figure \ref{region2}, and consider the points $p$ and $r$ shown in the figure (they lie on spectral network lines). 

Apply Corollary \ref{non-crit-regions-to-apts} with $P = p$ and $Q = r$. The interior of the region $U_{pr}$ bounded by the two spectral network lines containing $p$ and $r$ on the one side, and the thick red and thick black foliation lines on the other, does not contain any singularities of $h$, by Proposition \ref{phi-map-singularities}, since it does not contain any ramification point. Using this, it is easy to see that $\Omega_{PQ}$, in the notation of Corollary \ref{non-crit-regions-to-apts}, equals $U_{pr}$. It follows that this entire region is mapped into a single apartment. Since the inverse images of apartments are closed, the closure of this region also maps to a single apartment. Thus, we see that the region $R_1$ can be exhausted by a family of compact sets, each of which maps to a single apartment. Since we require our buildings to have a complete set of apartments (Definition \ref{complete-system}), it follows that the entire region $R_1$ is mapped into a single apartment $A$. Arguing exactly as in Lemma \ref{caustic-argument}, we see that $R_1$ must in fact be mapped to a single sector $S_1$ with vertex at $h(c_u)$, the image of the upper collision point. 

Furthermore,  the spectral network lines emanating from $c_u$ (resp. $c_l$)  map to panels in the building based at $h(c_u)$. This immediately implies the last statement of the lemma. 
\end{proof}

	\begin{figure}[h]
	\centering
	\includegraphics[width=.5\textwidth]{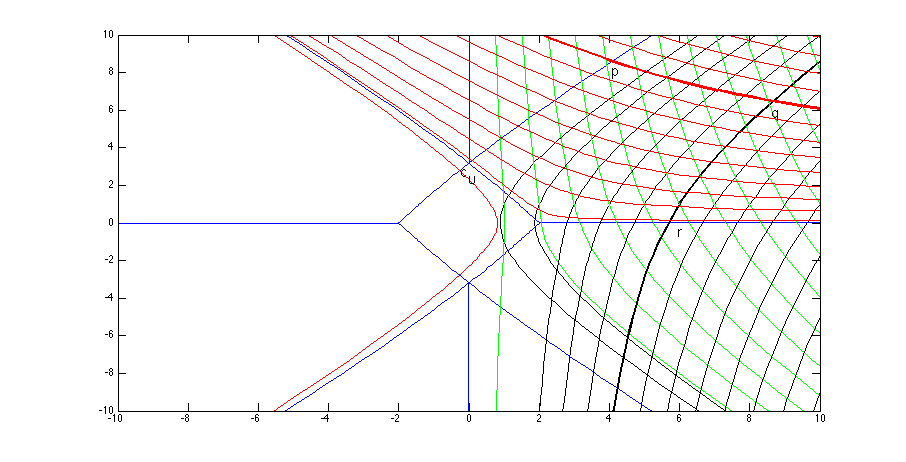}
	\caption{The foliation lines in a region}
	\label{region2}
	\end{figure}

	\begin{figure}[h]
	\centering
	\includegraphics[width=.5\textwidth]{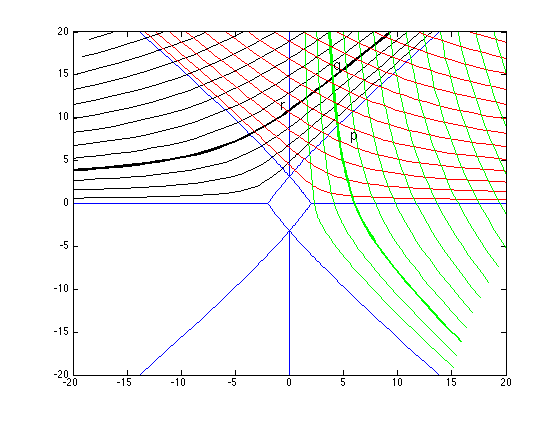}
	\caption{The foliation lines in another region}
	\label{region3}
	\end{figure}

\begin{proof}[Proof of Proposition \ref{Bpre-universal-property}]
Let $h: X \rightarrow \B$ be a $\phi$-map. We claim that for each MAR $M$ in $X$, there exists an apartment $A$ in $\B$ such that $h(M) \subset A$. The proposition follows immediately from this claim and Proposition \ref{gluing-proposition}. 

We now prove the claim. From Lemma \ref{regions-sectors} and \ref{caustic-argument} we see that every sequence of ``four adjacent regions'' on $X$ is mapped to a sequence of four adjacent sectors in the building under any $\phi$-map. By Corollary \ref{four-sectors} every such four sectors are contained in a single apartment. Thus, we have shown that any $\phi$-map carries each of the MARs that consists of at most 4 of the regions $R_0$ through $R_9$ into a single apartment. 

It remains to prove the claim for MAR3 and MAR4. Note that both of these MARs contain the yellow region $R_0$ at the center of the picture. Let us consider one of these MARs, say MAR3. The argument for the other MAR is identical. The argument of the previous paragraph shows that there is an apartment $A$ such that the complement of $R_0$ in MAR3 is mapped to the union of four adjacent sectors in $A$. 

Let $\epsilon$ be a small positive number and let $P_{\epsilon} = -2 + i \epsilon$ be a point just above the branch point $b_l$ and let $Q_{\epsilon} = 2 + i  \epsilon$ be a point just above the branch point $b_r$. Apply Corollary \ref{non-crit-regions-to-apts} with $P = P_{\epsilon}$ and $Q = Q_{\epsilon}$. Then the region $\Omega_{P_{\epsilon}Q_{\epsilon}}$ (in the notation of Corollary \ref{non-crit-regions-to-apts}) intersects $R_0$ in a the shaded region shown in Figure \ref{R0argument}. By Corollary \ref{non-crit-regions-to-apts}, every point in this region is in the Finsler convex hull of $h(P_{\epsilon})$ and $h(Q_{\epsilon})$. Since $h(P_{\epsilon})$ and $h(Q_{\epsilon})$ are contained in $A$, it follows that the entire shaded region is mapped into $A$. Since $\epsilon$ can be made arbitrarily small, and the inverse images of apartments under continuous maps are closed we see that all of $R_0$ is mapped into $A$. Thus, the entire maximal abelian region MAR3 is mapped into $A$. This completes the proof.\end{proof}

	\begin{figure}[h]
	\centering
	\includegraphics[width=.5\textwidth]{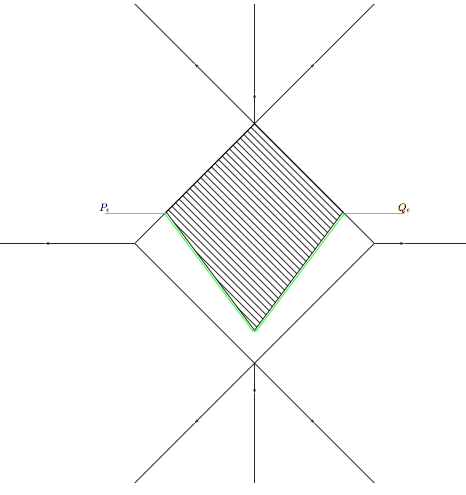}
	\caption{Showing that MAR3 is mapped to a single apartment}
	\label{R0argument}
	\end{figure}

One can immediately see from the construction of $\Bu$ that an appropriate version (see Caveat \ref{spec-net-caveat}) of Conjecture \ref{spectral-networks-Bu} holds in the BNR example:

\begin{proposition}\label{Bu-sing-BNR}
Let $\cW_{\bnr}$ denote the BNR spectral network, and let $\cW^{\mathrm{ext}}_{\bnr}$ be the extended spectral network, obtained as the union of the BNR spectral network with the ``backward collision line'' (the vertical foliation line joining the collision points). Let $\hu: X \rightarrow \Bu$ be the universal harmonic map to a building. Then we have the following:

\begin{enumerate}
\item The inverse image under $\hu$ of the singular set of $\Bu$ contains the spectral network $\cW_{\bnr}$.

\item The inverse image under $\hu$ of the singular set of $\Bu$ equals the extended spectral network $\cW^{\mathrm{ext}}_{\bnr}$. 

\end{enumerate}

\end{proposition}

Recall that a point in a building is singular if no neighborhood of the point is contained in a single apartment.

\section{An example with BPS states}

  In this section we carry out \textit{Step 1} from the BNR example for the spectral cover considered in Section 6 of  \cite{GMN-Snakes}. Let $x$ be a coordinate on $X=\mathbb A^1_{\mathbb C}$ and let $p$ be the coordinate along the fiber of the cotangent bundle $T_X^{\vee}\simeq \mathbb A^2_{\mathbb C}$. The spectral cover $\Lambda$ is given by the equation
  \[ p^3+4 x p +1 = 0 \:\]
  and thus $\Sigma \simeq (\mathbb A^1_\mathbb{C})^*$
  As in Section 5, we quantize the spectral curve to obtain a family of differential operators 
  \[ H_t = \frac{1}{t^3}\frac{d^3}{dx^3}-\frac{4}{t} x \frac{d}{dx}+ 1  \:.\]
  This, in turn, can be recast into a system of first order ODEs. \\
  The spectral network for $\Lambda \rightarrow X$ was constructed in \cite{GMN-Snakes}, see Figures \ref{SpecNetGMN}, \ref{SpecNetZoomedGMN}. The MARs were constructed as in the BNR example (see Figures \ref{MAR1GMN}, \ref{MAR2GMN}, \ref{MAR3GMN}, \ref{MAR4GMN} for some examples). 
  	\begin{figure}[h]
	\centering
	\includegraphics[width=1\textwidth]{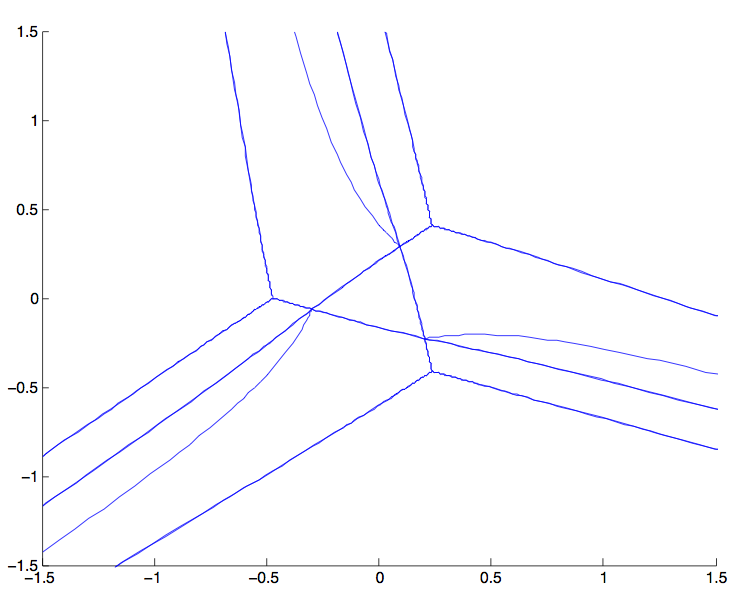}
	\caption{Spectral Network for $\theta=\pi/3$}
	\label{SpecNetGMN}
	\end{figure} 
	
	 \begin{figure}[h]
	\centering
	\includegraphics[width=1\textwidth]{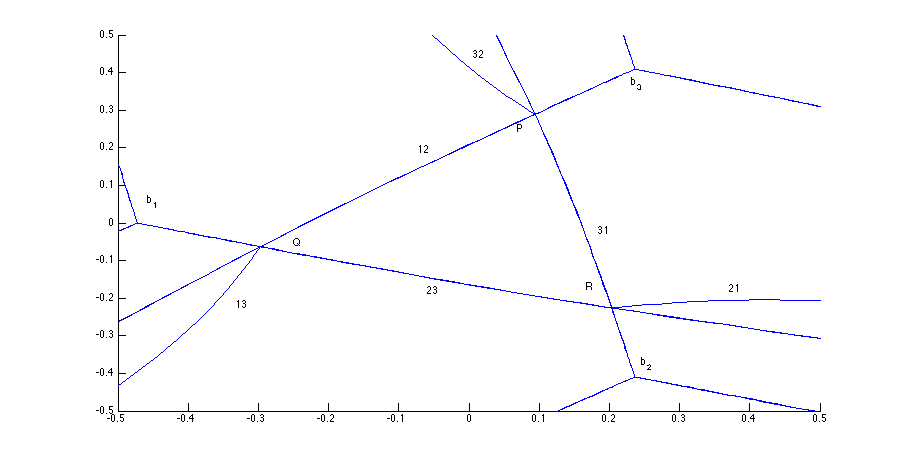}
	\caption{Spectral Network for $\theta=\pi/3$}
	\label{SpecNetZoomedGMN}
	\end{figure}  
	
	\begin{figure}[h]
	\centering
	\includegraphics[width=1\textwidth]{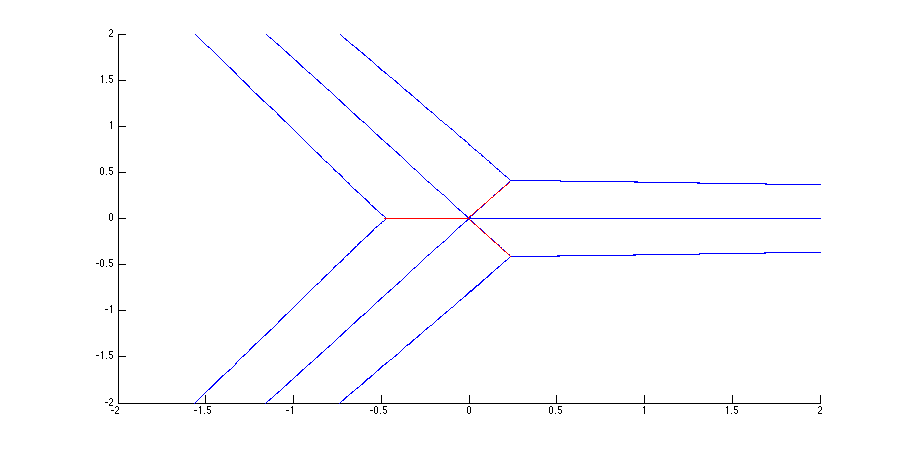}
	\caption{Spectral Network and BPS state for $\theta=\pi/2$}
	\label{BPS}
	\end{figure}  
	
	\begin{figure}[h]
	\centering
	\includegraphics[width=.8\textwidth]{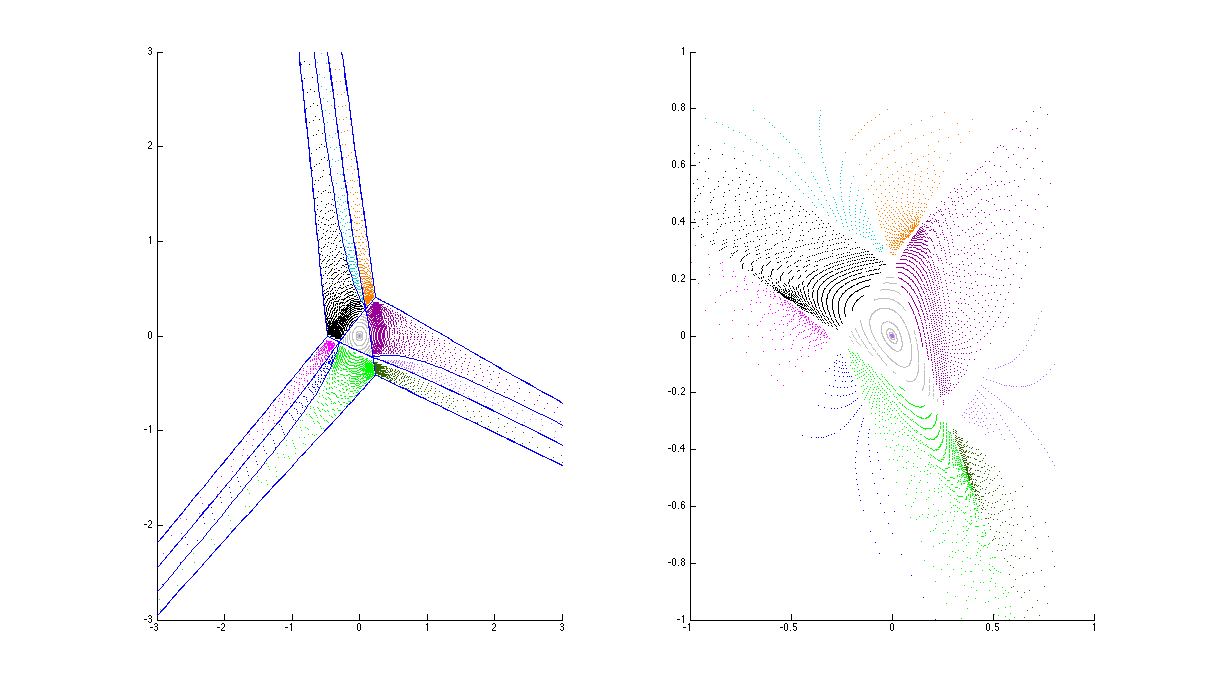}
	\caption{MAR1}
	\label{MAR1GMN}
	\end{figure}  
	
        \begin{figure}[h]
	\centering
	\includegraphics[width=.8\textwidth]{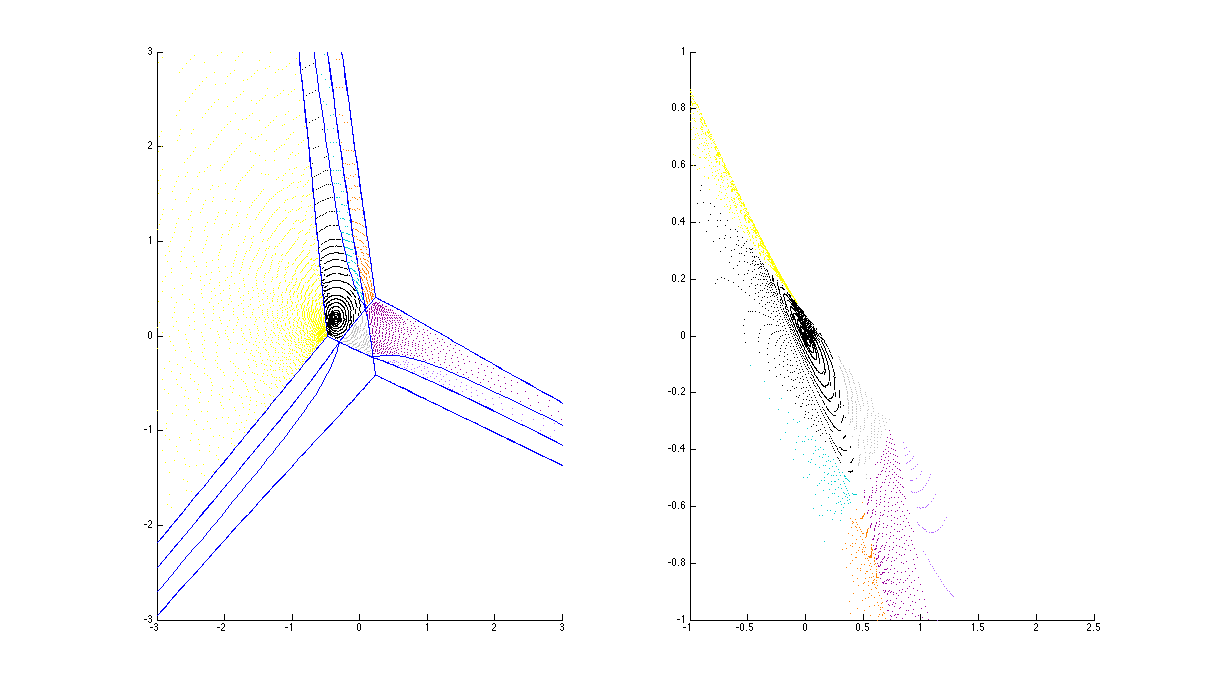}
	\caption{MAR2}
	\label{MAR2GMN}
	\end{figure}  
	
			\begin{figure}[h]
	\centering
	\includegraphics[width=.8\textwidth]{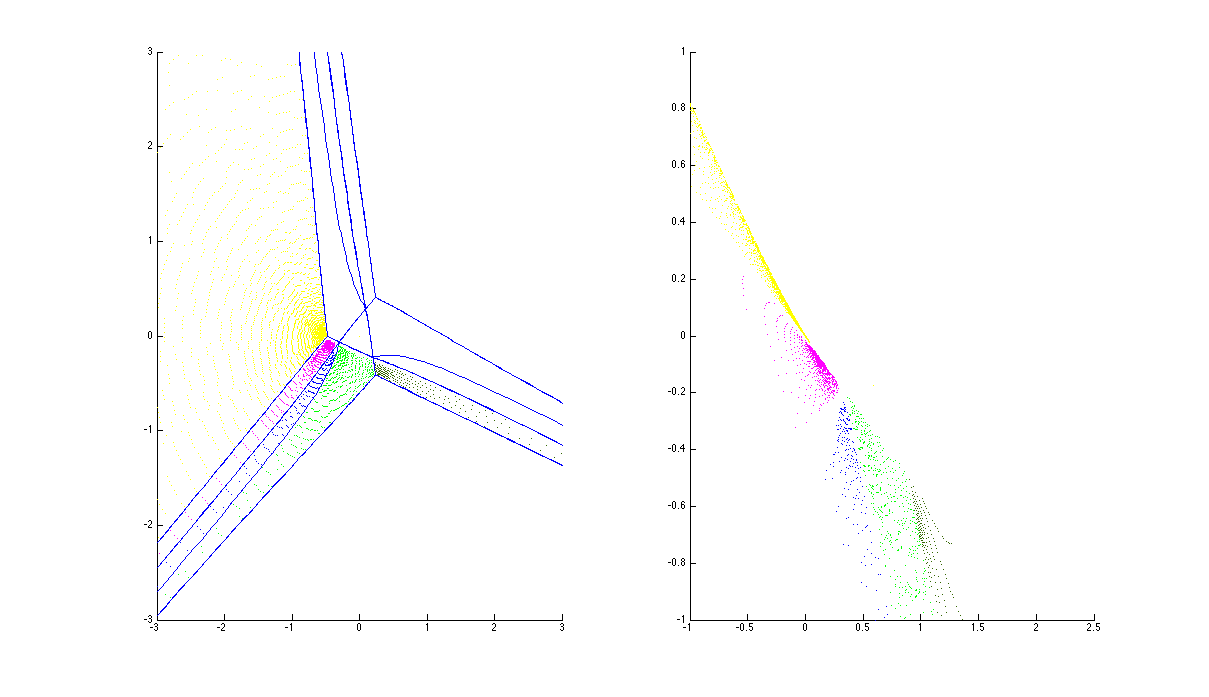}
	\caption{MAR3}
	\label{MAR3GMN}
	\end{figure}  

			\begin{figure}[h]
	\centering
	\includegraphics[width=.8\textwidth]{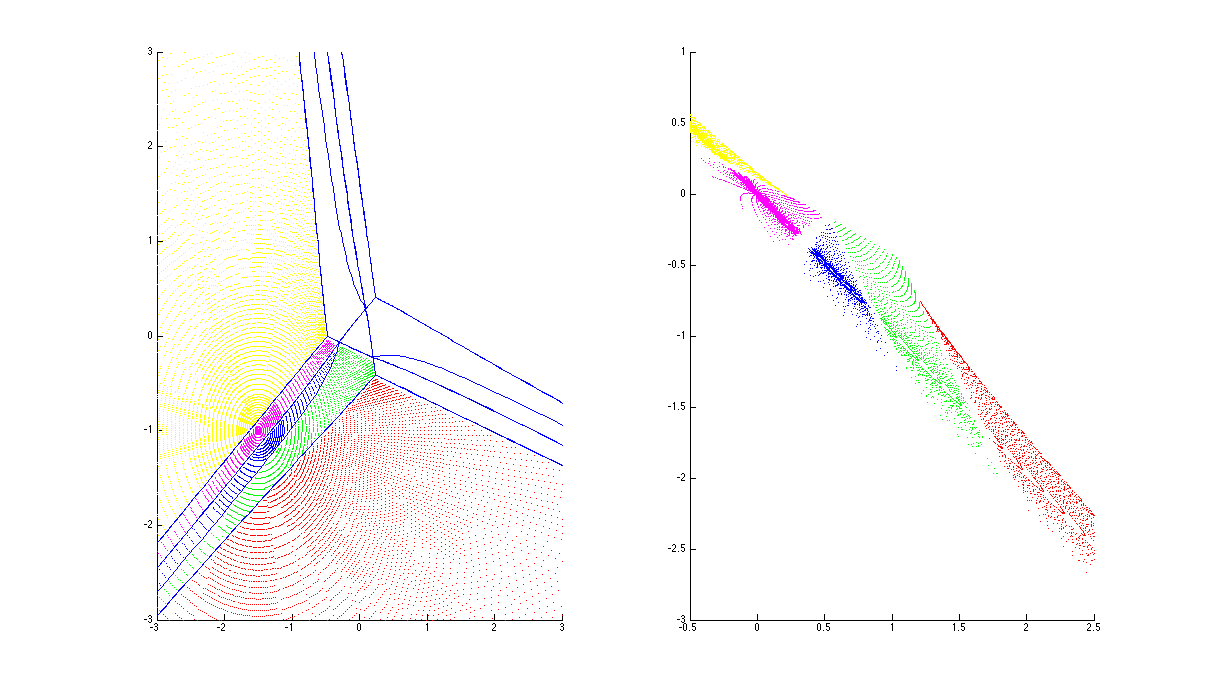}
	\caption{MAR4}
	\label{MAR4GMN}
	\end{figure}

  In this example, there is a closed BPS state $\overline{\gamma}$  (also described in described in \cite{GMN-Snakes}): the red \emph{BPS web} in Figure \ref{BPS} can be lifted to a closed cycle $[\overline{\gamma}]Ê\in H_1(\Lambda,\mathbb Z)$. We can see $\overline{\gamma}$ in the building as follows:

  \begin{definition}
	A simplex $S$ of maximal dimension in $\Bu$ whose codimension one facets are reflection hyperplanes is called a \textnormal{BPS simplex} if every vertex of $S$ is contained in the intersection of two singular walls.
  \end{definition}
  We can see that the building $\Bu$ contains a BPS simplex $S$, namely the grey triangle in Figure \ref{MAR1GMN}. It is interesting to note that $\re \int_{\overline{\gamma}} \lambda$ determines $S$ as follows: 
  From the definition of spectral networks, it follows that
  \[ \int_P^Q \lambda_{13}=\int_P^Q \lambda_{23}=\int_P^R \lambda_{23} =\int_P^R \lambda_{21}=l\:.\]
  Similarly,
  \[ \int_Q^P \lambda_{13}=\int_Q^R\lambda_{13}=\int_Q^R\lambda_{12}=l  \:.\]
  Using coordinates $(x,y)=(\int \lambda_1,\int \lambda_2)$, we obtain the following relations on the side lengths of $S$
  \[ \Delta^{PQ}_x = \Delta^{PQ}_y   \quad \quad \Delta^{QR}_x=\Delta^{QR}_y+l\quad\quad \Delta^{PR}_x=\Delta^{PR}_y+l\:.\]
  Now, $S$ is contained in a single apartment and is bounded by reflection hyperplanes.  

In fact, from these relations it follows that the volume $V$ of $S$ is equal to $l^2$. On the other hand side:

  \begin{equation}\label{BPScharge}
   l =\int_P^Q \lambda_{23}= \left(\int_P^Q+\int_Q^{b_1} +\int_{R}^Q\right) \lambda_{23}+\left(\int_R^{b_2} + \int_R^P \right)\lambda_{13} \:,\end{equation}
  since
  \begin{equation}\label{BPSref}
  \int_P^Q \lambda_{23}= \int_P^Q\lambda_2 + \int_Q^P\lambda_3 =\int_P^Q\lambda_2+\left(\int_Q^R +\int_R^P\right) \lambda_3=\int_P^Q\lambda_2+\int_Q^R\lambda_3+\int_R^P\lambda_1\:.
  \end{equation}
  Adding up all terms in (\ref{BPScharge}) and using (\ref{BPSref}) we get
  \[l=\re \int_{\overline{\gamma}} \lambda \:.\]\\

Thus, we obtain 
\[ V=\vert \re Z(\overline{\gamma}) \vert^2 \:. \]

  Note that as we vary the angle in $\mathcal W_\theta$, the volume of the triangle changes. Precisely at the angle for which there exists the closed BPS state $\overline{\gamma}$ (see Figure \ref{BPS}), the volume of the triangle vanishes. This situation is analogous to the case of trees and should lead to a notion of flips for buildings.
  
  We conjecture that the relation between stability conditions and buildings should always work similarly. In the case that $h(X) \subset \Bu$ contains finitely many BPS simplices, we propose a construction for the heart of the bounded t-structure in the $SL(3,\mathbb C)$ case as a quiver $Q$:
  \begin{itemize}
  	\item associate a vertex to every BPS simplex $S$
	\item associate an arrow to every edge in $S\cap S'$, where the direction of the arrow is determined by the orientation of $X$
  \end{itemize}
  By the Ginzburg construction, this determines a Calabi-Yau triangulated Category $\mathcal C(\Bu)$.\\
  We expect that one can also construct a heart of a bounded t-structure  from the building $\Bu$ if there are infinitely many BPS simplices.
  \begin{conjecture}
  	$\mathcal C(\Bu)$ is independent of the element $\left(\varphi_2, \varphi_3\right) \in H^0(X,\omega_X^2)\oplus H^0(X, \omega_X^3)=B$. \\
	Furthermore, a connected component of the moduli space of stability conditions can be identified with the Hitchin base $B$.
  \end{conjecture}

\section{Outlook}

Over the last several years, a variety of new categorical structures
have been discovered by physicists.  Furthermore, it has become
transparently evident that the higher categorical language is
beautifully suited to describing cornerstone concepts in modern
theoretical physics.

 In a striking series of papers
Gaiotto-Moore-Neitzke \cite{GMN-WKB, GMN-Spectral-Networks, GMN-Snakes} and Bridgeland-Smith \cite{Bridgeland-Smith, Smith} have established a
connection between Teichm\"uller theory and the theory of stability
conditions on triangulated categories.  An analogy between the
Teichm\"uller geodesic flow and the wall crossing on the space of
stability conditions had been noticed previously in the works of
Kontsevich and Soibelman \cite{Kontsevich-Soibelman-WCS, Kontsevich-Soibelman-summary, Kontsevich-Soibelman-Coh-Hall}.

These discoveries were taken further by Dimitrov, Haiden, Katzarkov and Konstevich
in \cite{DHKK}. They have defined  and studied
entropy in the context of triangulated and $A_\infty$-categories.
 Dynamical entropy typically arises as a
measure of the complexity of a dynamical system.  This notion exists
in a variety of flavors, e.g. the Kolmogorov-Sinai measure-theoretic
entropy, the topological entropy of Thurston and Gromov and algebraic
entropy.

The following categorical versions  were proven in \cite{DHKK}.

\begin{theorem} In the saturated case, the entropy of an
endofunctor may be computed as a limit of Poincar\'e polynomials of
Ext-groups.
\end{theorem}

\noindent
This result is connected to classical dynamical systems:

\begin{theorem}
  In the saturated case (under a certain generic technical condition),
  there is a lower bound on the entropy given by the logarithm of the
  spectral radius of the induced action on Hochschild homology.
\end{theorem}

\noindent
The following basic correspondences with classical theory of dynamical systems 
have emerged:

\begin{enumerate}
\item  geodesics $\leftrightarrow$ stable objects.

\item  compactifications of  Teichm\"uller spaces
 $\leftrightarrow$ stability
conditions.

\item  classical entropy of pseudo-Anosov transformations
  $\leftrightarrow$ 
  categorical entropy.

\item categories $\leftrightarrow$ differential equations.
\end{enumerate}
\noindent

We record our findings in the following table:

{\small
  \begin{tabular}{| c | c | c | c | c | }
  \hline
   \begin{tabular}{l}  \\
                      Category \\ \\ 
\end{tabular}    &  Stable objects  &    Stab. cond.            
& \begin{tabular}{l} Density  \\
                      of phases \end{tabular}  &   Diff. eq. \\ \hline
 $A_n$        &   \begin{tabular}{l}  \\
                       \includegraphics[width=3cm, height=1cm]{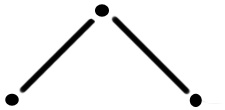}  \\ \\ \end{tabular}                   &  $e^{P(z)} (dz)^2$   &  NO         & $\left ( \left (\frac{d}{dz} \right )^2 +
                                                                                                           e^{P(z)} \right )f=0$\\ \hline
   $\hat{ A_n }$        &   \begin{tabular}{l}  \\
                    \includegraphics[width=3cm, height=1cm]{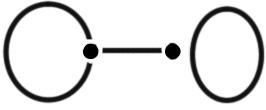} \\  \end{tabular}               &   \begin{tabular}{l}  \\
                     $q(z) (dz)^2$  \\ \\ \end{tabular}         &  YES        &   Schr\"odinger   eq. \\ \hline
   \begin{math}\tot(\xymatrix{A_{r-1} \ar[d] \\ C} )\end{math}       &    \begin{tabular}{l}  \\
                      Spectral networks  \\ \\ \end{tabular}                               & $\oplus_{k=2}^{r} H^{0}(X, \omega_{X}^{\otimes k}) $ &  YES        &     Lax pair        \\ \hline
  \end{tabular}
}

\vspace{5.mm}

\noindent

In this paper we  have initiated a study of these  correpondences in the case recorded in the third row of the table above -- this corresponds to the where the category  $\cC$ arises as the global sections of a  constructible sheaf of two dimensional CY-categories over a Lagrangian skeleton - see \cite{DHKK},

We have established initial evidence for the conjectures described below. In this paper, for simplicity, we have considered only the case of 2 dimensional buildings connected with $\SL_3 \cc$-Higgs bundles. However, we expect these conjectures to be true for buildings associated to arbitrary semisimple Lie groups.

\begin{conjecture}
  
Let $G$ be a semisimple complex Lie group, let $X$ be a Riemann surface, and let $\phi \in \oplus_{i = 1}^{r} H^{0}(X, \omega_{X}^{\otimes d_i})$ be a point in the corresponding Hitchin base. The singularities of the harmonic map to the universal building associated to $\phi$ determine a category 3d CY-category
$\cC$ .
\end{conjecture}

The study of the singularities in \cite{Gromov-Schoen} and
in \cite{Daskalopoulos-Mese}
suggest that the combinatorics of the buiding and
singularities of the harmonic map strongly affect the category.
We also establish a possibility of studying stability conditions on $\cC$:

\begin{conjecture}
  
There exists a component of the moduli space of stability conditions of $\cC$ 
which contains the Hitchin base $\oplus_{i=1}^{r} H^{0}(X, \omega_{X}^{\otimes d_i})$

\end{conjecture}

In general we expect this componnent to contain the base of the Hitchin system.
The connection with differential equations, initiated in \cite{Witten}, goes even further.

\begin{conjecture}
  
The behavior of the Stokes factors and resurgence phenomenon for thr Lax pair of the Hitchin system determine the wall-crossing phenomena on the 
component of the moduli space of stability conditions of $\cC$ 
containing the Hitchin base $\oplus_{i=1}^{r} H^{0}(X, \omega_{X}^{\otimes d_i})$. 

\end{conjecture}

This is just the tip of the iceberg. We expect deeper structures to come out of these considerations. In particular, we expect that we can embed our spectral curve in germs of differentials and consider, in such a way, a map of maximal dimension. This should give a fibration of tori, and Strebel type of behavior in higher dimensions.

We conclude by listing two natural questions that arise from this work. The following question was suggested to us by A. Goncharov:

\begin{question}
Is there a natural connection between ``the harmonic maps to buildings approach'' to constructing categories, on the one hand, and the approach of Goncharov and Shen \cite{Goncharov-Shen}, which is based on ``surface affine Grassmannians'' and ideal $\SL_m$ webs, on the other?
\end{question}

In their recent work \cite{Dyckerhoff-Kapranov-Segal}, Dyckerhoff and Kapranov have introduced the notion of a $2$-Segal space, which axiomatizes the properties of the Waldhausen $S$-construction, and is the structure underlying Hall algebras (``algebras of BPS states''). These $2$-Segal spaces are ubiquitous in mathematics, and have been used by Dyckerhoff and Kapranov to construct the Fukaya category of a surface \cite{Dyckerhoff-Kapranov-Fukaya}, implementing an instance of Kontsevich's program of localizing Fukaya categories along Lagrangian skeleta. In this context, it is intriguing to note that one of the examples of a $2$-Segal space is a Bruhat-Tits building (\cite[Example 3.1.5]{Dyckerhoff-Kapranov-Segal}). This naturally leads to the following:

\begin{question}
Is there a $2$-Segal space naturally associated to the universal building, and how is it related to the Hall algebra of the 3d CY-category $\cC$, and to spectral networks?
\end{question}

\newpage

\listoffigures

\bibliographystyle{amsalpha}
\bibliography{wkb-paper.bbl}

\providecommand{\bysame}{\leavevmode\hbox to3em{\hrulefill}\thinspace}
\providecommand{\MR}{\relax\ifhmode\unskip\space\fi MR }
\providecommand{\MRhref}[2]{%
  \href{http://www.ams.org/mathscinet-getitem?mr=#1}{#2}
}
\providecommand{\href}[2]{#2}
\begin{thebibliography}{{Kom}13}

\bibitem[AB08]{Abramenko-Brown}
P.~Abramenko and K.~S. Brown, \emph{Buildings}, Graduate Texts in Mathematics,
  vol. 248, Springer, New York, 2008, Theory and applications.

\bibitem[Aok05]{Aoki-fresh}
T.~Aoki, \emph{A fresh glimpse into the {S}tokes geometry of the
  {B}erk-{N}evins-{R}oberts equation through a singular coordinate
  transformation}, Technical report, Kyoto University, Research Institute for
  Mathematical Sciences, 2005.

\bibitem[BB{\'S}82]{Birula-Swiecicka-quotients}
A.~Bia{\l}ynicki-Birula and J.~{\'S}wi{\polhk{e}}cicka, \emph{Complete
  quotients by algebraic torus actions}, Group actions and vector fields
  ({V}ancouver, {B}.{C}., 1981), Lecture Notes in Math., vol. 956, Springer,
  Berlin, 1982, pp.~10--22.

\bibitem[BB{\'S}87]{Birula-Swiecicka-moment}
\bysame, \emph{Generalized moment functions and orbit spaces}, Amer. J. Math.
  \textbf{109} (1987), no.~2, 229--238.

\bibitem[BNR82]{BNR}
H.~L. Berk, W.~Nevins, and K.~V. Roberts, \emph{New {S}tokes' line in {WKB}
  theory}, J. Math. Phys. \textbf{23} (1982), no.~6.

\bibitem[BNR89]{BNR-spectral}
A.~Beauville, M.~S. Narasimhan, and S.~Ramanan, \emph{Spectral curves and the
  generalised theta divisor}, J. Reine Angew. Math. \textbf{398} (1989),
  169--179.

\bibitem[BS13]{Bridgeland-Smith}
T.~{Bridgeland} and I.~{Smith}, \emph{{Quadratic differentials as stability
  conditions}}, ArXiv e-prints (2013), 1302.7030 [math.AG].

\bibitem[BSS10]{Bennett-Schwer}
C.~D. {Bennett}, P.~N. {Schwer}, and K.~{Struyve}, \emph{{On axiomatic
  definitions of non-discrete affine buildings}}, ArXiv e-prints (2010),
  1003.3050 [math.MG].

\bibitem[Com77]{Comfort}
W.~W. Comfort, \emph{Ultrafilters: some old and some new results}, Bull. Amer.
  Math. Soc. \textbf{83} (1977), no.~4, 417--455.

\bibitem[dCHM12]{Hausel-et-al-PW}
M.~A. de~Cataldo, T.~Hausel, and L.~Migliorini, \emph{Topology of {H}itchin
  systems and {H}odge theory of character varieties: the case {$A\sb 1$}}, Ann.
  of Math. (2) \textbf{175} (2012), no.~3, 1329--1407.

\bibitem[DDW00]{Daskalopoulos-Dostolgou-Wentworth}
G.~Daskalopoulos, S.~Dostoglou, and R.~Wentworth, \emph{On the
  {M}organ-{S}halen compactification of the {${\rm SL}(2,{\bf C})$} character
  varieties of surface groups}, Duke Math. J. \textbf{101} (2000), no.~2,
  189--207.

\bibitem[DHKK13]{DHKK}
G.~{Dimitrov}, F.~{Haiden}, L.~{Katzarkov}, and M.~{Kontsevich},
  \emph{{Dynamical systems and categories}}, ArXiv e-prints (2013), 1307.8418
  [math.CT].

\bibitem[DK12]{Dyckerhoff-Kapranov-Segal}
T.~{Dyckerhoff} and M.~{Kapranov}, \emph{{Higher Segal spaces I}}, ArXiv
  e-prints (2012), 1212.3563 [math.AT].

\bibitem[DK13]{Dyckerhoff-Kapranov-Fukaya}
\bysame, \emph{{Triangulated surfaces in triangulated categories}}, ArXiv
  e-prints (2013), 1306.2545 [math.AG].

\bibitem[DM]{Daskalopoulos-Mese}
G.~{Daskalopoulos} and C.~{Mese}, \emph{On the singular set of harmonic maps
  into {D}{M}-complexes}, Mem. Amer. Math. Soc., To appear.

\bibitem[Don95]{Donagi-spectral}
R.~Donagi, \emph{Spectral covers}, Current topics in complex algebraic geometry
  ({B}erkeley, {CA}, 1992/93), Math. Sci. Res. Inst. Publ., vol.~28, Cambridge
  Univ. Press, Cambridge, 1995, pp.~65--86.

\bibitem[DW07]{Daskalopoulos-Wentworth}
G.~{Daskalopoulos} and R.~{Wentworth}, \emph{Harmonic maps and {T}eichm\"uller
  theory}, Handbook of {T}eichm\"uller theory. {V}ol. {I}, IRMA Lect. Math.
  Theor. Phys., vol.~11, Eur. Math. Soc., Z\"urich, 2007, pp.~33--109.

\bibitem[EKPR12]{EKPR}
P.~Eyssidieux, L.~Katzarkov, T.~Pantev, and M.~Ramachandran, \emph{Linear
  {S}hafarevich conjecture}, Ann. of Math. (2) \textbf{176} (2012), no.~3,
  1545--1581.

\bibitem[{Eve}12]{Everitt}
B.~{Everitt}, \emph{{A (very short) introduction to buildings}}, ArXiv e-prints
  (2012), 1207.2266 [math.GR].

\bibitem[FW01]{Farb-Wolf}
B.~Farb and M.~Wolf, \emph{Harmonic splittings of surfaces}, Topology
  \textbf{40} (2001), no.~6, 1395--1414.

\bibitem[GMN12]{GMN-Snakes}
D.~{Gaiotto}, G.~W. {Moore}, and A.~{Neitzke}, \emph{{Spectral Networks and
  Snakes}}, ArXiv e-prints (2012), 1209.0866 [hep-th].

\bibitem[GMN13a]{GMN-Spectral-Networks}
D.~Gaiotto, G.~W. Moore, and A.~Neitzke, \emph{Spectral {N}etworks}, Ann. Henri
  Poincar\'e \textbf{14} (2013), no.~7, 1643--1731.

\bibitem[GMN13b]{GMN-WKB}
\bysame, \emph{Wall-crossing, {H}itchin systems, and the {WKB} approximation},
  Adv. Math. \textbf{234} (2013), 239--403.

\bibitem[Gro81]{Gromov-metric-structures}
M.~Gromov, \emph{Structures m\'etriques pour les vari\'et\'es riemanniennes},
  Textes Math\'ematiques [Mathematical Texts], vol.~1, CEDIC, Paris, 1981,
  Edited by J. Lafontaine and P. Pansu.

\bibitem[Gro07]{Gromov-metric-structures-revised}
\bysame, \emph{Metric structures for {R}iemannian and non-{R}iemannian spaces},
  english ed., Modern Birkh\"auser Classics, Birkh\"auser Boston Inc., Boston,
  MA, 2007, Based on the 1981 French original, With appendices by M. Katz, P.
  Pansu and S. Semmes, Translated from the French by Sean Michael Bates.

\bibitem[GS92]{Gromov-Schoen}
M.~Gromov and R.~Schoen, \emph{Harmonic maps into singular spaces and
  {$p$}-adic superrigidity for lattices in groups of rank one}, Inst. Hautes
  \'Etudes Sci. Publ. Math. (1992), no.~76, 165--246.

\bibitem[GS13]{Goncharov-Shen}
A.~B. {Goncharov} and L.~{Shen}, \emph{{Geometry of canonical bases and mirror
  symmetry}}, ArXiv e-prints (2013), 1309.5922 [math.RT].

\bibitem[Hau98]{Hausel-compactification}
T.~Hausel, \emph{Compactification of moduli of {H}iggs bundles}, J. Reine
  Angew. Math. \textbf{503} (1998), 169--192.

\bibitem[Hel01]{Helgason}
S.~Helgason, \emph{Differential geometry, {L}ie groups, and symmetric spaces},
  Graduate Studies in Mathematics, vol.~34, American Mathematical Society,
  Providence, RI, 2001, Corrected reprint of the 1978 original.

\bibitem[Hon]{Honda}
N.~Honda, \emph{Degenerate {S}tokes geometry and some geometric structure
  underlying a virtual turning point}, Algebraic analysis and the exact {WKB}
  analysis for systems of differential equations, RIMS K\^oky\^uroku Bessatsu,
  B5, Res. Inst. Math. Sci. (RIMS), Kyoto, pp.~15--49.

\bibitem[{Kat}94]{Ludmil-factorization}
L.~{Katzarkov}, \emph{{Factorization theorems for the representations of the
  fundamental groups of quasiprojective varieties and some applications}},
  eprint arXiv:alg-geom/9402012, 1994.

\bibitem[Kat95]{Ludmil-thesis}
L.~Katzarkov, \emph{Factorization theorems for the representations of the
  fundamental groups of quasiprojective varieties and some applications},
  ProQuest LLC, Ann Arbor, MI, 1995, Thesis (Ph.D.)--University of
  Pennsylvania.

\bibitem[KL97]{Kleiner-Leeb}
B.~Kleiner and B.~Leeb, \emph{Rigidity of quasi-isometries for symmetric spaces
  and {E}uclidean buildings}, Inst. Hautes \'Etudes Sci. Publ. Math. (1997),
  no.~86, 115--197 (1998).

\bibitem[{Kom}13]{Komyo}
A.~{Komyo}, \emph{{On compactifications of character varieties of
  \$n\$-punctured projective line}}, ArXiv e-prints (2013), 1307.7880
  [math.AG].

\bibitem[KR98]{Ludmil-Ramachandran}
L.~Katzarkov and M.~Ramachandran, \emph{On the universal coverings of algebraic
  surfaces}, Ann. Sci. \'Ecole Norm. Sup. (4) \textbf{31} (1998), no.~4,
  525--535.

\bibitem[KS93]{Korevaar-Schoen}
N.~J. Korevaar and R.~M. Schoen, \emph{Sobolev spaces and harmonic maps for
  metric space targets}, Comm. Anal. Geom. \textbf{1} (1993), no.~3-4,
  561--659.

\bibitem[KS10]{Kontsevich-Soibelman-summary}
M.~Kontsevich and Y.~Soibelman, \emph{Motivic {D}onaldson-{T}homas invariants:
  summary of results}, Mirror symmetry and tropical geometry, Contemp. Math.,
  vol. 527, Amer. Math. Soc., Providence, RI, 2010, pp.~55--89.

\bibitem[KS11]{Kontsevich-Soibelman-Coh-Hall}
\bysame, \emph{Cohomological {H}all algebra, exponential {H}odge structures and
  motivic {D}onaldson-{T}homas invariants}, Commun. Number Theory Phys.
  \textbf{5} (2011), no.~2, 231--352.

\bibitem[KS13]{Kontsevich-Soibelman-WCS}
M.~{Kontsevich} and Y.~{Soibelman}, \emph{{Wall-crossing structures in
  Donaldson-Thomas invariants, integrable systems and Mirror Symmetry}}, ArXiv
  e-prints (2013), 1303.3253 [math.AG].

\bibitem[Lei13]{Leinster}
T.~Leinster, \emph{Codensity and the ultrafilter monad}, Theory Appl. Categ.
  \textbf{28} (2013), No. 13, 332--370.

\bibitem[LSS13]{LoraySaitoSimpson}
F.~{Loray}, M.-H. {Saito}, and C.~T. {Simpson}, \emph{{Foliations on the moduli
  space of rank two connections on the projective line minus four points}},
  Geometric and differential Galois theory (Luminy 2010), D. Bertrand, Ph.
  Boalch, J-M. Couveignes, P. D\`ebes, eds., S\'eminaires et Congr\`es,
  vol.~78, S.M.F., Paris, 2013, pp.~115--168.

\bibitem[Mau09]{Maubon}
J.~Maubon, \emph{Symmetric spaces of the non-compact type: differential
  geometry}, G\'eom\'etries \`a courbure n\'egative ou nulle, groupes discrets
  et rigidit\'es, S\'emin. Congr., vol.~18, Soc. Math. France, Paris, 2009,
  pp.~1--38.

\bibitem[Par00a]{Parreau-thesis}
A.~Parreau, \emph{D\'eg\'en\'erescences de sous-groupes discrets de groupes de
  lie semi-simples et actions de groupes sur des immeubles affines}, 2000,
  Thesis, Universit\'e Paris-Sud, Orsay.

\bibitem[Par00b]{Parreau-norms}
\bysame, \emph{Immeubles affines: construction par les normes et \'etude des
  isom\'etries}, Crystallographic groups and their generalizations ({K}ortrijk,
  1999), Contemp. Math., vol. 262, Amer. Math. Soc., Providence, RI, 2000,
  pp.~263--302.

\bibitem[Par12]{Parreau-compactification}
\bysame, \emph{Compactification d'espaces de repr\'esentations de groupes de
  type fini}, Math. Z. \textbf{272} (2012), no.~1-2, 51--86.

\bibitem[Ron09]{Ronan}
M.~Ronan, \emph{Lectures on buildings}, University of Chicago Press, Chicago,
  IL, 2009, Updated and revised.

\bibitem[Rou]{Rousseau}
G.~Rousseau, \emph{Euclidean buildings}, G\'eom\'etries \`a courbure n\'egative
  ou nulle, groupes discrets et rigidit\'es, S\'emin. Congr., vol.~18,
  pp.~77--116.

\bibitem[Sim91]{Carlos-book}
C.~Simpson, \emph{Asymptotic behavior of monodromy}, Lecture Notes in
  Mathematics, vol. 1502, Springer-Verlag, Berlin, 1991, Singularly perturbed
  differential equations on a Riemann surface.

\bibitem[Sim92]{Carlos-Higgs-Local}
C.~T. Simpson, \emph{Higgs bundles and local systems}, Inst. Hautes \'Etudes
  Sci. Publ. Math. (1992), no.~75, 5--95.

\bibitem[Sim04]{Carlos-infinity}
C.~Simpson, \emph{Asymptotics for general connections at infinity},
  Ast\'erisque (2004), no.~297, 189--231, Analyse complexe, syst{\`e}mes
  dynamiques, sommabilit{\'e} des s{\'e}ries divergentes et th{\'e}ories
  galoisiennes. II.

\bibitem[{Smi}13]{Smith}
I.~{Smith}, \emph{{Quiver algebras as Fukaya categories}}, ArXiv e-prints
  (2013), 1309.0452 [math.SG].

\bibitem[Ste06]{Stepanov-dual-complex}
D.~A. Stepanov, \emph{A remark on the dual complex of a resolution of
  singularities}, Uspekhi Mat. Nauk \textbf{61} (2006), no.~1(367), 185--186.

\bibitem[Ste08]{Stepanov-resolution}
\bysame, \emph{A note on resolution of rational and hypersurface
  singularities}, Proc. Amer. Math. Soc. \textbf{136} (2008), no.~8.

\bibitem[Sun03]{Sun}
X.~Sun, \emph{Regularity of harmonic maps to trees}, Amer. J. Math.
  \textbf{125} (2003), no.~4, 737--771.

\bibitem[Thu07]{Thuillier-toroidal}
A.~Thuillier, \emph{G\'eom\'etrie toro\"\i dale et g\'eom\'etrie analytique non
  archim\'edienne. {A}pplication au type d'homotopie de certains sch\'emas
  formels}, Manuscripta Math. \textbf{123} (2007), no.~4, 381--451.

\bibitem[Was85]{Wasow}
W.~Wasow, \emph{Linear turning point theory}, Applied Mathematical Sciences,
  vol.~54, Springer-Verlag, New York, 1985.

\bibitem[Wit11]{Witten}
E.~Witten, \emph{A new look at the path integral of quantum mechanics}, Surveys
  in differential geometry. {V}olume {XV}. {P}erspectives in mathematics and
  physics, Surv. Differ. Geom., vol.~15, Int. Press, Somerville, MA, 2011,
  pp.~345--419.

\end{thebibliography}

\end{document}